\documentclass[12pt]{ucsd}

\usepackage{amssymb}
\usepackage{amsmath}
\usepackage{latexsym}
\usepackage{verbatim}

\newtheorem{theorem}{Theorem}
\newtheorem{lemma}{Lemma}

\newtheorem{defn}{Definition}

\def\face#1{\hbox{$\bigcirc$\llap{\lower 2pt\hbox{\"{}\rlap{\lower 3pt
\hbox{\kern -2.5pt$#1{}$}}}\kern 2.5pt}}}
\def\happy{\face\breve}

\def\tongue{\hbox{$\bigcirc$\llap{\lower 2pt\hbox{\"{}\rlap{\lower 1pt
\hbox{\kern -4.5pt${}^{{}_\sigma}$}}}\kern 2.5pt}}}

\newenvironment{proof}{\par{\em Proof\/.}}{$\happy$\medskip\par}
\newenvironment{ex}{\noindent{\bf Example}:}{\smallskip}

\title{How to encode a tree}
\author{Sally Picciotto}

\begin{document}

\degreeyear{1999}
\degree{Doctor of Philosophy}
\chair{Professor Peter Doyle}
\othermembers{Professor Adriano Garsia\\
Professor Mark Haiman\\
Professor Audrey Terras\\
Professor Mihir Bellare\\
Professor Ramamohan Paturi}
\numberofmembers{6}
\prevdegrees{B.S. (Yale University) 1993\\
M.A. (University of California, San Diego), 1995}
\field{Mathematics}
\campus{San Diego}

\begin{frontmatter}
\maketitle
\copyrightpage
\approvalpage
\begin{dedication}
\setcounter{page}{4}% So page numbering is right even if no sign page
\null\vfil
{\large
\begin{center}
To\\
Ryan Garibaldi\\
Jean Isaacs\\
and Dalit Baum\\
\vspace{8pt}
who have witnessed (and greatly assisted in) \\
my growth over the last few years.
\end{center}
}
\vspace{2in}
\hfill\parbox[t]{3in}{
\begin{flushright}
\emph{``I'm afraid you misunderstood...  I said I'd like a mango.''--G. Larsen}
\end{flushright}
}
\vfil\null
\end{dedication}

\tableofcontents
\begin{acknowledgements}
It is probably impossible to thank everyone to whom I am grateful, but I'm
going to give it my best shot.  First of all, I thank Peter Doyle for his
ability to spark my interest in the questions I explored in my research and
also, just as importantly, for his patience.  His enthusiasm is contagious and 
his nontraditional teaching methods inspire me.  With another advisor, I might
never have completed this degree.

I am indebted to the other faculty members on my committee as well.
Professor Garsia pointed me to important papers and offered valuable 
insights at several points in my research, and his courses have helped me to
understand the beauty of both bijective and ``manipulatoric'' proofs. 
His Involution Principle was obviously an important foundation for my 
research, without which my own results would not have been possible.
Professor Haiman's ``applied algebra'' qualifying class initially
introduced me to combinatorics, and he suggested 
ways that my research might relate to questions he and other
combinatorists investigate.   Professor Terras has been an unflagging 
source of encouragement, inspiration, and comfort.  Professor Bellare was
my supervisor for one of my favorite TAing assignments at UCSD and also
suggested future directions for my career. 

In addition to my committee, courses and papers from (and discussions with)
Jeff Remmel, as well as
discussions with Glenn Tesler, Jeb Willenbring, and Mike Zabrocki,
aided in my understanding of the subject.  Mike Zabrocki 
actually read my first rough draft of the thesis 
and offered dozens of helpful suggestions.
Furthermore, the Math Department staff was of great assistance on numerous
occasions, so I thank Lois Stewart, Joe Keefe, Wilson Cheung,
Judy Gregg, Zee Collins, Lee Monta\~no,
and Mike Stegen, plus Nieves Rankin and Kathy Johns from CSE.  Professors
Harold Stark and Ruth Williams have both helped me to develop my teaching
abilities by visiting my classes and making suggestions.

Beyond mathematics, 
there are many people who have vastly improved my quality of life
during graduate school.  I wish to express my deep appreciation.

\end{acknowledgements}
\begin{vitapage}
\addcontentsline{toc}{chapter}{\protect\numberline{}Vita}
\begin{vita}
\item [January 12, 1972] Born, San Francisco, California
\item [1993] B.\ S., \emph{cum laude},  
Yale University
\item [1993--1998] Teaching assistant, Department of Mathematics, 
University of California San Diego
\item [1995] M.\ A., University of California San Diego
\item [1997] Lecturer, Department of Mathematics, 
University of California San Diego
\item [1998]  Choreographed ``Near The Sand Comets'' for ``New Works'' dance
concert, University of California San Diego
\item [1998--1999] Teaching assistant, Department of Computer Science and 
Engineering, University of California San Diego
\item [1999] Ph.\ D., University of California San Diego 
\end{vita}
\end{vitapage}

\begin{abstract}
We construct bijections giving three ``codes'' for trees.  These codes 
follow naturally from the Matrix Tree Theorem of Tutte and have many 
advantages over the one produced by Pr\"ufer in 1918.  
One algorithm gives explicitly a bijection that is
implicit in Orlin's manipulatorial proof of Cayley's formula (the formula
was actually found first by Borchardt).
Another is based on a proof of Knuth.  The third is an implementation of
Joyal's pseudo-bijective proof of the formula, and is equivalent to one
previously found by E{\u{g}}ecio{\u{g}}lu and Remmel.  In each case, we have at
least two algorithms, one of
which involves hands-on manipulations of the tree while the other
involves a combinatorial and linear algebraic manipulation of a matrix.
\end{abstract}

\end{frontmatter}

%Here is chapter 1!

\chapter{Introduction}

This dissertation is a contribution to the history of progressing 
from algebraic proofs to bijective proofs.  In particular, for theorems
involving graphs, there is a long history of proofs using matrices.
We
start with linear algebra, but automatically something is going
on beneath the surface that turns out to be a simple bijection.  

\section{Definitions}

\begin{defn}
A {\em directed graph} is a quadruple $G=(V,E,\alpha,\omega)$, 
where the elements of the set $V$ are
called {\em vertices} and the elements of the set $E$ are called {\em edges},
and $\alpha$ and $\omega$ are the boundary maps from $E$ to $V$. If $e\in E$,
then $\alpha(e)\in V$ is the {\em initial vertex} or {\em tail} of $e$ and
$\omega(e)\in V$ is the {\em terminal vertex} or {\em head} of $e$.
\end{defn}
Note that this definition allows for multiple edges with tail $v_1$ and
head $v_2$.
An edge in a directed graph can be represented by an arrow pointing from 
the initial vertex to the terminal vertex.

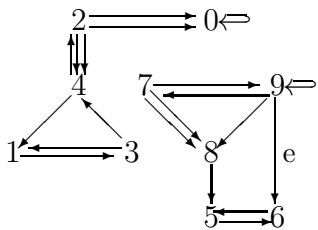
\begin{ex}

\begin{picture}(100,90)
\put(25,75){2}
\put(32,80){\vector(1,0){40}}
\put(32,76){\vector(1,0){40}}
\put(27,73){\vector(0,-1){15}}
\put(30,73){\vector(0,-1){15}}
\put(25,58){\vector(0,1){15}}
\put(25,50){4}
\put(25,50){\vector(-1,-1){20}}
\put(0,25){1}
\put(6,27){\vector(1,0){35}}
\put(45,25){3}
\put(44,34){\vector(-1,1){15}}
\put(44,30){\vector(-1,0){35}}
\put(75,0){5}
\put(81,2){\vector(1,0){20}}
\put(100,0){6}
\put(99,6){\vector(-1,0){20}}
\put(75,25){8}
\put(78,24){\vector(0,-1){15}}
\put(50,50){7}
\put(53,49){\vector(1,-1){19}}
\put(55,52){\vector(1,-1){19}}
\put(56,54){\vector(1,0){40}}
\put(100,50){9}
\put(105,50){$\hookleftarrow$}
\put(109,55.5){\line(1,0){7}}
\put(99,49){\vector(-1,-1){19}}
\put(102,49){\vector(0,-1){40}}
\put(105,25){\small{e}}
\put(100,50){\vector(-1,0){40}}
\put(75,75){0}
\put(80,75){$\hookleftarrow$}
\put(84,80.5){\line(1,0){7}}
\end{picture}

\noindent
Here the vertices are labelled with integers.
The edge labelled e satisfies $\alpha({\rm e})=9$ and $\omega({\rm e})=6$.
\end{ex}
An edge is said
to {\em point from} or {\em out of} 
its tail and {\em point to} or {\em into} its head.
%\begin{defn}
%A digraph is {\em connected} if one can get from any vertex to any
%other vertex by travelling along edges 
%(ignoring directions).
%\end{defn}
This dissertation deals with directed graphs
whose vertices are labelled $0, 1, \dots, n$.  Sometimes the edges have
weights associated to them.  Sometimes we refer to a directed graph as
simply a graph.
\begin{defn}
A function $W:E\to S$, where $S$ is any set, defines a 
{\em weight} for each edge.
\end{defn}
\begin{defn}
The {\em indegree} of a vertex in a directed graph is the number of 
edges of which the vertex is the head, and the {\em outdegree} is the
number of edges of which the vertex is the tail.
\end{defn}
\begin{defn}
A {\em path} in a directed graph is an alternating sequence of vertices 
and edges
$v_1, e_1, v_2, e_2,\dots,e_4, v_{r+1}$ where $v_i$ is the tail of the
edge $e_i$ and $v_{i+1}$ is the head of the edge $e_i$.
\end{defn}
\begin{defn}
A {\em cycle} in a directed graph is a closed path (a path where 
$v_1=v_{r+1}$).  A cycle with only one edge, $v\to v$, is called a {\em loop}. 
\end{defn}
\begin{defn}
The {\em complete digraph} (with a given number of vertices) is a directed
graph with exactly one edge $v_1\to v_2$ for each pair of vertices.  
\end{defn}
\begin{defn}
A {\em rooted tree} is a digraph with a unique path connecting each vertex
to the (unique) vertex with outdegree 0 called the {\em root}.
Any vertex whose indegree is $0$ is called a {\em leaf}.
\end{defn}
Unless otherwise noted, all trees are rooted at 0.  Any ``free tree'' (an
undirected connected graph with no cycles) can be transformed uniquely into 
a tree rooted at 0 by directing all edges toward 0.

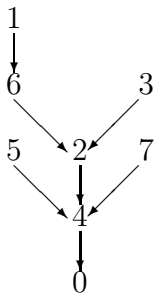
\begin{ex}

\begin{picture}(100,110)
\put(50,0){0}
\put(50,25){4}
\put(53,23){\vector(0,-1){15}}%  4->0
\put(25,50){5}
\put(28,48){\vector(1,-1){20}}%  5->4
\put(50,50){2}
\put(53,48){\vector(0,-1){15}}%  2->4
\put(75,50){7}
\put(75,48){\vector(-1,-1){19}}% 7->4
\put(25,75){6}
\put(75,75){3}
\put(28,73){\vector(1,-1){20}}%  6->2
\put(75,73){\vector(-1,-1){19}}% 3->2
\put(25,100){1}
\put(28,98){\vector(0,-1){15}}%  1->6
\end{picture}
\end{ex}

\noindent
In this tree, the leaves are 1, 3, 5, and 7.
\begin{defn}
The {\em weight} of a tree is the product of the weights of its edges.
\end{defn}
In the example above, 
if the weight of the edge $i\to j$ is $a_{ij}$ then the
weight of the tree is $a_{16} a_{24} a_{32} a_{40} a_{54} a_{62} a_{74}$.
\begin{defn}
A {\em spanning tree} of a graph $G$ is a tree whose vertices are the same 
as the vertices of $G$ and whose edges are a subset of the edges of $G$.
\end{defn}
\begin{defn}
A {\em functional digraph} 
is a directed graph where 
each vertex is the tail of exactly one edge. 
\end{defn}
In a functional digraph, 
there may be many edges pointing into a vertex but only one pointing out.
A functional digraph is a collection of disjoint cycles whose vertices are
roots of trees leading
into them.  

\begin{ex}  This is a functional digraph:

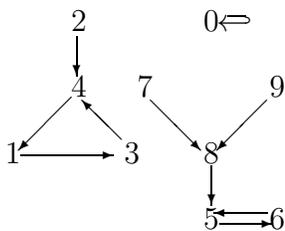
\begin{picture}(100,90)
\put(25,75){2}
\put(27,73){\vector(0,-1){15}}
\put(25,50){4}
\put(25,50){\vector(-1,-1){20}}
\put(0,25){1}
\put(6,28){\vector(1,0){35}}
\put(45,25){3}
\put(44,34){\vector(-1,1){15}}
\put(75,0){5}
\put(81,2){\vector(1,0){20}}
\put(100,0){6}
\put(99,6){\vector(-1,0){20}}
\put(75,25){8}
\put(78,24){\vector(0,-1){15}}
\put(50,50){7}
\put(55,49){\vector(1,-1){19}}
\put(100,50){9}
\put(99,49){\vector(-1,-1){19}}
\put(75,75){0}
\put(80,75){$\hookleftarrow$}
\put(84,80.5){\line(1,0){7}}
\end{picture}
\end{ex}

\noindent
A functional digraph
represents a function $f:\{0,1,2,\dots,n\}\to\{0,1,2,\dots,n\}$, 
where $f(i)=j$ if and only if the edge $i\to j$ is in the digraph.
\begin{defn}
Since each vertex $i$ ($\neq 0$ in a rooted tree) in a functional digraph
is the initial vertex for exactly one
edge, it makes sense to define $\verb+succ+(i)=j$ to be the terminal vertex
of the edge $i\to j$ in the (tree or) functional digraph.
\end{defn}
In the tree example above, $\verb+succ+(6)=2$; in the functional digraph,
$\verb+succ+(6)=5$.
\begin{defn}
A {\em happy functional digraph} is a functional digraph without an 
edge out of $0$, and in which $1$ is in the same connected component as $0$.  
\end{defn}
A happy functional digraph is a collection of trees leading into
disjoint cycles, together
with a tree rooted at 0 and also containing 1.
\begin{defn}
An {\em ascent} is an edge $i\to j$ where $j>i$.
\end{defn}
\begin{defn}
An {\em Escher cycle} is a cycle in which each edge except one is an
ascent.
\end{defn}
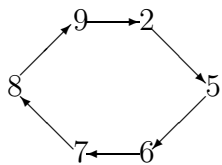
\begin{ex}

\begin{picture}(100,90)
\put(25,75){9}
\put(30,78){\vector(1,0){20}}
\put(50,75){2}
\put(55,75){\vector(1,-1){20}}
\put(75,50){5}
\put(75,50){\vector(-1,-1){20}}
\put(50,25){6}
\put(50,28){\vector(-1,0){20}}
\put(25,25){7}
\put(25,30){\vector(-1,1){20}}
\put(0,50){8}
\put(5,58){\vector(1,1){19}}
\end{picture}
\end{ex}

\vspace{-15pt}
\noindent
Note that each vertex is smaller than its successor, except
for the greatest vertex in the cycle, 9.
\begin{defn}
The ``na\"\i ve code'' for a tree is defined to be 
\[
{\rm na\ddot{\i} ve}=(\verb+succ+(1),\verb+succ+(2),\dots,\verb+succ+(n)).
\]
\end{defn}
The ``na\"\i ve code'' requires no work to find, but not every $n$-tuple
corresponds to a tree.  For example, the na\"\i ve code (3,2,0,5,4) 
would correspond to the following graph:

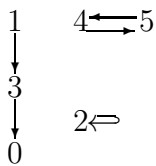
\begin{picture}(50,75)
\put(0,50){1}
\put(3,49){\vector(0,-1){15}}
\put(0,25){3}
\put(3,24){\vector(0,-1){15}}
\put(0,0){0}
\put(25,12){2}
\put(30,12){$\hookleftarrow$}
\put(34,17.5){\line(1,0){7}}
\put(25,50){4}
\put(30,50){\vector(1,0){18}}
\put(50,50){5}
\put(49,55){\vector(-1,0){18}}
\end{picture}

\noindent
This graph is not a tree because it has a loop and a cycle.  It is, however,
a happy functional digraph.

We borrow the notation of discrete geometry for some of the proofs in
this paper:
\begin{defn}
A {\em{signed set}} $S=S^+\sqcup S^-$, 
where $\sqcup$ represents the disjoint union, 
is an oriented 
zero-dimensional complex (that is, a collection of distinguishable points
that can be partitioned into two subsets,
one containing the elements considered ``positive'' and the other containing
the elements considered ``negative.'').  
\end{defn}
\begin{defn}
Let $T=T^+\sqcup T^-$ and $S=S^+\sqcup S^-$ be two signed sets.  
Their {\em difference} is defined to be the disjoint union of the
sets, with the following signs on elements of the union:  
\[(S-T)^+=S^+\sqcup T^- {\text{ and  }}
(S-T)^-=S^-\sqcup T^+.
\]
\end{defn}
\begin{ex}
If $S=\{a,b,c,-d,-e\}$ and $T=\{x,-y,-z\}$, then 
\[
S-T=\{a,b,c,y,z,-d,-e,-x\}.
\]
\end{ex}
\vspace{-30pt}
\begin{defn}
The {\em Kronecker delta} function $\delta_{xy}$ takes value 1 if $x=y$
and 0 otherwise.
\end{defn}
\begin{defn}
An {\em involution} $\phi:S\to S$ is a map on a signed set $S$ that
satisfies $\phi\circ\phi(x)=x$ for all $x\in S$.
\end{defn}
\begin{defn}
An involution is {\em sign-reversing} if for any $x\in S^+$, $\phi(x)\in S^-$
and for any $x\in S^-$, $\phi(x)\in S^+$.
\end{defn}
A sign-reversing involution does not have any fixed points.

\section{Some History}
In 1860,
Borchardt \cite{B} discovered through evaluation of a certain determinant
(namely, the principal (0,0)-minor of the matrix Tutte used a hundred years
later, see \S\ref{toggling}) that the number of labelled trees is
$\left(n+1\right)^{n-1}$.  Cayley \cite{Ca} independently derived 
this formula in 1889, and his short paper on the topic alludes to a 
bijection.  However, the invention of a coding algorithm
for trees, by Pr\"ufer in 1918, was the first combinatorial proof that this
is the formula for the number of trees.
His idea was that any tree can 
be encoded by a vector:  an ordered $(n-1)$-tuple
of labels chosen from $0$ to $n$.  This is done in 
such a way that the tree can be recovered from the code and vice versa.  
The number of possible codes (which is of course equal to the number of
possible trees) is $(n+1)^{n-1}$. 
\section{The Pr\"ufer Code}
In 1918, Pr\"ufer \cite{P} gave the following bijective proof 
of this formula.  

Given a labelled tree, we suppose that the least leaf is labelled 
$i_1$, and that $\verb+succ+(i_1)=j_1$.  Remove $i_1$ and its edge
from the tree, 
and let $i_2$ be the least leaf on the new tree, with $\verb+succ+(i_2)=j_2$.
If we repeat this process until there are only two vertices left, the 
Pr\"ufer code $(j_1,\dots,j_{n-1})$ uniquely determines the tree.

To recover the tree from any $(n-1)$-tuple, we note that for each vertex 
except 
the root, the number of occurrences of that label in the Pr\"ufer code
is equal to the indegree of that vertex.  The number of occurrences of
$0$ in the code is one less than the indegree of 0.  There must be at least two
labels that don't appear in the code, since there are $n+1$
vertices and only $n-1$ entries in the code.  Any nonzero vertex not occurring
in the code is a leaf in the original tree, so we know that the least
one, $i_1$, has $\verb+succ+(i_1)=j_1$, the first vertex in the code.  
We can also tell whether any new leaves were formed when $i_1$ was removed
because we know the indegree of $j_1$.  Step by step, from beginning to end,
we can reconstruct each edge of the tree.

Hence, the Pr\"ufer code gives a bijection between trees with $n+1$ vertices
and $(n-1)$-tuples of the vertex-labels.  Since the number of $(n-1)$-tuples is
clearly $(n+1)^{n-1}$, this bijection proves the formula that Borchardt
discovered.

However, the algorithm is a bit unnatural.  
The inverse does not undo the steps in
the backwards order; we have to look at the overall code and decipher what
had to be true in the tree by starting from the  beginning of the code and
working our way to the end.

\subsection{An example}\label{proofer}
Consider the tree: 

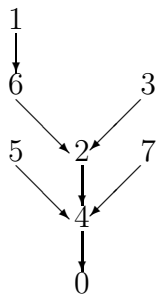
\vspace{15pt}
\begin{picture}(100,100)
\put(50,0){0}
\put(50,25){4}
\put(53,23){\vector(0,-1){15}}%  4->0
\put(25,50){5}
\put(28,48){\vector(1,-1){20}}%  5->4
\put(50,50){2}
\put(53,48){\vector(0,-1){15}}%  2->4
\put(75,50){7}
\put(75,48){\vector(-1,-1){19}}% 7->4
\put(25,75){6}
\put(75,75){3}
\put(28,73){\vector(1,-1){20}}%  6->2
\put(75,73){\vector(-1,-1){19}}% 3->2
\put(25,100){1}
\put(28,98){\vector(0,-1){15}}%  1->6
\end{picture}

\noindent
with leaves $\{1,5,3,7\}$.  
Step by step, we build up the code and remove leaves from the tree. 
First, we see that $1$ is the least leaf, so we write down
$\verb+succ+(1)$ and remove $1$ from the tree.  

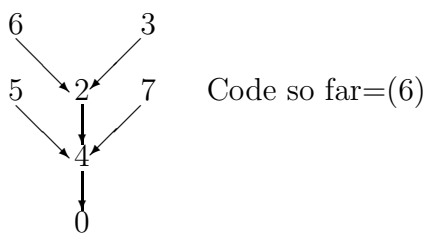
\begin{picture}(150,100)
\put(50,0){0}
\put(50,25){4}
\put(53,23){\vector(0,-1){15}}%  4->0
\put(25,50){5}
\put(28,48){\vector(1,-1){20}}%  5->4
\put(50,50){2}
\put(53,48){\vector(0,-1){15}}%  2->4
\put(75,50){7}
\put(75,48){\vector(-1,-1){19}}% 7->4
\put(25,75){6}
\put(75,75){3}
\put(28,73){\vector(1,-1){20}}%  6->2
\put(75,73){\vector(-1,-1){19}}% 3->2
\put(100,50){Code so far=(6)}
\end{picture}

\noindent
Here, the removal of $1$ created a new leaf.  Now the leaves are 
$\{5,6,3,7\}$, so the new least leaf is $3$.

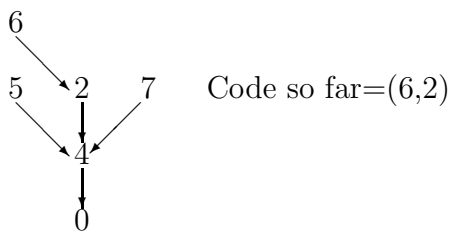
\begin{picture}(150,100)
\put(50,0){0}
\put(50,25){4}
\put(53,23){\vector(0,-1){15}}%  4->0
\put(25,50){5}
\put(28,48){\vector(1,-1){20}}%  5->4
\put(50,50){2}
\put(53,48){\vector(0,-1){15}}%  2->4
\put(75,50){7}
\put(75,48){\vector(-1,-1){19}}% 7->4
\put(25,75){6}
\put(28,73){\vector(1,-1){20}}%  6->2
\put(100,50){Code so far=(6,2)}
\end{picture}

\noindent
The next leaf to fall off of the tree is $5$, leaving us with the following
tree and code:

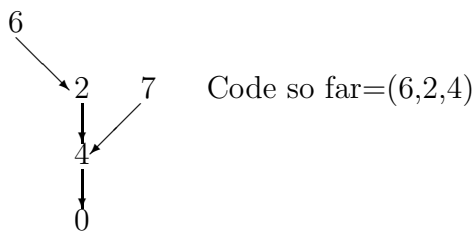
\begin{picture}(150,100)
\put(50,0){0}
\put(50,25){4}
\put(53,23){\vector(0,-1){15}}%  4->0
\put(50,50){2}
\put(53,48){\vector(0,-1){15}}%  2->4
\put(75,50){7}
\put(75,48){\vector(-1,-1){19}}% 7->4
\put(25,75){6}
\put(28,73){\vector(1,-1){20}}%  6->2
\put(100,50){Code so far=(6,2,4)}
\end{picture}

\noindent
No new leaves have been created, so the smallest leaf now is $6$ and we
remove it.

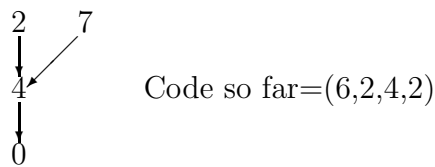
\begin{picture}(150,75)
\put(50,0){0}
\put(50,25){4}
\put(53,23){\vector(0,-1){15}}%  4->0
\put(50,50){2}
\put(53,48){\vector(0,-1){15}}%  2->4
\put(75,50){7}
\put(75,48){\vector(-1,-1){19}}% 7->4
\put(100,25){Code so far=(6,2,4,2)}
\end{picture}

\noindent
Now that we've removed both $3$ and $6$, the
indegree of $2$ is $0$. We remove $2$ to obtain:

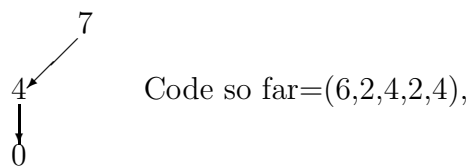
\begin{picture}(150,75)
\put(50,0){0}
\put(50,25){4}
\put(53,23){\vector(0,-1){15}}%  4->0
\put(75,50){7}
\put(75,48){\vector(-1,-1){19}}% 7->4
\put(100,25){Code so far=(6,2,4,2,4),}
\end{picture}

\noindent
and finally:

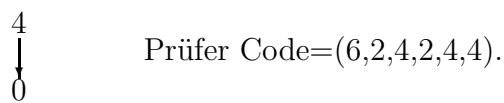
\begin{picture}(150,50)
\put(50,0){0}
\put(50,25){4}
\put(53,23){\vector(0,-1){15}}%  4->0
\put(100,15){Pr\"ufer Code=(6,2,4,2,4,4).}
\end{picture}

\subsection{Finding the tree for a code}
To get the other direction of the bijection, we start by counting 
occurrences of each vertex label in the code
to find the list of indegrees.  
(The indegree of $0$ is one greater than the number
of occurrences of $0$ in the code.)  For the code (6,2,4,2,4,4) we have

\begin{center}
\begin{tabular}{c|c}
Vertex & Indegree\\
\hline
0 & 1\\
1 & 0\\
2 & 2\\
3 & 0\\
4 & 3\\
5 & 0\\
6 & 1\\
7 & 0
\end{tabular}
\end{center}

\noindent
The four vertices with indegree of 0 are the leaves on the original tree.
So far, our knowledge consists of this:

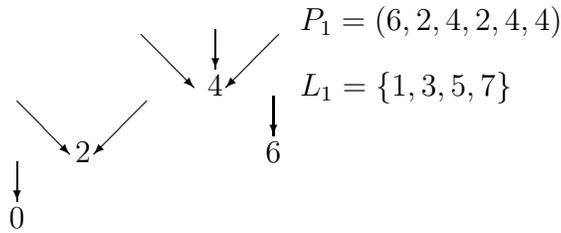
\begin{picture}(200,100)
\put(0,0){0}
\put(3,25){\vector(0,-1){15}}
\put(25,25){2}
\put(3,48){\vector(1,-1){20}}
\put(52,48){\vector(-1,-1){20}}
\put(75,50){4}
\put(50,73){\vector(1,-1){20}}
\put(102,73){\vector(-1,-1){20}}
\put(78,75){\vector(0,-1){15}}
\put(97,25){6}
\put(100,50){\vector(0,-1){15}}
\put(110,50){$L_1 = \{1,3,5,7\}$}
\put(110,75){$P_1 = (6,2,4,2,4,4)$}
\end{picture}

\vspace{5pt}
\noindent where $P_i$ is the code at step $i$ and 
$L_i$ is the Leaf Set at step $i$.  
The Leaf Set $L_i$ consists of all the vertices whose labels
are not listed in $P_i$ and whose outgoing edges have yet to be determined.
It is actually the set of vertices
that are leaves after all of the previous ``least leaves''  
have been removed.

Since the smallest leaf in this example is 1, and we know the code starts with 6, we can
see that the edge whose head is 6 must have tail 1.  We also see that removing
1 from the tree created a new leaf, 6, so we add 6 to the Leaf Set.

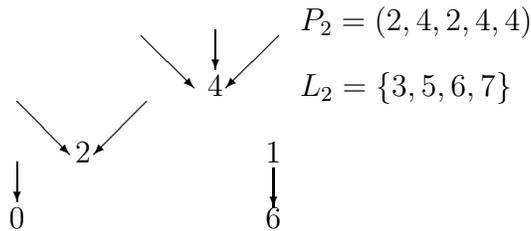
\begin{picture}(200,100)
\put(0,0){0}
\put(3,25){\vector(0,-1){15}}
\put(25,25){2}
\put(3,48){\vector(1,-1){20}}
\put(52,48){\vector(-1,-1){20}}
\put(75,50){4}
\put(50,73){\vector(1,-1){20}}
\put(102,73){\vector(-1,-1){20}}
\put(78,75){\vector(0,-1){15}}
\put(97,0){6}
\put(100,23){\vector(0,-1){15}}
\put(97,25){1}
\put(110,50){$L_2 = \{3,5,6,7\}$}
\put(110,75){$P_2 = (2,4,2,4,4)$}
\end{picture}

\vspace{5pt}
\noindent The least leaf of $L_2$ is 3, so its edge points
to 2, the first element in $P_2$.  
Removing 3 does not create a new leaf because 2 appears twice in $P_2$.

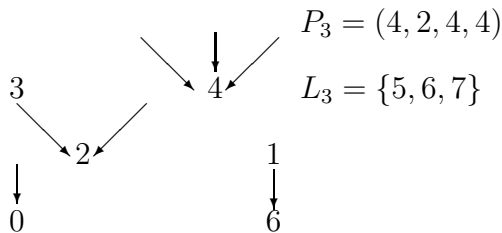
\begin{picture}(200,100)
\put(0,0){0}
\put(3,25){\vector(0,-1){15}}
\put(25,25){2}
\put(3,48){\vector(1,-1){20}}
\put(0,50){3}
\put(52,48){\vector(-1,-1){20}}
\put(75,50){4}
\put(50,73){\vector(1,-1){20}}
\put(102,73){\vector(-1,-1){20}}
\put(78,75){\vector(0,-1){15}}
\put(97,0){6}
\put(100,23){\vector(0,-1){15}}
\put(97,25){1}
\put(110,50){$L_3 = \{5,6,7\}$}
\put(110,75){$P_3 = (4,2,4,4)$}
\end{picture}

\vspace{5pt}
\noindent  The least leaf of $L_3$ is 5.  It will point
at 4, and removing 5 will not create a new leaf.

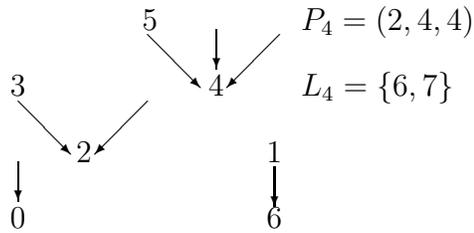
\begin{picture}(200,100)
\put(0,0){0}
\put(3,25){\vector(0,-1){15}}
\put(25,25){2}
\put(3,48){\vector(1,-1){20}}
\put(0,50){3}
\put(52,48){\vector(-1,-1){20}}
\put(75,50){4}
\put(52,73){\vector(1,-1){20}}
\put(50,75){5}
\put(102,73){\vector(-1,-1){20}}
\put(78,75){\vector(0,-1){15}}
\put(97,0){6}
\put(100,23){\vector(0,-1){15}}
\put(97,25){1}
\put(110,50){$L_4 = \{6,7\}$}
\put(110,75){$P_4 = (2,4,4)$}
\end{picture}

\vspace{5pt}
\noindent
Once we remove 5, the smallest element of $L_4$ is 6, 
so there is an edge $6\to 2$.  
Also, removing 6 will turn 2 into a leaf, since this is the only occurrence
of 2 in $P_4$.

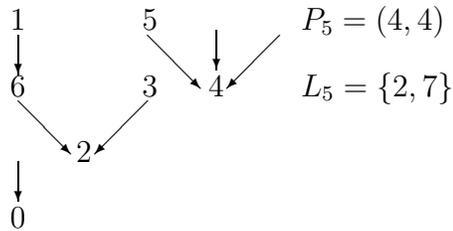
\begin{picture}(200,100)
\put(0,0){0}
\put(3,25){\vector(0,-1){15}}
\put(25,25){2}
\put(3,48){\vector(1,-1){20}}
\put(0,50){6}
\put(0,75){1}
\put(3,73){\vector(0,-1){15}}
\put(50,50){3}
\put(52,48){\vector(-1,-1){20}}
\put(75,50){4}
\put(52,73){\vector(1,-1){20}}
\put(50,75){5}
\put(102,73){\vector(-1,-1){20}}
\put(78,75){\vector(0,-1){15}}
\put(110,50){$L_5 = \{2,7\}$}
\put(110,75){$P_5 = (4,4)$}
\end{picture}

\vspace{5pt}
\noindent
The smallest element of $L_5$ is 2, so it must point at 4.
Since there is still an unaccounted-for edge into 4, removing 2 does not
make 4 into a leaf.  Thus 7 will be the only leaf left after that step.\\
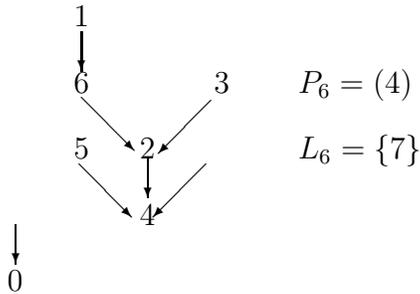
\begin{picture}(200,125)
\put(0,0){0}
\put(3,25){\vector(0,-1){15}}
\put(50,50){2}
\put(28,73){\vector(1,-1){20}}
\put(25,75){6}
\put(25,100){1}
\put(28,98){\vector(0,-1){15}}
\put(78,75){3}
\put(77,72){\vector(-1,-1){20}}
\put(50,25){4}
\put(27,48){\vector(1,-1){20}}
\put(25,50){5}
\put(75,48){\vector(-1,-1){20}}
\put(53,50){\vector(0,-1){15}}
\put(110,50){$L_6 = \{7\}$}
\put(110,75){$P_6 = (4)$}
\end{picture}

\vspace{5pt}
\noindent  Now, since 7 is the smallest leaf, its edge has head at 4.  From
that we can also conclude that the edge $4\to 0$ is the remaining edge in
the tree.  In general, whichever vertex did not yet have an outgoing edge
will have to point to $0$ at the end.

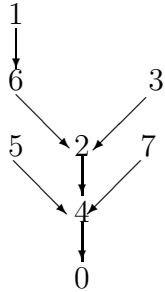
\begin{picture}(200,125)
\put(50,0){0}
\put(53,25){\vector(0,-1){15}}
\put(50,50){2}
\put(28,73){\vector(1,-1){20}}
\put(25,75){6}
\put(25,100){1}
\put(28,98){\vector(0,-1){15}}
\put(78,75){3}
\put(77,72){\vector(-1,-1){20}}
\put(50,25){4}
\put(27,48){\vector(1,-1){20}}
\put(25,50){5}
\put(75,48){\vector(-1,-1){20}}
\put(75,50){7}
\put(53,50){\vector(0,-1){15}}
\end{picture}

\vspace{5pt}
\noindent
Given the code, we were able to reconstruct the tree, and this can be done
no matter what $(n-1)$-tuple we are given.  It is clear
that this algorithm is the inverse of the 
algorithm given by Pr\"ufer.

\section{The Matrix Tree Theorem}\label{toggling}
In 1948, Tutte \cite{T} 
associated a matrix $A_T$ to the complete loopless
directed graph on vertices 
$\{0,\dots,n\}$, with edge from $i$ to $j$
of weight $a_{ij}$.  The general matrix is $A_T=(A_{ij})$, 
with $i$ and $j$ indexed from $0$ to $n$:
$$
A_{ij} ~=~ \left\{
           \begin{array}{ll}
                 -a_{ij} & i \neq j\\[1ex]
    \displaystyle{\sum_{k \neq i, 0\leq k \leq n}}  a_{ik} & i = j.
           \end{array}
           \right.
$$
The diagonal entry in row $i$ is the sum of the weights of the edges with tail
at $i$.
The row sums of such a matrix are zero, so the determinant of the matrix is
zero.  However, the following result by Tutte is very useful.
Denote by $A$ the $n\times n$
submatrix of $A_T$ obtained by crossing out its zeroth row and column.

\begin{theorem}[Matrix Tree Theorem]
The determinant of $A$ is the sum of the weights 
of all spanning trees (rooted at vertex $0$) of the graph. 
\end{theorem}
Zeilberger \cite{Z} published a nice bijective proof, also discovered
independently by Garsia.
A bijective proof of a more
general version of the theorem is due to Chaiken \cite{C}.  
We will think of the entries in our matrices as being indeterminates.
When the $i,j$-entry of the matrix (not on the diagonal) consists of a sum of
$k$ indeterminates, the matrix corresponds to a 
graph with $k$ edges $i\to j$, each having monomial weight.
Note that if
$a_{ij}$ is an integer, it can represent the number of edges $i\to j$ in
a graph (if $a_{ij}=0$, then there is no edge $i\to j$).  
Then $\det(A)$ is the number of spanning trees of the graph.

Throughout this dissertation we will be defining 
signed sets that come from matrices.
Each element of a matrix set is an array consisting of exactly one monomial 
entry from the matrix 
in each row and each column.  Each array comes with the sign 
corresponding to the array position in the determinant.
An element of a matrix set can be thought of as 
a signed permutation times a diagonal matrix.
The matrix set corresponding to a matrix $M$ consists of all possible such 
arrays.

The matrix $\hat{A}=(a_{ij})$ (where $i$ and $j$ are indexed from
$0$ to $n$) has the indeterminate weight corresponding to the edge
$i\to j$ in its $i,j$-entry.  If we formally
subtract this matrix from the diagonal $(n+1)\times (n+1)$
matrix $\hat{D}$ whose $i^{\rm th}$ 
diagonal entry is $\sum_{j=0}^n a_{ij}$, without 
simplifying, then we
obtain a matrix $\hat{D}-\hat{A}$ 
whose row sums are zero:  this matrix corresponds
to the complete directed graph {\em with} 
loops.  It differs from Tutte's matrix only
by the presence of 
$a_{ii}-a_{ii}$ in the $i^{\rm th}$ diagonal entry--essentially
we have added zero to each diagonal entry in Tutte's matrix.  Obviously
this doesn't change the $(0,0)$-minor; loops never appear in trees. 

Zeilberger's bijective proof \cite{Z} 
of the Matrix Tree Theorem hinges on the idea
that every functional digraph with a 
cycle corresponds to an array some of whose entries occur both on the diagonal
and off the diagonal of Tutte's submatrix $A$, 
with opposite signs.  In the determinant, these terms would cancel.  
He effectively 
introduces a surjective map from the matrix set 
corresponding to $A$ to the set of digraphs representing functions
from $\{1,2,\dots,n\}$ to $\{0,1,\dots,n\}$ according to the following rule:
The entry in row $i$ of the array represents the edge from $i$, and if this
entry is $\pm b_j$, either on or off the diagonal, then the edge is $i\to j$.

\begin{ex}
For $n=2$, the submatrix of $A$ is $\left[\begin{smallmatrix} 
a_{10}+a_{12}& -a_{12}\\-a_{21}&a_{20}+a_{21}\end{smallmatrix}\right]$.  
The matrix set is
\[
\left\{\left[\begin{smallmatrix} a_{10}&\\&a_{20}\end{smallmatrix}\right],
\left[\begin{smallmatrix} a_{10}&\\&a_{21}\end{smallmatrix}\right],
\left[\begin{smallmatrix} a_{12}&\\&a_{20}\end{smallmatrix}\right],
\left[\begin{smallmatrix} a_{12}&\\&a_{21}\end{smallmatrix}\right],
\left[\begin{smallmatrix} &-a_{12}\\-a_{21}&\end{smallmatrix}\right]\right\}
\]
and the surjective map is given in the following diagram:

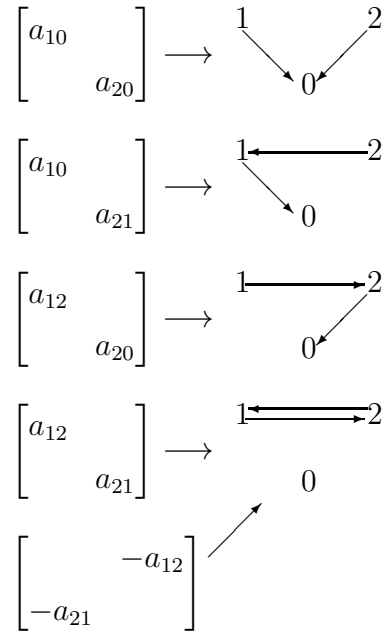
\begin{picture}(130,240)
\put(0,212){$\begin{bmatrix} a_{10}&\\&a_{20}\end{bmatrix}\longrightarrow$}
\put(85,225){1}
\put(110,200){0}
\put(135,225){2}
\put(88,224){\vector(1,-1){19}}
\put(135,224){\vector(-1,-1){19}}
\put(0,162){$\begin{bmatrix} a_{10}&\\&a_{21}\end{bmatrix}\longrightarrow$}
\put(85,175){1}
\put(110,150){0}
\put(135,175){2}
\put(88,174){\vector(1,-1){19}}
\put(135,178){\vector(-1,0){45}}
\put(0,112){$\begin{bmatrix} a_{12}&\\&a_{20}\end{bmatrix}\longrightarrow$}
\put(85,125){1}
\put(110,100){0}
\put(135,125){2}
\put(89,128){\vector(1,0){45}}
\put(135,124){\vector(-1,-1){19}}
\put(0,62){$\begin{bmatrix} a_{12}&\\&a_{21}\end{bmatrix}\longrightarrow$}
\put(85,75){1}
\put(110,50){0}
\put(135,75){2}
\put(89,77){\vector(1,0){45}}
\put(135,81){\vector(-1,0){45}}
\put(0,12){$\begin{bmatrix} &-a_{12}\\-a_{21}&\end{bmatrix}$}
\put(75,25){\vector(1,1){20}}
\end{picture}

\noindent
Note that the only graph with a cycle gets mapped to twice.  Since we
are only interested in counting trees, we can eliminate graphs with cycles
and instead map the two preimages to one another.
\end{ex}

Using these ideas, Zeilberger constructs what amounts to a sign-reversing
involution on the matrix set corresponding to $A$ minus the set of trees.

In our case, we think of $A$ as being morally equal to 
$\hat{D}-\hat{A}$, and ``on the
diagonal'' as meaning ``occurring in $\hat{D}$'' and ``off the
diagonal'' as meaning ``occurring in the matrix $-\hat{A}$.'' 
By defining these terms in this way, we allow for loops.  Most of our
algorithms for finding codes using a matrix method will require us to
know how to ``toggle the diagonality'' of a cycle.  Toggling the diagonality
of a cycle in an array in a matrix set 
simply entails finding the unique array in the same set that satisfies two 
conditions:  (1)
the variable corresponding to any edge {\em not}
in the cycle is in the same location as in the original array, 
and (2) any variable corresponding to an 
edge that {\em is} in the cycle occurs within the same row but 
has the opposite ``diagonality'' from its location in the
original array.  Toggling the diagonality of a cycle is a sign-reversing 
involution
on the matrix set's subset corresponding to graphs containing
cycles.  An off-diagonal cycle will always come with a negative sign because
a cycle of odd length has a permutation 
sign of +1, but an odd number of negative
terms; a cycle of even length has an even number of negative terms but a
negative sign.

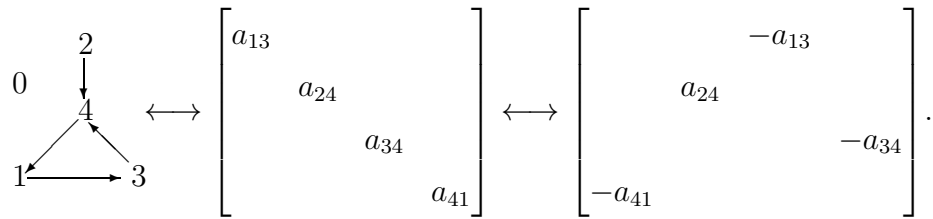
\begin{ex} 

\begin{picture}(400,100)
\put(0,60){0}
\put(25,75){2}
\put(27,73){\vector(0,-1){15}}
\put(25,50){4}
\put(25,50){\vector(-1,-1){20}}
\put(0,25){1}
\put(6,28){\vector(1,0){35}}
\put(45,25){3}
\put(44,34){\vector(-1,1){15}}
\put(50,50){$\longleftrightarrow$}
\put(75,50){$\begin{bmatrix}a_{13}&&&\\&a_{24}&&\\&&a_{34}&\\&&&a_{41}
\end{bmatrix}\longleftrightarrow 
\begin{bmatrix}&&-a_{13}&\\&a_{24}&&\\&&&-a_{34}\\-a_{41}&&&\end{bmatrix}$.} 
\end{picture}

\noindent
The graph above contains a cycle; the elements of the
matrix set corresponding to the (0,0)-minor of 
$\hat{D}-\hat{A}$ that correspond to this tree are both above: the
one on the left
consists entirely of entries from $\hat{D}$ while the one on the right has some
entries from $-\hat{A}$.
$a_{24}$ corresponds to the edge $2\to 4$ which is not in a cycle, so it
appears on the diagonal in both arrays, but the cycle $(134)$ could 
appear either on or off the diagonal.  The sign of the first array is +1
because all entries are on the diagonal.  The second array turns out to be
negative because the 3-cycle has sign +1 but there are 3 negative entries.
\end{ex}

For loops, it is a little bit less clear:

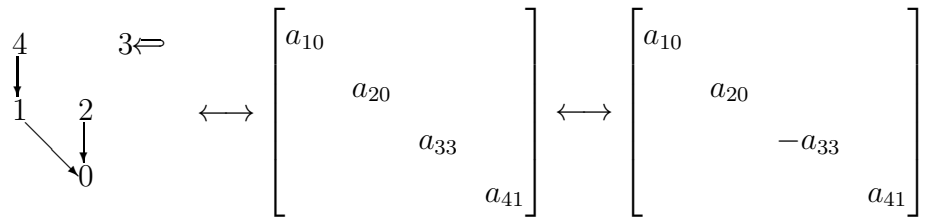
\begin{ex}

\begin{picture}(200,100)
\put(25,25){0}
\put(25,50){2}
\put(27,49){\vector(0,-1){15}}
\put(0,50){1}
\put(5,49){\vector(1,-1){20}}
\put(0,75){4}
\put(2,74){\vector(0,-1){15}}
\put(40,75){3}
\put(45,75){$\hookleftarrow$}
\put(49,80.5){\line(1,0){7}}
\put(70,50){$\longleftrightarrow
\begin{bmatrix}a_{10}&&&\\&a_{20}&&\\&&a_{33}&\\&&&a_{41}
\end{bmatrix}\longleftrightarrow
\begin{bmatrix}a_{10}&&&\\&a_{20}&&\\&&-a_{33}&\\&&&a_{41}\end{bmatrix}$.}
\end{picture}

\noindent
Here, although the entries are all apparently on the diagonal, we think of
the diagonality of the loop at 3 as having changed from the first matrix to 
the 
second.  The first array consists of entries only from $\hat{D}$ while the
$-a_{33}$ in the second one is an entry from $-\hat{A}$.
\end{ex}

If a graph has more than one cycle (including loops), we raise the issue of 
which cycle's diagonality gets toggled.  
Zeilberger arbitrarily chose to move 
the cycle with the smallest
element in it; we arbitrarily choose to move the cycle with the largest.  
All choices are equally valid but result in slightly different codes.  
The choice of the largest element in a cycle is consistent with some
tree surgical methods that give the same bijections as our matrix methods.

\section{Linear Algebra Setup}\label{setup}

If we set $a_{ij} = b_j$ for all $i,j$
in Tutte's matrix $A_T$, we get a matrix with each entry in column $j$ 
$=-b_j$ except on the diagonal.  We can calculate the $(0,0)$-minor
using row and column operations.  

If we ignore the zeroth row and column, we could find the 
determinant using the following operations.  We start with the submatrix $A$:
\[
\det\left[
\begin{array}{ccc}
b_0+b_2+b_3 & -b_2 & -b_3\\
-b_1 & b_0+b_1+b_3 & -b_3\\
-b_1 & -b_2 & b_0+b_1+b_2
\end{array}
\right].
\]
Subtract row 2 from row 3:
\[
=\det\left[
\begin{array}{ccc}
b_0+b_2+b_3 & -b_2 & -b_3\\
-b_1 & b_0+b_1+b_3 & -b_3\\
0 & -b_0-b_1-b_2-b_3 & b_0+b_1+b_2+b_3
\end{array}
\right]
\]
Add column 3 to column 2:
\[
=\det\left[
\begin{array}{ccc}
b_0+b_2+b_3 & -b_2-b_3 & -b_3\\
-b_1 & b_0+b_1 & -b_3\\
0 & 0 & b_0+b_1+b_2+b_3
\end{array}
\right]
\]
Subtract row 1 from row 2:
\[
=\det\left[
\begin{array}{ccc}
b_0+b_2+b_3 & -b_2-b_3 & -b_3\\
-b_0-b_1-b_2-b_3 & b_0+b_1+b_2+b_3 & 0\\
0 & 0 & b_0+b_1+b_2+b_3
\end{array}
\right]
\]
Add column 2 to column 1:
\[
=\det\left[
\begin{array}{ccc}
b_0 & -b_2-b_3 & -b_3\\
0 & b_0+b_1+b_2+b_3 & 0\\
0 & 0 & b_0+b_1+b_2+b_3
\end{array}
\right]
\]
(Call this last matrix $M$.)  Now it is evident (since we have an 
upper-triangular matrix) that
$\det M = b_0\,[b_0+b_1+b_2+b_3]^2$.  In general, $\det M = 
b_0 \, \left[\sum_{j=0}^n b_j\right]^{n-1}$. The number of trees is
$(n+1)^{(n-1)}$, and it is clear that this is also the number of terms in 
$\det M$ (we have an $(n-1)$-fold product of a sum of $n+1$ terms). 
Let sequences of the $b_j$ as read down the diagonal of the matrix be called
``codes.''  One would
like to have a bijection relating these codes to trees.  Each array of 
diagonal entries from $M$ should correspond to a tree.  

Note that in a matrix with this much redundancy, there are many different 
sequences of row and column
operations that can lead to an easily calculated determinant.

In the course of this research we found that allowing loops was more natural.
Consequently, instead of Tutte's matrix $A$ we use variations on 
$\hat{D}-\hat{A}$ as defined in \S\ref{toggling}.  For our purposes, we
will set $a_{ij}=b_j$ in $\hat{A}$ and $a_{ij}=B_j$ in $\hat{D}$.  At the
end of the long process of row and/or column operations, we set $B_j=b_j$.

\chapter{The Happy Code}\label{happyhell}

We can use the Matrix Tree Theorem to find a more ``natural'' code than the 
Pr\"ufer code by expanding on Knuth's ideas in \cite{K}.  As mentioned
in \S\ref{setup},
we specialize $a_{ij}$ to be $b_j$ in $\hat{A}$ and $B_j$ in $\hat{D}$.
Following Knuth, we introduce another indeterminate $\lambda$,
which will be a placeholder, by
putting $\lambda-b_0$ in the $(0,0)$-entry in the matrix, calling this new 
matrix $M_0'$.  We keep in mind that we are interested in the coefficient of
$\lambda$ in the determinant of $M_0'$, since it is equal to the 
$(0,0)$-minor of the original Matrix Tree Theorem matrix.  We will do 
row operations to form
a series of matrices, all with the same determinant.  The coefficient of
$\lambda$ in the final determinant represents the
sum of the weights of all the trees, because that was true of the original
matrix; the row operations do not affect that.  
The sequence of matrices is formed by subtracting the zeroth row from each of
the other rows, one at a time.  (In \cite{K}, the row operations are all 
performed simultaneously.)

Specifically, we begin with the matrix $M_0'$ whose $i,j$-entry is $-b_j$ when
$i\neq j$ and whose $i^{\rm th}$ diagonal entry is 
$-b_i+\delta_{i0}\lambda + (1-\delta_{i0})\sum_{j=0}^n B_j$.
(If $B_j$ is set equal to $b_j$ then the row sums are zero for rows $1$ 
through $n$.  Using $B_j$ for the diagonal entries enables us to keep track 
of loops.)   Let $B=\sum_{j=0}^n B_j$.
\[
M_0'=\left[\begin{array}{ccccc}
\lambda-b_0 & -b_1 & -b_2 & ... &-b_n\\
-b_0& B-b_1 & -b_2 & ... &-b_n \\
.&.&.&...&.\\
.&.&.&...&.\\
.&.&.&...&.\\
-b_0 &-b_1&-b_2&...&B-b_n
\end{array}\right].
\]
Subtract row $0$ from row $n$, without cancelling anything.  Then
\[
M_1=\left[\begin{array}{ccccc}
\lambda-b_0 & -b_1 & -b_2 & ... &-b_n\\
-b_0& B-b_1 & -b_2 & ... &-b_n \\
.&.&.&...&.\\
.&.&.&...&.\\
.&.&.&...&.\\
-\lambda+b_0-b_0 &b_1-b_1&b_2-b_2&...&b_n+B-b_n
\end{array}\right].
\]
The next step consists of arithmetic within entries:
\[
M_1'=\left[\begin{array}{ccccc}
\lambda-b_0 & -b_1 & -b_2 & ... &-b_n\\
-b_0& B-b_1 & -b_2 & ... &-b_n \\
.&.&.&...&.\\
.&.&.&...&.\\
.&.&.&...&.\\
-\lambda &0&0&...&B
\end{array}\right].
\]
Next, subtract row $0$ from row $n-1$, again without cancelling; repeat the
process.  The $i^{\rm th}$ step is:
\[
M_i=\left[\begin{smallmatrix}
\lambda-b_0 & -b_1 & \dots & -b_{n-i+1} & -b_{n-i+2}& \dots & -b_n\\
-b_0 & B-b_1 & \dots & -b_{n-i+1} &-b_{n-i+2}& \dots & -b_n\\
\vdots & \vdots & \ddots& \vdots &\vdots& & \vdots\\
-b_0-\lambda+b_0 & -b_1+b_1 & \dots & B-b_{n-i+1}+b_{n-i+1}
  &-b_{n-i+2}+b_{n-i+2}&\dots&-b_n+b_n\\
-\lambda& 0 &\dots&0 & B & \dots & 0\\
\vdots&\vdots&&\vdots&\vdots&\ddots&\vdots\\
-\lambda& 0 & \dots & 0 &0&\dots & B
\end{smallmatrix}\right],
\]
where the complicated row is row $n-i+1$.  Remember that the matrix is
indexed from 0 to $n$.
\[
M_i'=\begin{bmatrix}
\lambda-b_0 & -b_1 & \dots & -b_{n-i+1} & -b_{n-i+2}& \dots & -b_n\\
-b_0 & B-b_1 & \dots & -b_{n-i+1} &-b_{n-i+2}& \dots & -b_n\\
\vdots & \vdots & \ddots& \vdots &\vdots& & \vdots\\
-\lambda&0 &\dots&B&0&\dots&0\\
-\lambda& 0 &\dots&0 & B & \dots & 0\\
\vdots&\vdots&&\vdots&\vdots&\ddots&\vdots\\
-\lambda& 0 & \dots & 0 &0&\dots & B
\end{bmatrix}
\]
The last matrix is 
\[
M_n'=\left[\begin{array}{ccccc}
\lambda-b_0 & -b_1 & -b_2 & ... &-b_n\\
-\lambda& B & 0 & ... &0 \\
.&.&.&...&.\\
.&.&.&...&.\\
.&.&.&...&.\\
-\lambda &0&0&...&B
\end{array}\right].
\]
The coefficient of $\lambda$ in the determinant of this matrix is 
\[S=B^n - b_1B^{n-1} - Bb_2B^{n-2}-B^2b_3B^{n-3}-...-B^{n-1}b_n,\]
where we write each term with its
factors in the same order in which their columns appeared in the final 
matrix, $M_n'$.

\section{The Sets}
We define a sequence of signed
sets $A_0,A_0',A_1,A_1',\dots,A_{n+1},A_{n+1}'$.  
$A_0$ is the set of trees on vertices $0,\dots,n$, where each tree comes
with a positive sign.

The sets $A_0',A_1,A_1',\dots, A_{n}'$ are
matrix sets as described in \S\ref{toggling}:  For $1\leq i\leq n$, 
$A_i$ is the matrix set of arrays 
corresponding to $M_i$ and for $0\leq i\leq n$, $A_i'$ 
is the matrix set of arrays from $M_i'$.  For
example, when $n=2$, two of the elements of $A_2$ are: 
\begin{math}\left[ \begin{smallmatrix} \lambda & &\\ & B_0 &\\ & & 
B_1 \end{smallmatrix}\right]\end{math} and \begin{math}\left[
\begin{smallmatrix} & -b_1 & \\ -\lambda & & \\ & & B_2\end{smallmatrix}
\right]\end{math}.  
%This will avoid any ambiguity when a monomial occurs multiple times 
%and/or with different signs in a determinant.
(These arrays with one element in each row 
and column are understood to come with the sign they would have in the 
determinant.)

$A_{n+1}$ is the set of signed monomials (written as ordered $n$-tuples) 
occurring in $S$, the coefficient of $\lambda$ in the determinant of $M_n'$:
\[
A_{n+1}=B^n-\left(\{b_1\}\times B^{n-1}\right)-\left(B\times \{b_2\}\times 
B^{n-2}\right)-\hdots-\left(B^{n-1}\times\{b_n\}\right).
\]  
Here, we think of $B$ as $B=\{B_0,B_1,\hdots,B_n\}$
and $B^k$ as the $k$-fold direct product of $B$ with itself.
We write the factors
in the left-to-right order of the columns in which the entries appeared.
The final set, $A_{n+1}'$, is the set of monomials (all positive now)
remaining when $B_j$ is set equal to $b_j$ and arithmetic is done on $S$:  
$A_{n+1}'=\{b_0\}\times B^{n-1}$.  \begin{comment}(Actually, what
we have done is to set the first $B$ in $B^n$ equal to 
$b_0 + b_1 +\dots + b_n$,
and then allowed $b_j$ to commute with the $B_k$s, while not allowing the
$B_k$s to commute with each other.)\end{comment}
Ignoring the initial $b_0$, this is isomorphic to the set of codes (the 
codes are simply the subscripts of these monomials taken in order).

\section{The involutions}
We define a sequence of sign-reversing involutions 
$\phi_0,\phi_0',\phi_1,\phi_1',\dots,\phi_n',\phi_{n+1}$ 
on differences of two consecutive sets.  
In this set-up, when we write a negative sign in front of an array  
it implies that the matrix comes from the subtracted set.  

\subsubsection{Defining $\phi_0$}
$\phi_0: A_0-A_0'\to A_0-A_0'$ is defined as follows.  If $t$ is a tree, 
then $\phi_0 (t)$
is the negative of the array 
given by the bijective proof of the Matrix Tree Theorem:  
in the $i^{\rm th}$ diagonal, 
the $B_j$ term is taken if $\verb+succ+(i)=j$.  If $t$ is an
array in the negative matrix set, we look at the graph formed by the edges 
$i\to j$
for all $i,j$ where an indeterminate
with the subscript $j$ is in the $i^{\rm th}$
row of $t$.  If this is a tree, then it is $\phi_0(t)$.  If not, then
$\phi_0(t)$ can be found by toggling the diagonality of the cycle containing
the greatest vertex in a cycle in this graph
(see \S\ref{toggling}). 
In the case where a tree matches an
array, this is clearly a
sign-reversing involution.  For the case of the pairings of
two elements of $A_0'$, since we only moved one cycle on or off the diagonal,
and we know how to find it, it is clear that repeating the process will get
us back where we started.  $\phi_0$ is sign-reversing, as noted in 
\S\ref{toggling}

\subsubsection{Defining $\phi_i'$ for $0\leq i\leq n-1$}
Recall that for $0\leq i\leq n-1$, $M_{i+1}$ 
is obtained from $M_i'$ by row subtraction without cancellation.
$\phi_i':A_i' - A_{i+1}\to A_i' - A_{i+1}$ is defined as follows.  If
$a\in -A_{i+1}$ and
the entry in row $n-i+1$ is $-\lambda$ or $\pm b_j$ for some $j$, then
$\phi_i'(-a)=-a'\in -A_{i+1}$ where $a'$ is obtained from $a$ by
interchanging and negating rows 0 and $n-i+1$.  
(Remember that the matrices
are indexed from 0 to $n$.)  Otherwise,
$\phi_i'(a)=-a$ (in $-A_{i+1}$ if $a\in A_i$, and vice versa).

An example may help to clarify the method.  
In the $n=2$ case, $A_0'$ and $A_1$ are
the sets of arrays in which $\lambda$ occurs in $M_0'$ and $M_1$ respectively: 
\[
M_0'=\left[\begin{array}{ccc}
\lambda-b_0 & -b_1 & -b_2 \\
-b_0& B-b_1 & -b_2 \\
-b_0 &-b_1&B-b_2
\end{array}\right]\text{ and}
\]
\[
M_1=\left[\begin{array}{ccc}
\lambda-b_0 & -b_1 & -b_2 \\
-b_0& B-b_1 & -b_2 \\
-\lambda+b_0-b_0 &b_1-b_1&b_2+B-b_2
\end{array}\right].
\]
In the easier situation, where the array does not change, we have:
\[\phi_0'\left(
           \left[\begin{smallmatrix}
           \lambda &  &  \\
                 & B_0 & \\
                  &  & B_1
           \end{smallmatrix}\right]\right)= -\left[\begin{smallmatrix}
                                 \lambda &  &  \\
                                        & B_0 & \\
                                         &  & B_1
                                 \end{smallmatrix}\right]\in -A_1.
\]
Here, we started with an element of $A_0'$ and ended with an element of $-A_1$;
the two arrays look identical other than the negative sign outside.
Meanwhile, in the more confusing case:
\[\phi_0'\left(-
           \left[\begin{smallmatrix}
           \lambda &  &  \\
                 & B_0 & \\
                  &  & b_2
           \end{smallmatrix}\right]\right)= -\left[\begin{smallmatrix}
                                        &  & -b_2 \\
                                        & B_0 & \\
                                 -\lambda &  &
                                 \end{smallmatrix}\right].
\]
Note that in this example, both arrays appear in the set $-A_1$ but 
do not exist in $A_0'$, 
and the actual sign of $\phi_0'(-a)$ is different from that of 
$-a$.  We have switched the rows in which two of the entries appeared,
changing their signs but
leaving them in their original columns.  This is always the procedure
for $\phi_i'$.  Another possibility is:
\[
\phi_0'\left(-\left[\begin{smallmatrix} \lambda & & \\ & B_0 & \\ & & -b_2
\end{smallmatrix}\right]\right)=\left[\begin{smallmatrix}\lambda & & \\
& B_0 & \\ & & -b_2 \end{smallmatrix}\right]\in A_0'.
\]
In this example, we started with an element of $-A_1$ and $\phi_0'$ 
returned an element
of $A_0'$; $\phi_0'$ is an involution because if we apply it twice we get
back the same element we started with.  The involution is sign-reversing
because interchanging two rows of a matrix changes the sign of the determinant
and changing the signs of two rows has no effect.

\subsubsection{Defining $\phi_i$ for $1\leq i\leq n$}
Since $M_i'$ is obtained from $M_i$ by arithmetic within entries of the matrix,
the rest of the involutions for $1\leq i \leq n$
are of the form $\phi_i:A_i-A_i'\to A_i-A_i'$.
If $a\in A_i$ and the entry in the $(n-i+1)^{\rm th}$ row is $\pm b_j$,
then $\phi_i(a)=a'\in A_i$ where $a'$ is obtained from
$a$ by changing the sign of the entry in the $(n-i+1)^{\rm th}$ row.
Otherwise, $\phi_i(a)=-a\in A_i'$.
If $-a\in -A_i'$, then $\phi_i(-a)=a\in A_i$.
Returning to the $n=2$ example,
\[
M_1'=\left[\begin{array}{ccc}
           \lambda-b_0 & -b_1 & -b_2 \\
                  -b_0 & B-b_1 &-b_2 \\
           -\lambda    &  0    & B
           \end{array}\right].
\]
So we have
\[\phi_1\left(
           \left[\begin{smallmatrix}
           \lambda &  &  \\
                 & B_0 & \\
                  &  & -b_2
           \end{smallmatrix}\right]\right)= \left[\begin{smallmatrix}
                                \lambda &  &  \\
                                        & B_0 & \\
                                         &  & +b_2
                                \end{smallmatrix}\right]\in A_1,
\]
and
\[\phi_1\left(-
           \left[\begin{smallmatrix}
                 & -b_1 & \\
                 &  & -b_2\\
           -\lambda &  &
           \end{smallmatrix}\right]\right)= \left[\begin{smallmatrix}
                                  & -b_1 &  \\
                                        &  & -b_2\\
                                -\lambda &  & 
                                 \end{smallmatrix}\right]\in A_1.
\]
Note that in the first of these two examples, $\phi_1(a)$ and $a$ had
opposite signs but were both elements of $A_1$, whereas in the second
example, $-a\in -A_1'$ and $\phi_1(-a)\in A_1$.  This is clearly an
involution, since there is only one row of $M_i'$ in which entries appear
twice with opposite signs.  

\subsubsection{Defining $\phi_n'$ and $\phi_{n+1}$}
The last two involutions are a little bit different.
$\phi_n':A_n'-A_{n+1}\to A_n'-A_{n+1}$ takes an array in the matrix set $A_n'$ 
and matches it with the product of its non-$\lambda$ 
entries in the order of their columns
(with the sign the determinant would assign this term), and it takes signed 
monomials to the location of the corresponding array.  There is
always a $\lambda$ in the zeroth column.

For example, in
\[
M_2'=\begin{bmatrix}
    \lambda-b_0 & -b_1 & -b_2\\
    -\lambda & B & 0\\
    -\lambda & 0 & B
    \end{bmatrix},
\]
we have
\[
\phi_2'\left(\left[\begin{smallmatrix} 
                    & -b_1 & \\
                -\lambda & & \\
                    & & B_0 \end{smallmatrix}\right]\right)=-b_1 B_0,
\]
\[
\phi_2'(-B_0 b_2)=\left[\begin{smallmatrix}
                    &  & -b_2\\
                    & B_0 & \\
              -\lambda & & \end{smallmatrix}\right],
\]
and
\[
\phi_2'\left(\left[\begin{smallmatrix}
                   \lambda &  & \\
                         & B_0 & \\
                         & & B_0 \end{smallmatrix}\right]\right)=B_0 B_0.
\]
Again, this is an involution because it matches elements of $A_n'$ (the set 
of arrays) with monomials, in perfect pairs.

The final involution, $\phi_{n+1}:A_{n+1}-A_{n+1}'\to A_{n+1}-A_{n+1}'$, takes
any positive element of $A_{n+1}$
and matches it to another
monomial, obtained according to the following formula:
\[\phi_{n+1}\left(\prod_{k=1}^n B_{j_k}\right)=\begin{cases}
  -b_0 \displaystyle{\prod_{k=2}^n B_{j_k}}\in -A_{n+1}' &\text{if $j_1=0$},\\ 
   -B_{j_{j_1}}
     \displaystyle{\left(\prod_{k=2}^{j_1-1} B_{j_k}\right)} b_{j_1}
      \displaystyle{\left(\prod_{k=j_1+1}^n B_{j_k}\right)}\in A_{n+1}
        & \text{otherwise.}\end{cases}
\]
$\phi_{n+1}$ applied to any element of $-A_{n+1}'$ gives the same monomial,
only with the initial $-b_0$ changed to a positive $B_0$, in $A_{n+1}$.  If
we start with a negative element of $A_{n+1}$, it must have exactly one 
$b_j$ in the
$j^{th}$ position for some $j$.  
When we apply $\phi_{n+1}$, we make this $b_j$ upper-case and
switch it with the indeterminate in the first position, and change the sign.  This
is clearly a sign-reversing involution.

The ugliness of the formula belies the simplicity of the process.  A few 
examples with $n=6$ should help.
\[
\phi_7\left(B_3 B_4 B_6 B_0 B_2 B_0\right)=-B_6 B_4 b_3 B_0 B_2 B_0\in A_7.
\]
All we have done is toggle the capitalization of $B_3$ (in the first position 
of the product) and switch this new lower-case entry 
with the element in the third (its subscript) position (which is $B_6$).
The easiest possible case is:
\[
\phi_7\left(B_0 B_1 B_6 B_2 B_4 B_2\right)=-b_0 B_1 B_6 B_2 B_4 B_2\in -A_7'.
\]
More often some switching is involved, as in the first case and the next one:
\[
\phi_7\left(-B_4 B_3 B_0 B_3 b_5 B_1\right)=B_5 B_3 B_0 B_3 B_4 B_1\in A_7.
\]
(Remember, if there is a lower-case $b_j$ in the product, we switch it with
the first element of the product.)

These involutions are a key ingredient in the creation of the Happy Code.

\section{Garsia and Milne's Involution Principle}

Garsia and Milne \cite{G-M} found an extremely useful method while 
investigating bijective proofs for the Rogers-Ramanujan identities.

\begin{defn} A pseudo-sign-reversing involution is an involution
on a signed set, with the property that any point that is not fixed is sent to
a point with the opposite sign.  
\end{defn}

\begin{lemma}[Scholium: The Involution Principle \cite{G-M}]
Let $A$ be a finite signed set, $A=A^+ - A^-$, with pseudo-sign-reversing 
involutions
$\phi$ and $\psi$ whose fixed-point sets are $F(\phi)$ and $F(\psi)$ 
respectively.  Then there is a (fixed-point-free) sign-reversing
involution $\gamma$ on the set $F(\phi)-F(\psi)$.  Furthermore, 
$\gamma$ can be constructed using the following algorithm:
\end{lemma}
\begin{quote}
\begin{tabbing}
{\bf begin}\= \\
\>{\bf if} $\phi(x)$\=$\,=x$ {\bf then} \\
\> \>$y\leftarrow x$\\
\> \>{\bf repeat} \= \\
\> \> \>$z\leftarrow\psi(y)$\\
\> \> \>$y\leftarrow\phi(z)$\\
\> \>{\bf until} $\phi(z)\,=\,z$ {\bf or} $\psi(y)\,=\,y$\\
\> \>{\bf if} $\phi(z)\,=\,z$ {\bf then}\\
\> \> \>$\gamma(x)\leftarrow z$\\
\> \>{\bf else}\\
\> \> \>$\gamma(x)\leftarrow y$\\
\>{\bf else if} $\psi(x)\,=x$ {\bf then}\\
\> \>$y\leftarrow x$\\
\> \>{\bf repeat}\\
\> \> \>$z\leftarrow \phi(y)$\\
\> \> \>$y\leftarrow \psi(z)$\\
\> \>{\bf until} $\psi(z)\,=\,z$ {\bf or} $\phi(y)\,=\,y$\\
\> \>{\bf if} $\psi(z)\,=\,z$ {\bf then}\\
\> \> \>$\gamma(x)\leftarrow z$\\
\> \>{\bf else}\\
\> \> \>$\gamma(x)\leftarrow y$\\
\>{\bf else}\\
\> \>\{$x$ is not a fixed point of $\phi$ or $\psi$\}\\
{\bf end.}\footnotemark
\end{tabbing}
\end{quote}
\footnotetext{Pseudo-code quoted from \cite{S-W}, pages 141-142}
This Lemma is extremely important because it not only establishes the
existence of the involution $\gamma$ but actually shows how to construct
it.

\begin{comment}
In order to use the algorithm, we first need to construct two pseudo-sign-
reversing involutions on one large set.  For example, we could begin with
the set $A_i-A_i'+A_i'-A_{i+1}$ for each $i$.  We define the involution 
$\overline{-I_{A_i'}}$, which is the negative identity map on $-A_i'
+A_i'$ and the identity map on $A_i-A_{i+1}$.  
\[
\overline{-I_{A_i'}}\left(a\right)=\begin{cases} 
                                -a & \text{if $a\in -A_i'+A_i'$},\\
                                 a & \text{if $a\in A_i-A_{i+1}$.}\end{cases}
\]
Our second involution is 
\[
\left(\phi_i+\phi_i'\right)\left(a\right)=\begin{cases}
                \phi_i\left(a\right) & \text{if $a\in A_i-A_i'$},\\
                \phi_i'\left(a\right) &\text{if $a\in A_i'-A_{i+1}$}.
\end{cases}
\]
$F(\phi_i+\phi_i')=\emptyset$ since the involution has no fixed points, and
$F(\overline{-I_{A_i'}})=A_i-A_{i+1}$.  
The involution principle gives us an algorithm
for finding a sign-reversing involution on the difference of
these two sets, which is $A_i-A_{i+1}$.  Thus our lemma allows us 
to ``eliminate'' all of the ``primed'' sets except the last one
(the sets are still there, and we will have to pass through 
them as we use the involution principle).  
Then we can use the Involution Principle to construct an involution on 
$A_0-A_1+A_1-A_2$, in exactly the same way, and so on.  

More formally:
\end{comment}

%here we are
\begin{lemma}[The Bread Lemma]
Given two sign-reversing involutions, $\phi:A-B\to
A-B$ and $\psi:B-C\to B-C$, there is a sign-reversing 
involution on $A-C$.  
\end{lemma}\label{bread}
\begin{proof}
Let $\overline{-I_B}$ represent the negative identity map on $B-B$, extended
to be the identity on $A-C$.  
\[
\overline{-I_B}\left(x\right)=\begin{cases}
                                -x & \text{if $x\in -B+B$},\\
                                 x & \text{if $x\in A-C$.}\end{cases}
\]
Let $\phi + \psi:A-B+B-C\to A-B+B-C$ be defined as follows:
\[
\left(\phi +\psi\right)\left(x\right) = \begin{cases} \phi\left(x\right) &
\text{if $x\in A-B$}\\ \psi\left(x\right) & \text{if $x\in B-C$}\end{cases}.
\]
Then both $\overline{-I_B}$ and $(\phi+\psi)$ are pseudo-sign-reversing
involutions, and 
$F(\overline{-I_B})=A-C$ and $F(\phi + \psi)=\emptyset$.  The algorithm
of the Involution Principle provides a sign-reversing
involution on $F(\phi +\psi)-F(\overline{-I_B})=A-C-\emptyset=A-C$.  
\end{proof}

We call it the Bread Lemma because it can be visualized as a process to
remove all of the insides from a $B$ sandwich, leaving the diner with
only a couple of slices of bread (the sets $A$ and $C$).

\begin{lemma}\label{general}  
Given any sequence of signed sets $S_0,S_1,\dots,S_{k+1}$, 
where $S_0$ and $S_{k+1}$ contain only positive elements, and
sign-reversing involutions $\beta_0,\dots,\beta_k$ where $\beta_i$ acts 
on $S_i-S_{i+1}$, there is a constructible bijection 
between $S_0$ and $S_{k+1}$.
\end{lemma}
\begin{proof}
By repeated applications of The Bread Lemma (Lemma \ref{bread}), we can
``eliminate'' all of the in-between sets as follows. 
\begin{comment}
First, let $\phi$ in
the Bread Lemma be $\phi_i$ for each $i$, and let $\psi$ in the Bread Lemma
be $\phi_i'$ for each $i$, $0\leq i\leq n$.  
Then $A$ is $A_i$ for each $i$, $B$ is $A_i'$, and $C$ is $A_{i+1}$. 
The Bread Lemma gives us involutions on $A_i-A_{i+1}$ that still 
satisfy the hypotheses, and we can now apply the Bread Lemma
again.   Let $A=A_0$, $B=A_1$, and $C=A_2$.  
\end{comment}
Let $A=S_0$, $B=S_1$, $C=S_2$, $\phi=\beta_0$,
and $\psi=\beta_1$.  The Bread Lemma constructs a
sign-reversing involution on $S_0-S_2$, and this involution still
satisfies the hypotheses of the Bread Lemma.  Now let $B=S_2$, $C=S_3$, etc.
We keep sandwiching in until
we arrive at $A=S_0$, $B=S_k$, and $C=S_{k+1}$, where $\phi$ is 
the involution achieved by so many applications of the Bread Lemma and 
$\psi$ is $\beta_{k+1}$.  One more application, and we have a
sign-reversing involution on $S_0-S_{k+1}$.  However,
$\left(S_0-S_{k+1}\right)^+=S_0$ and $\left(S_0-S_{k+1}\right)^-=-S_{k+1}$ 
since these two sets contained only positive elements.
Hence the only way for the
involution to be sign-reversing is for each element of $S_0$ to be mapped
to an element of $S_{k+1}$.
Thus, we have found a bijection between $S_0$ and $S_{k+1}$.  Note that we have
not simply
proven the existence of a bijection, but actually provided an algorithm for 
constructing it.
\end{proof}
\begin{theorem}
Given the sets $A_0,A_0',A_1,\dots,A_{n+1}'$ 
and the sign-reversing involutions 
$\phi_0,\phi_0',\dots,\phi_n,\phi_n',\phi_{n+1}$ 
defined above, there is a constructible
bijection between $A_0$ (the set of trees) and $A_{n+1}'$ (the set of codes).
\end{theorem}
\begin{proof}
The sets $A_0,\dots,A_{n+1}'$ and the involutions $\phi_0,\dots,\phi_{n+1}$
satisfy the hypotheses of Lemma \ref{general}.  Thus we can construct the
bijection between the set of trees and the set of codes.
\end{proof}
\section{An example}

Consider the case $n=2$.  We will apply the theorem to find the code that
corresponds to the tree $1\to 2\to 0 \in A_0$.

First we apply $\phi_0$ to get an element of $-A_0'$:
\[
\phi_0\left(1\to 2\to 0\right)=-\left[\begin{smallmatrix}
                         \lambda & & \\
                          & B_2 & \\
                          & & B_0 \end{smallmatrix}\right]\in -A_0'
\]
Next we apply $\overline{-I_{A_0'}}$:
\[
\overline{-I_{A_0'}}\left(-\left[\begin{smallmatrix}
                         \lambda & & \\
                          & B_2 & \\
                          & & B_0 \end{smallmatrix}\right]\right)=
                    \left[\begin{smallmatrix}
                         \lambda & & \\
                          & B_2 & \\
                          & & B_0 \end{smallmatrix}\right]\in A_0'
\]
We alternate between $\phi$s and $\overline{-I}$s.  Each application of a
$\overline{-I}$ merely changes the sign of the element (and of the subset it
lies in):
\[
\overline{-I_{A_1}}\circ\phi_0'\left(-\left[\begin{smallmatrix}
                         \lambda & & \\
                          & B_2 & \\
                          & & B_0 \end{smallmatrix}\right]\right)=
                    \left[\begin{smallmatrix}
                         \lambda & & \\
                          & B_2 & \\
                          & & B_0 \end{smallmatrix}\right]\in A_1
\]
\[
\overline{-I_{A_1'}}\circ\phi_1\left(\left[\begin{smallmatrix}
                         \lambda & & \\
                         & B_2 & \\
                         & & B_0 \end{smallmatrix}\right]\right)=
                    \left[\begin{smallmatrix}
                         \lambda & & \\
                          & B_2 & \\
                          & & B_0 \end{smallmatrix}\right]\in A_1'
\]
\[
\overline{-I_{A_2}}\circ\phi_1'\left(\left[\begin{smallmatrix}
                         \lambda & & \\
                          & B_2 & \\
                          & & B_0 \end{smallmatrix}\right]\right)=
                    \left[\begin{smallmatrix}
                         \lambda & & \\
                          & B_2 & \\
                          & & B_0 \end{smallmatrix}\right]\in A_2
\]
\[
\overline{-I_{A_2'}}\circ\phi_2\left(\left[\begin{smallmatrix}
                         \lambda & & \\
                          & B_2 & \\
                          & & B_0 \end{smallmatrix}\right]\right)=
                    \left[\begin{smallmatrix}
                         \lambda & & \\
                          & B_2 & \\
                          & & B_0 \end{smallmatrix}\right]\in A_2'.
\]
Since $n=2$, we are in the last matrix set.  
\[
\overline{-I_{A_3}}\circ
\phi_2'\left(\left[\begin{smallmatrix}
                         \lambda & & \\
                          & B_2 & \\
                          & & B_0 \end{smallmatrix}\right]\right)=
B_2 B_0 \in A_3
\]
This is the exciting part!
\[
\phi_3\left(B_2 B_0\right)=-B_0 b_2 \in A_3
\]
\[
\phi_2'\circ\overline{-I_{A_3}}\left(-B_0 b_2\right)=\left[\begin{smallmatrix}
 & & -b_2 \\ & B_0 & \\ -\lambda & & 
\end{smallmatrix}\right]\in A_2'
\]
Now there is nothing to stop us from passing through several sets in a row
on our way back up the sequence of sets
via the following involutions:
\[
\phi_2\circ\overline{-I_{A_2'}}\left(\left[\begin{smallmatrix}
 & & -b_2 \\ & B_0 & \\ -\lambda & &
\end{smallmatrix}\right]\right)=\left[\begin{smallmatrix}
 & & -b_2 \\ & B_0 & \\ -\lambda & &
\end{smallmatrix}\right]
\in A_2
\]

\[
\phi_1'\circ\overline{-I_{A_2}}\left(\left[\begin{smallmatrix}
 & & -b_2 \\ & B_0 & \\ -\lambda & &
\end{smallmatrix}\right]\right)=
\left[\begin{smallmatrix}
 & & -b_2 \\ & B_0 & \\ -\lambda & &
\end{smallmatrix}\right]\in A_1'
\]

\[
\phi_1\circ\overline{-I_{A_1'}}\left(\left[\begin{smallmatrix}
 & & -b_2 \\ & B_0 & \\ -\lambda & &
\end{smallmatrix}\right]\right)=
\left[\begin{smallmatrix}
 & & -b_2 \\ & B_0 & \\ -\lambda & &
\end{smallmatrix}\right]\in A_1
\]
At this point we apply $\phi_0'\circ\overline{-I_{A_1}}$.  
$\overline{-I_{A_1}}$ takes us to the set $-A_1$, and in this case 
an application of $\phi_0'$ maps to another element of $-A_1$:
\[
\phi_0'\circ\overline{-I_{A_1}}\left(\left[\begin{smallmatrix}
 & & -b_2 \\ & B_0 & \\ -\lambda & &
\end{smallmatrix}\right]\right)=-\left[\begin{smallmatrix}
\lambda & & \\ & B_0 & \\ & & b_2
\end{smallmatrix}\right]\in -A_1
\]
\[
\overline{-I_{A_1}}\left(-\left[\begin{smallmatrix}
\lambda & & \\ & B_0 & \\ & & b_2
\end{smallmatrix}\right]\right)=\left[\begin{smallmatrix}
\lambda & & \\ & B_0 & \\ & & b_2
\end{smallmatrix}\right]\in A_1
\]
\[
\phi_1\left(\left[\begin{smallmatrix}
\lambda & & \\ & B_0 & \\ & & b_2
\end{smallmatrix}\right]\right)=\left[\begin{smallmatrix}
\lambda & & \\ & B_0 & \\ & & -b_2
\end{smallmatrix}\right]\in A_1
\]
\[
\overline{-I_{A_0'}}\circ\phi_0'\circ\overline{-I_{A_1}}\left(
\left[\begin{smallmatrix}
\lambda & & \\ & B_0 & \\ & & -b_2
\end{smallmatrix}\right]\right)=-\left[\begin{smallmatrix}
\lambda & & \\ & B_0 & \\ & & -b_2
\end{smallmatrix}\right]\in -A_0'
\]
At this point it is the Matrix Tree Theorem that comes to the rescue, in
the form of $\phi_0$:
\[
\phi_0\left(-\left[\begin{smallmatrix}
\lambda & & \\ & B_0 & \\ & & -b_2
\end{smallmatrix}\right]\right)=-\left[\begin{smallmatrix}
\lambda & & \\ & B_0 & \\ & & B_2
\end{smallmatrix}\right]\in -A_0'
\]
The involutions now take us directly down the sequence of matrices to
the last one.
\[
\overline{-I_{A_1'}}\circ\phi_1\circ\overline{-I_{A_1}}\circ\phi_0'\circ
\overline{-I_{A_0'}}\left(-\left[\begin{smallmatrix}
\lambda & & \\ & B_0 & \\ & & B_2
\end{smallmatrix}\right]\right)=
    \left[\begin{smallmatrix}
   \lambda & & \\ & B_0 & \\ & & B_2
   \end{smallmatrix}\right]\in A_1'
\]
\[
\overline{-I_{A_2'}}\circ\phi_2\circ\overline{-I_{A_2}}\circ\phi_1'
\left(\left[\begin{smallmatrix}
   \lambda & & \\ & B_0 & \\ & & B_2
   \end{smallmatrix}\right]\right)=
    \left[\begin{smallmatrix}
   \lambda & & \\ & B_0 & \\ & & B_2
   \end{smallmatrix}\right]\in A_2'
\]
Coming down the home stretch:
\[
\overline{-I_{A_3}}\circ\phi_2'\left(\left[\begin{smallmatrix}
   \lambda & & \\ & B_0 & \\ & & B_2
   \end{smallmatrix}\right]\right)=B_0 B_2 \in A_3
\]
And finally:
\[
\phi_3\left(B_0 B_2\right)=-b_0 B_2 \in -A_3'.
\]
Thus, the Happy Code for the tree $1\to 2\to 0$ is $B_2$.

To find the tree for a code, we can easily follow the involutions through
backwards, undoing the whole process.  In this sense, the Happy Code is more
natural (hence ``happier'') than the Pr\"ufer Code.

Computationally, finding the Happy Code for a tree is a slow process.  
However, later we will see a method for calculating the Happy Code that does 
not resort to matrices but works directly with the tree.

\chapter{The Blob Code}
Another code results from a different sequence of sets and involutions, but 
still using the Involution Principle and the Bread Lemma.  
We begin with the $n\times n$ submatrix from the Matrix Tree Theorem (obtained
by crossing out the zeroth row and column):
\[
C_0'=\begin{bmatrix}
B-b_1 & -b_2 & \hdots & -b_n \\
-b_1 & B-b_2 & \hdots & -b_n \\
. & . & \hdots & . \\
. & . & \hdots & . \\
. & . & \hdots & . \\
-b_1 & -b_2 & \hdots & B-b_n
\end{bmatrix}
\]
The Blob Code is related to the process of alternately 
performing row operations and column
operations on adjacent rows and columns as follows.
The first step is to subtract row $n-1$ from row $n$ (without cancellation).
\[
R_1=\begin{bmatrix}
B-b_1 & -b_2 & \hdots & -b_{n-1} & -b_n \\
-b_1 & B-b_2 & \hdots & -b_{n-1} & -b_n \\
. & . & \hdots & . & . \\
. & . & \hdots & . & . \\
. & . & \hdots & . & . \\
-b_1 & -b_2 & \hdots & B-b_{n-1} & -b_n \\
-b_1+b_1 & -b_2+b_2 & \hdots & -b_{n-1}-B+b_{n-1} & B-b_n+b_n
\end{bmatrix}
\]
Now we perform arithmetic within entries, but only in row $n$:
\[
R_1'=\begin{bmatrix}
B-b_1 & -b_2 & \hdots & -b_{n-1} & -b_n \\
-b_1 & B-b_2 & \hdots & -b_{n-1} & -b_n \\
. & . & \hdots & . & . \\
. & . & \hdots & . & . \\
. & . & \hdots & . & . \\
-b_1 & -b_2 & \hdots & B-b_{n-1} & -b_n \\
0 & 0 & \hdots & -B & B
\end{bmatrix}
\]
Then we add column $n$ to column $n-1$.  
\[
C_1=\begin{bmatrix}
B-b_1 & -b_2 & \hdots & -b_{n-1}-b_n & -b_n \\
-b_1 & B-b_2 & \hdots & -b_{n-1}-b_n & -b_n \\
. & . & \hdots & . & . \\
. & . & \hdots & . & . \\
. & . & \hdots & . & . \\
-b_1 & -b_2 & \hdots & B-b_{n-1}-b_n & -b_n \\
0 & 0 & \hdots & -B+B & B
\end{bmatrix}
\]
And once again, perform arithmetic within entries in row $n$
({\em not} column $n-1$):
\[
C_1'=\begin{bmatrix}
B-b_1 & -b_2 & \hdots & -b_{n-1}-b_n & -b_n \\
-b_1 & B-b_2 & \hdots & -b_{n-1}-b_n & -b_n \\
. & . & \hdots & . & . \\
. & . & \hdots & . & . \\
. & . & \hdots & . & . \\
-b_1 & -b_2 & \hdots & B-b_{n-1}-b_n & -b_n \\
0 & 0 & \hdots & 0 & B
\end{bmatrix}
\]
All of that was the first step.
We work our way up the matrix this way:  at the $i^{\rm th}$ step we
first subtract row $n-i$ from row $n-i+1$, 
then add column $n-i+1$ to column $n-i$, cancelling only within row $n-i+1$,
until the matrix consists of $B$ on the diagonals and $0$ elsewhere, except 
in the first row.  At the end of the $i^{\rm th}$ 
step the $(n-i)^{\rm th}$ diagonal
entry  consists of $B-\sum b_j$ where the sum is over 
$n-i\leq j\leq n$. After the
last column operation, we set $B_j=b_j$ so that the diagonal entry in row 1
consists only of $b_0$.
At the end of the whole process, the first row consists of 
$b_0$ in its first entry and a bunch of garbage
in the other entries, but the rest of the matrix is just $B$ on the diagonal.
The last 2 matrices are:
\[
C_{n-1}=\begin{bmatrix}
B-b_1-\displaystyle{\sum_{k=2}^n} b_k & -\displaystyle{\sum_{k=2}^n} b_k & 
\displaystyle{\sum_{k=3}^n} b_k& \hdots & -b_{n-1}-b_n & -b_n \\
-B+B & B & 0 &\hdots & 0 & 0 \\
. & . & . &\hdots & . & . \\
. & . & . &\hdots & . & . \\
. & . & . &\hdots & . & . \\
0 & 0 & . &\hdots & B & 0 \\
0 & 0 & 0 &\hdots & 0 & B
\end{bmatrix}
\]
and
\[
C_{n-1}'=\begin{bmatrix}
b_0 & -b_2-b_3-\dots-b_n & \hdots & -b_{n-1}-b_n & -b_n \\
0 & \sum_{j=0}^n b_j & \hdots & 0 & 0 \\
. & . & \hdots & . & . \\
. & . & \hdots & . & . \\
. & . & \hdots & . & . \\
0 & 0 & \hdots &  \sum_{j=0}^n b_j & 0 \\
0 & 0 & \hdots & 0 &  \sum_{j=0}^n b_j
\end{bmatrix}.
\]
This matrix clearly has determinant equal to $b_0 B^{n-1}$.

\section{Orlin's ideas}\label{oops}
In \cite{O}, Orlin introduced the idea of identifying two vertices of a graph.
We explain how this notion is used with the matrices in the 
construction of the Blob Code.
Assume we have a weighted directed graph on vertices $0$ through $n$.  
Loops are allowed.  We assume there are no multiple edges, because multiple
edges can be subsumed into the weights.
The weight of the
edge from $i$ to $j$ is $a_{ij}$.  %Let $\sigma_i=\sum_{j=0}^n a_{ij}=$
%weighted outdegree of $i$.
\begin{defn}In a directed graph $D$, two vertices $i$ and $j$ are
identifiable when $a_{ik}=a_{jk}$ for all $k$.
\end{defn}
Note that this definition includes $k=i$ and $k=j$; 
so if there is an edge from $i$ to $j$ then there needs to be a loop on $j$.

Identifiability is an equivalence relation, so there is some sense in which
we can think of two identifiable vertices as being redundant (their outgoing
edges have the same heads).  
\begin{defn}
If we ``identify'' two identifiable vertices $i$ and
$j$ to a generalized vertex, called \verb+blob+, 
and eliminate one set of the duplicate edges, 
we end up with a new digraph in which there are $a_{ik}$($=a_{jk}$) 
edges $\verb+blob+\rightarrow k$ for all $k\neq i,j$, 
and there are $a_{ki}+a_{kj}$
edges $k\rightarrow\verb+blob+$.  
There are also $a_{ij}+a_{ji}$ loops on the \verb+blob+.
\end{defn}
We take full blame for the naming of the \verb+blob+.
We differ from Orlin in our visualization of this process.  He considered
this ``\verb+blob+'' 
to be a new vertex; we think of it as {\em containing} the 
original two vertices being identified.  Each incoming edge actually points
not at the \verb+blob+ as a whole but rather at its original terminal vertex
within the \verb+blob+.

Here, we set our edge weights to $W(i\to j)=b_j$ for all edges.
and look at an example: 

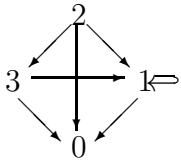
\begin{picture}(100,65)
\put (50,0){0}
\put(25,25){3}
\put(75,25){1}
\put(50,50){2}
\put(30,22){\vector(1,-1){16}}
\put(75,22){\vector(-1,-1){16}}
\put(35,30){\vector(1,0){35}}
\put(50,50){\vector(-1,-1){16}}
\put(52,50){\vector(0,-1){40}}
\put(56,50){\vector(1,-1){16}}
\put(78,25){$\hookleftarrow$}
\put(82,30.5){\line(1,0){7}}
\end{picture}

\noindent
Vertices $1$ and $3$ are identifiable:  each has exactly one edge to $0$
and one edge to $1$.  If we identify the two, we obtain the following graph
(with weights $w(2\to 0)=b_0$, $w(2\to 1)=b_1$, $w(2\to 3)=b_3$, 
$w(\verb+blob+\to 0)=b_0$, and $w(\verb+blob+\to 1)=b_1$):

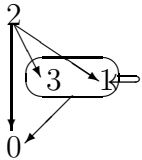
\begin{picture}(120,65)
\put (50,0){0}
\put(75,30){\oval(35,15)}
\put(50,50){2}
\put(65,25){3}
\put(85,25){1}
\put(75,23){\vector(-1,-1){18}}
\put(53,50){\vector(1,-2){10}}
\put(52,50){\vector(0,-1){40}}
\put(53,50){\vector(3,-2){32}}
\put(88,25){$\hookleftarrow$}
\put(92,30.5){\line(1,0){7}}
\end{picture}

\noindent
In this graph, 2 and \verb+blob+ are not identifiable because 2 has 
an edge to 3, while there is no loop $\verb+blob+\to 3$.

In the complete digraph with loops, 
all vertices are identifiable.  This was why we altered the matrix to allow
for loops.
Orlin used this idea
to manipulate formulas to get the formula $(n+1)^{n-1}$ for the number of 
trees. 
We examine the 
relationship between this idea and the matrix method.

For the moment we illustrate the process with $n=3$.
We begin with the complete directed graph.  The $4\times 4$ matrix 
corresponding to it, where the edges are weighted by indeterminates
indexed by the terminal vertex, is
\[
\hat{D}-\hat{A}=\Upsilon_0=\begin{bmatrix}
B-b_0 & -b_1 & -b_2 & -b_3\\
-b_0 & B-b_1 & -b_2 & -b_3\\
-b_0 & -b_1 & B-b_2 & -b_3\\
-b_0 & -b_1 & -b_2 & B-b_3
\end{bmatrix}.
\]
Once we identify vertices $2$ and $3$, the graph looks like this (omitting
edges whose initial vertex is 0, since they never appear in a tree and we
will not be identifying vertex 0 with any of the others).

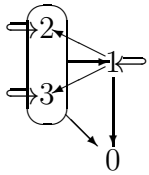
\begin{picture}(150,75)
\put(50,0){0}
\put(25,25){3}
\put(12,25){$\hookrightarrow$}
\put(14,30.5){\line(1,0){6.5}}
\put(50,37){1}
\put(53,37){$\hookleftarrow$}
\put(57,42.5){\line(1,0){7}}
\put(25,50){2}
\put(12,50){$\hookrightarrow$}
\put(14,55.5){\line(1,0){6.5}}
\put(28,40){\oval(15,45)}      %blob containing 2 & 3
\put(35,21){\vector(1,-1){12}} %edge blob->0
\put(53,35){\vector(0,-1){26}}%edge 1->0
\put(50,43){\vector(-2,1){20}} %edge 1->2
\put(50,39){\vector(-2,-1){20}} %edge 1->3
\put(35.5,41){\vector(1,0){16}}  %edge blob->1
\end{picture}

\noindent
The corresponding matrix is
\[
\Upsilon_1=\begin{bmatrix}
B-b_0 & -b_1 & -b_2-b_3\\
-b_0 & B-b_1 & -b_2-b_3\\
-b_0 & -b_1 & B-b_2-b_3
\end{bmatrix}.
\]
The proper way to think of this is that there are two relevant
rows (the zeroth row, representing edges from $0$, is not relevant);
the ``oneth'' row represents edges out of $1$ and the second represents edges
out of \verb+blob+. The zeroth column represents edges into $0$; the 
``oneth'' column (not including diagonal entries) 
represents edges into $1$, and the second represents edges into \verb+blob+.
An edge $1\to\verb+blob+$ can be either $1\to 2$ or $1\to 3$.
If we were to cancel terms in the last row, we would have only $b_0 + b_1$ 
in the diagonal entry.  The positive and negative copies of $b_2$ and $b_3$
represent the loops $\verb+blob+\to 2$ and $\verb+blob+\to 3$ respectively.

Once we have identified $1$ with \verb+blob+ (which we can do because 
both have edges to 0, 1, 2, and 3), 
the graph is (with undrawn edges from 0):

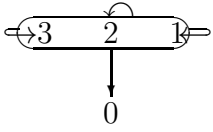
\begin{picture}(80,50)
\put(50,0){0}
\put(25,30){3}
\put(12,30){$\hookrightarrow$}
\put(14,35.5){\line(1,0){4}}
\put(50,30){2}
\put(50,40){$\curvearrowleft$}
\put(75,30){1}
\put(78,30){$\hookleftarrow$}
\put(82.5,35.5){\line(1,0){6.5}}
\put(50,34){\oval(65,12)}
\put(53,28){\vector(0,-1){18}}
\end{picture}

\noindent
The matrix corresponding to this is
\[
\Upsilon_2=\begin{bmatrix}
B-b_0 & -b_1-b_2-b_3\\
-b_0 & B-b_1-b_2-b_3\end{bmatrix} 
=
\begin{bmatrix}
B-b_0 & -b_1-b_2-b_3\\
-b_0 & b_0\end{bmatrix}. 
\]
We will be crossing out the zeroth row and column.  Thus, only one
entry appears in the part of the matrix in whose determinant we are interested.
This is true 
because now there is only one vertex besides $0$ and it only has one 
non-loop edge.

The determinants of $\Upsilon_0,\Upsilon_1,$ and $\Upsilon_2$ are related
as follows:
$b_0 B^2=\det(\Upsilon_0)=\det(\Upsilon_1)\times B=\det(\Upsilon_2)\times B^2$.
\section{The sets}

Much as we did with the Happy Code, we use the matrices in the definition
of a sequence of signed sets, but now we insert some of Orlin's ideas as well.
The sets are
$G_0,G_0',S_1,S_1',T_1,T_1',G_1,\dots,T_{n-1}',G_{n-1}, G_{n-1}',S_n$.

In this sequence, the set $G_0$ is the set of trees,
and $G_0'$ is the matrix set of arrays defined by $C_0'$.
There are more matrices than we had for the Happy Code, and extra sets in 
between.  For $1\leq i\leq n-1$,
$G_i$ is the set of ordered pairs $(\tau,\gamma)$ where $\tau$ is a 
spanning tree (rooted at $0$) on a directed graph $D_i$ (described below) 
and $\gamma$ is an
ordered $i$-tuple of $b_j$'s.  
$D_i$ is defined to be the complete digraph with 
$n-i$ vertices, where vertex $n-i$ is actually \verb+blob+ which contains
$n-i, n-i+1, \dots, n$.  The labels in \verb+blob+ are terminal vertices to
edges, but they all share the same outgoing edges; in any tree, 
\verb+blob+ has only one outgoing edge.

For $1\leq i\leq n$, $S_i$ and $S_i'$ are the sets of arrays from $R_i$ and 
$R_i'$ respectively, and $T_i$ denotes the
set of arrays from $C_i$. Finally, we use both $T_i'$ 
{\em and} $G_i'$ 
to denote the set of arrays from $C_i'$, for $0\leq i\leq n-1$. Arrays are
signed, as they were in the Happy Code.
The final set is $S_n=\{b_0\}\times B^{n-1}$ where,
in the set notation, $B$ is understood to stand for the set 
$B=\{b_0,\hdots,b_n\}$
and $B^{n-1}$ stands for the $(n-1)$-fold direct product
$B\times B\times \hdots \times B$.

As an example we list the sets for the case $n=3$.  Matrices are thought of 
as sets of arrays.  (In the graphs, edges with initial vertex 
$0$ have been omitted from the pictures, since they never appear in
a spanning tree and the zeroth row of the matrix has already been ignored.)

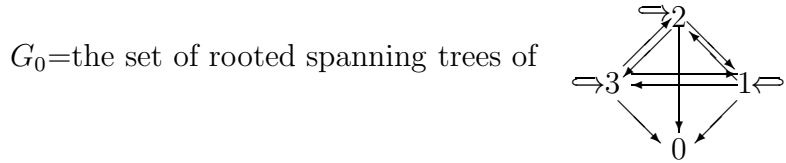
\begin{picture}(300,70)
\put(0,35){$G_0$=the set of rooted spanning trees of}
\put(250,0){0}
\put(225,25){3}
\put(212,25){$\hookrightarrow$}
\put(214,30.5){\line(1,0){7}}
\put(275,25){1}
\put(280,25){$\hookleftarrow$}
\put(284,30.5){\line(1,0){7}}
\put(250,50){2}
\put(237,52){$\hookrightarrow$}
\put(239,57.5){\line(1,0){7}}
\put(230,22){\vector(1,-1){16}} %edge 3->0
\put(275,22){\vector(-1,-1){16}}%edge 1->0
\put(253,50){\vector(0,-1){40}} %edge 2->0
\put(232,35){\vector(1,1){18}}  %edge 3->2
\put(256,52){\vector(1,-1){18}} %edge 2->1
\put(250,49){\vector(-1,-1){18}}%edge 2->3
\put(275,30){\vector(-1,1){18}} %edge 1->2
\put(275,28){\vector(-1,0){40}} %edge 1->3
\put(235,32){\vector(1,0){40}}  %edge 3->1
\end{picture}
\[
G_0' \leftrightarrow
C_0'=\begin{bmatrix} 
B-b_1 & -b_2 & -b_3\\ -b_1 & B-b_2 & -b_3\\ 
-b_1 & -b_2 & B-b_3\end{bmatrix}
\]
\[
S_1\leftrightarrow
R_1=\begin{bmatrix} 
B-b_1 & -b_2 & -b_3\\ -b_1 & B-b_2 & -b_3\\ 
-b_1+b_1 & -b_2-B+b_2& B-b_3+b_3\end{bmatrix}
\]
\[
S_1'\leftrightarrow
R_1'=\begin{bmatrix} 
B-b_1 & -b_2 & -b_3\\ -b_1 & B-b_2 & -b_3\\ 
0 & -B & B\end{bmatrix}
\]
\[
T_1\leftrightarrow
C_1=\begin{bmatrix} 
B-b_1 & -b_2-b_3 & -b_3\\ 
-b_1 & B-b_2-b_3 & -b_3\\ 
0 & -B+B & B\end{bmatrix}
\]
\[
T_1'\leftrightarrow
C_1'=\begin{bmatrix} 
B-b_1 & -b_2-b_3 & -b_3\\ 
-b_1 & B-b_2-b_3 & -b_3\\ 
0 & 0 & B\end{bmatrix}
\]

\begin{picture}(330,75)
\put(0,35){$G_1=$ the set of rooted spanning trees of}
\put(250,0){0}
\put(225,25){3}
\put(212,25){$\hookrightarrow$}
\put(214,30.5){\line(1,0){6.5}}
\put(250,37){1}
\put(253,37){$\hookleftarrow$}
\put(257,42.5){\line(1,0){7}}
\put(225,50){2}
\put(212,50){$\hookrightarrow$}
\put(214,55.5){\line(1,0){6.5}}
\put(228,40){\oval(15,45)}      %blob containing 2 & 3
\put(235,21){\vector(1,-1){12}} %edge blob->0
\put(253,35){\vector(0,-1){26}}%edge 1->0
\put(250,43){\vector(-2,1){20}} %edge 1->2
\put(250,39){\vector(-2,-1){20}} %edge 1->3
\put(235.5,41){\vector(1,0){16}}  %edge blob->1
\put(270,35){$\times B$}
\end{picture}
\[
G_1'\leftrightarrow
C_1'=\begin{bmatrix} 
B-b_1 & -b_2-b_3 & -b_3\\ 
-b_1 & B-b_2-b_3 & -b_3\\ 
0 & 0 & B\end{bmatrix}
\]
\[
S_2\leftrightarrow
R_2=\begin{bmatrix} 
B-b_1 & -b_2-b_3 & -b_3\\ 
-b_1-B+b_1 & B-b_2-b_3+b_2+b_3 & -b_3+b_3\\ 
0 & 0 & B\end{bmatrix}
\]
\[
S_2'\leftrightarrow
R_2'=\begin{bmatrix} 
B-b_1 & -b_2-b_3 & -b_3\\ -B & B & 0\\ 0 & 0 & B
\end{bmatrix}
\]
\[
T_2\leftrightarrow
C_2=\begin{bmatrix} 
B-b_1-b_2-b_3 & -b_2-b_3 & -b_3\\ -B+B & B & 0
\\ 0 & 0 & B
\end{bmatrix}
\]
\[
T_2'\leftrightarrow
C_2'=\begin{bmatrix} b_0 & -b_2-b_3 & -b_3\\ 0 & 
\displaystyle{\sum_{j=0}^3} b_j & 0\\ 0 & 0 & \displaystyle{\sum_{j=0}^3} b_j
\end{bmatrix}
\]

\begin{picture}(330,50)
\put(0,20){$G_2=$ the set of rooted spanning trees of}
\put(250,0){0}
\put(225,30){3}
\put(212,30){$\hookrightarrow$}
\put(214,35.5){\line(1,0){4}}
\put(250,30){2}
\put(250,40){$\curvearrowleft$}
\put(275,30){1}
\put(278,30){$\hookleftarrow$}
\put(282.5,35.5){\line(1,0){6.5}}
\put(250,34){\oval(65,12)}
\put(253,28){\vector(0,-1){18}}
\put(290,20){$\times B^2$}
\end{picture}
\[
G_2'\leftrightarrow
C_2'=\begin{bmatrix} b_0 & -b_2-b_3 & -b_3\\ 0 & 
\displaystyle{\sum_{j=0}^3} b_j & 0\\ 0 & 0 & \displaystyle{\sum_{j=0}^3} b_j
\end{bmatrix}
\]
The final set is $S_3=\{b_0\}\times B\times B$, where $B=\{b_0,b_1,\dots,b_n\}$
by abuse of notation.

\section{The involutions}
Some of the involutions are 
defined similarly to the involutions we used for the Happy
Code, but there are many more of them.
\subsubsection{Defining $\mu_0'$}  
$\mu_0':G_0-G_0'\to G_0-G_0'$ maps a tree to an array from
$-G_0'$ by taking the $b_j$ in the $i^{\rm{th}}$ diagonal entry for each 
edge $i\to
j$.  The remaining elements of $-C_0'$ are matched in pairs (by toggling
the diagonality of the cycle with the largest element) according to the
bijective proof of the Matrix Tree Theorem, 
just as they were for the Happy Code.
\subsubsection{Defining $\rho_i$ for $1\leq i\leq n-1$}
For $1\leq i\leq n-1$,
$\rho_i:G_{i-1}'-S_i\to G_{i-1}'-S_i$ maps 
corresponding arrays in $G_{i-1}'$ and $-S_i$
to each other, and then takes the extra elements of $-S_i$ and matches them
up according to the row operation that took $C_{i-1}'$ to $R_i$.
If $-a\in -S_i$ and the entry in row $n-i+1$ is $+b_j$ or $-B_j$, 
then  $\rho_i(-a)=-a'\in -S_i$ where $a'$ is
obtained from $a$ by interchanging and negating rows $n-i$ and $n-i+1$.
Otherwise, $\rho_i(a)=-a$ (in $G_{i-1}'$ if $a\in -S_i$ and vice versa).
Consider what happens if we begin with an element of $G_1'$:
\[
\rho_2\left(\left[\begin{smallmatrix} B_0 & & \\ & B_1 &\\ & & B_2
\end{smallmatrix}\right]\right)=-
\left[\begin{smallmatrix} B_0 & & \\ & B_1 &\\ & & B_2
\end{smallmatrix}\right]\in -S_2.
\]
If we start with an element of $G_1'$, there is always a corresponding 
element of $-S_2$: the same array
but with a negative sign outside it.  We could also start with an element
of $-S_2$:
\[
\rho_2\left(-\left[\begin{smallmatrix} B_0 & & \\ & B_0 &\\ & & B_1
\end{smallmatrix}\right]\right)=
\left[\begin{smallmatrix} B_0 & & \\ & B_0 &\\ & & B_1
\end{smallmatrix}\right]\in G_1'.
\]
In that example, there was a corresponding element of $G_1'$.  Sometimes,
there isn't:
\[
\rho_2\left(-\left[\begin{smallmatrix} & -b_2 & \\ -B_3 & &\\ & & B_0
\end{smallmatrix}\right]\right)=-
\left[\begin{smallmatrix} B_3 & & \\ & +b_2 &\\ & & B_0
\end{smallmatrix}\right]\in -S_2.
\]
Here we switched the rows (and signs!) of the entries in the 
first and second rows
without changing the columns of these entries.  Note that the resulting
element of $-S_2$ does not have a corresponding element in $G_1'$ either
(because the $b_2$ on that diagonal is not the one from B).
It is clear that $\rho_i$ is a sign-reversing involution.

\subsubsection{Defining $\rho_i'$ for $1\leq i\leq n-1$}
$\rho_i':S_i-S_i'\to S_i-S_i'$, for
$1\leq i\leq n-1$, is defined as follows:  If $a\in S_i$ and
the entry in row $n-i+1$ is $\pm b_j$, then $\rho_i'(a)=a'\in S_i$
where $a'$ is obtained by changing the sign of the entry in row $n-i+1$
of $a$ and all other entries remain unchanged.
Otherwise, $\rho_i'(a)=-a$ (in $-S_i'$ if $a\in S_i$ and vice versa).
For example, if $n=3$ and we start with an element of $S_1$,
\[
\rho_1'\left(\left[\begin{smallmatrix}  & & -b_3\\ & B_0 & \\ -b_1 & &
\end{smallmatrix}\right]\right)=\left[\begin{smallmatrix} 
& & -b_3\\ & B_0 & \\ +b_1 & &
\end{smallmatrix}\right]\in S_1.
\]
Other elements of $S_1$ get mapped to elements of $-S_1'$ (and all
elements of $-S_1'$ get mapped to elements of $S_1$):
\[
\rho_1'\left(\left[\begin{smallmatrix} B_2 & & \\ & B_0 & \\ & & B_0
\end{smallmatrix}\right]\right)=-\left[\begin{smallmatrix} 
B_2 & & \\ & B_0 & \\ & & B_0
\end{smallmatrix}\right]\in -S_1'
\]
and
\[
\rho_1'\left(\left[\begin{smallmatrix} B_2 & & \\ & & -b_3\\ & -B_0 &
\end{smallmatrix}\right]\right)=-\left[\begin{smallmatrix} 
B_2 & & \\ & & -b_3 \\ & -B_0 &
\end{smallmatrix}\right]\in -S_1'.
\]

\subsubsection{Defining $\kappa_i$ for $1\leq i\leq n-1$}
For $1\leq i\leq n-1$, $\kappa_i:S_i'-T_i\to S_i'-T_i$ works 
similarly to $\rho_i$.  The difference is that now the column operation
is the key.  If we begin with an element of $S_1'$, we get an element of
$-T_1$:
\[
\kappa_1\left(\left[\begin{smallmatrix} B_2 & & \\ & B_0 &\\ &  & B_2
\end{smallmatrix}\right]\right)=-
\left[\begin{smallmatrix} B_2 & & \\ & B_0 &\\ & & B_2
\end{smallmatrix}\right]\in -T_1.
\]
In fact, $\kappa_i$ applied to an element of $S_i'$ always gives the
corresponding element of $T_i$:  the same array, but with a negative
sign.  The reverse sometimes happens if we apply $\kappa_i$ to an 
element of $T_i$:
\[
\kappa_1\left(-\left[\begin{smallmatrix} B_0 & & \\ & B_3 &\\ &  & B_1
\end{smallmatrix}\right]\right)=
\left[\begin{smallmatrix} B_0 & & \\ & B_3 &\\ & & B_1
\end{smallmatrix}\right]\in S_1'.
\]
However, if there is no corresponding element in $S_i'$, we switch the
columns of the entries in columns $n-i+1$ and $n-i$ (but not the signs
this time, since the column operations were addition).  For example,
\[
\kappa_1\left(-\left[\begin{smallmatrix}  & -b_3 & \\ -b_1 & &\\ &  & B_0
\end{smallmatrix}\right]\right)=-
\left[\begin{smallmatrix} & & -b_3 \\ -b_1 & &\\ & B_0 &
\end{smallmatrix}\right]\in -T_1.
\]
Here we switched the entries in columns $2$ and $3$, leaving them in their
original rows.
In general, we have:  If $a\in -T_i$ and the entry in column $n-i$ is 
$B_j$ in row $n-i+1$ or $-b_j$ where $j\geq n-i+1$, then 
$\kappa_i(a)=a'\in -T_i$ where $a'$ is obtained from $a$ by interchanging
columns $n-i$ and $n-i+1$.  For all other $a$, $\kappa_i(a)=-a$ (in $S_i$
if $a\in -T_i$ and vice versa).

\subsubsection{Defining $\kappa_i'$ for $1\leq i\leq n-2$}
$\kappa_i':T_i-T_i'\to T_i-T_i'$ works similarly to $\rho_i'$ for 
$1\leq i\leq n-1$.  If $a\in T_i$ and the entry in row $n-i+1$ is in
column $n-i$, then $\kappa_i'(a)=a'\in T_i$ where $a'$ is obtained from
$a$ by changing the sign of the entry in the $(n-i+1,n-i)$ position.
Otherwise, $\kappa_i(a)=-a$ (in $T_i$ if $a\in -T_i'$ and vice versa).
For example, starting with an element of $T_1$:
\[
\kappa_1'\left(\left[\begin{smallmatrix} B_0 & & \\ & B_2 &\\ & & B_2
\end{smallmatrix}\right]\right)=-
\left[\begin{smallmatrix} B_0 & & \\ & B_2 &\\ & & B_2
\end{smallmatrix}\right]\in -T_1'.
\]
If we start with an element of $T_1$, there are two possibilities:
either there is a corresponding array in $T_1'$ as above, or else the
array cancels via arithmetic within an entry, as in the next example:
\[
\kappa_1'\left(\left[\begin{smallmatrix}  &  & -b_3\\ -b_1 & &\\ & B_0 &
\end{smallmatrix}\right]\right)=
\left[\begin{smallmatrix} & & -b_3  \\ -b_1 & &\\ & -B_0 &
\end{smallmatrix}\right]\in T_1.
\]

\subsubsection{Defining $\kappa_{n-1}'$}
$\kappa_{n-1}':T_{n-1}-T_{n-1}'\to T_{n-1}-T_{n-1}'$ is essentially
the same as the previous $\kappa_i'$s, except that now we set $B_j=b_j$ 
and cancel in row 1.
If $a\in T_{n-1}$ and the entry in column 1 is anything other than $B_0$ in
the upper-left corner of the matrix, then $\kappa_{n-1}'(a)=a'\in T_{n-1}$
where $a'$ is obtained from $a$ by changing the sign of the entry in column
1 and making all $B_j$ lower-case.  
For all other $a$, $\kappa_{n-1}'(a)=-a'\in T_{n-1}'$ obtained by leaving
all entries the same but making $B_j$ lower-case.
\[
\kappa_2'\left(\left[\begin{smallmatrix} B_3 & &\\&B_0&\\&& B_2
\end{smallmatrix}\right]\right)=\left[\begin{smallmatrix} -b_3 &&\\&b_0&\\
&&b_2\end{smallmatrix}\right]\in T_2,
\]
\[
\kappa_2'\left(\left[\begin{smallmatrix} -b_2 & & \\ & B_0 &\\ & & B_0
\end{smallmatrix}\right]\right)=
\left[\begin{smallmatrix} b_2 & & \\ & b_0 &\\ & & b_0
\end{smallmatrix}\right]\in T_2, \text{ and}
\]
\[
\kappa_2'\left(\left[\begin{smallmatrix} 
& -b_2 &\\-B_0&&\\&&B_0
\end{smallmatrix}\right]\right)=\left[\begin{smallmatrix}&-b_2&\\
b_0&&\\&&b_0\end{smallmatrix}\right]\in T_2.
\]

\subsubsection{Defining $\mu_i$ and $\mu_i'$ for $1\leq i\leq n-1$}
$\mu_i:T_i'-G_i\to T_i'-G_i$ reads the entries of the matrix and translates
them into the digraph.  The upper-left $(n-i)\times(n-i)$ corner represents
(by the Matrix Tree Theorem) the spanning trees of $D_i$ (in fact, these
submatrices are obtained from the matrices $\Upsilon_i$ from the example in 
\S \ref{oops} by crossing out the zeroth row and column).
When $\mu_i$ is applied to an element $x$ in $T_i'$, one of two things happens.
Case 1:  if edges are drawn from $k$ to $j$ for each $b_j$ appearing in
row $k\leq n-i$ in $x$ (remembering that it is okay for $j$ to be inside 
\verb+blob+), and the resulting graph is a tree, then $\mu_i(x)$ is the
pair whose first element is that tree, and whose second element is the 
$i$-tuple found by reading down the diagonal starting at row $n-i+1$.
Case 2:  if those edges do not form a tree, then there is at least one 
cycle, and 
by moving the cycle with the largest element onto or off of the diagonal 
(according to where it 
already was), we find the element of $T_i'$ that is $\mu_i(x)$. 
(Actually, there can only be one cycle, so we don't have to worry about
which cycle to move).  It is clear that for elements $x\in T_i'$ such that
$\mu_i(x)\in T_i'$, $\mu_i$ acts as an involution. 
Define $\mu_i((\tau,\gamma))$ to be the element of $T_i'$ found by 
putting $b_j$ in the diagonal entry in row $k$ whenever there is an edge
in the tree $k\to j$, and filling in the rest of the diagonal entries 
from left to right by taking them from the code.  Then it is clear that 
$\mu_i$ acts as an involution in the rest of the cases too.

If $\mu_i(x)\in -G_i$ (in other words, it is a $-(\tau, \gamma)$ pair), 
then step $i$ of the overall procedure is finished.
For example, 
\[
\mu_1\left(\left[\begin{smallmatrix}
B_2 &&\\&B_0&\\&&B_3
\end{smallmatrix}\right]\right)=
-\left(\begin{picture}(55,20)
\put(0,0){1}
\put(5,3){\vector(1,0){20}}
\put(28,10){\oval(9,23)}
\put(25,0){2}
\put(25,12){3}
\put(33,5){\vector(1,0){15}}
\put(50,0){0}
\end{picture},(b_3)\right)\in -G_2,
\]
which signifies that step 1 is finished.  On the other hand,
\[
\mu_1\left(\left[\begin{smallmatrix}
B_3&&\\&B_1&\\&&B_0
\end{smallmatrix}\right]\right)=
\left[\begin{smallmatrix}
&-b_3&\\-b_1&&\\&&B_0
\end{smallmatrix}\right],
\]
indicating that we will have to apply several more involutions before
step 1 is done.

$\mu_i':G_i-G_i'\to G_i-G_i'$ is essentially the negative of $\mu_i$, 
since the sets $T_i'$ and $G_i'$ are identical.  Moving from $G_i$
into $G_i'$ is the beginning of the $(i+1)^{\rm th}$ step.

\subsubsection{Defining $\rho_n$}
The final involution, $\rho_n:G_{n-1}'-S_n\to G_{n-1}'-S_n$ matches arrays
from $C_{n-1}'$ to negative monomials that consist of the
entries in order from left to right, similarly to $\phi_n'$ in the 
Happy Code.  
\[
\rho_n\left(\left[\begin{smallmatrix}b_0 &&&\\&b_{k_2}&&\\&&\ddots&\\&&&b_{k_n}
\end{smallmatrix}\right]\right)=b_0 b_{k_2}\dots b_{k_n}.
\]
For example,
\[
\rho_3\left(\left[\begin{smallmatrix}b_0 & &\\ & b_3 &\\ & & b_0
\end{smallmatrix}\right]\right)=- b_0 b_3 b_0,
\]
\[
\rho_3\left(\left[\begin{smallmatrix}b_0 & &\\ & b_2 &\\ & & b_2
\end{smallmatrix}\right]\right)=- b_0 b_2^2,\text{ and}
\]
\[
\rho_3\left(-b_0 b_1 b_3\right)=
\left[\begin{smallmatrix}b_0 & &\\ & b_1 &\\ & & b_3
\end{smallmatrix}\right].
\]
In fact, since all elements of $G_{n-1}'$ are positive, and so are all 
elements of $S_n$, $\rho_n$ is a simple bijection between the elements 
of $G_{n-1}'$ and the elements of $-S_n$.

\section{How to Find the Blob Code}

We use these involutions the same way we did for the Happy Code.
\begin{theorem}
Given the sets $G_0,G_0',S_1,S_1',T_1,T_1',G_1,\dots,G_{n-1}',S_n$ 
and the sign-reversing involutions
$\mu_0',\rho_1,\rho_1',\kappa_1,\kappa_1',\mu_1,\mu_1',\dots,\kappa_{n-1}',
\mu_{n-1},\mu_{n-1}',\rho_n$,
there is a bijection between $G_0$ (the set of trees) and $S_n$ (the set 
of codes).
\end{theorem}
\begin{proof}
\begin{comment}
Again, we will apply the Bread Lemma repeatedly.  First we consider, for
each $i$, the sets $A=S_i$, $B=S_i'$, and $C=T_i$.  Then $\phi=\rho_i'$
and $\psi=\kappa_i$.  The Bread Lemma gives us involutions on $S_i-T_i$ (having
``eliminated'' $S_i'$ for all $i$).  
We also consider the sets $A=T_i$, $B=T_i'$, and
$C=G_i$.  Then $\phi=\kappa_i'$ and $\psi=\mu_i$, so using the
Bread Lemma gives us involutions on $T_i-G_i$.  
Also, we can let $A=G_i$, $B=G_i'$, and $C=S_{i+1}$, so that $\phi=\mu_i'$
and $\psi=\rho_{i+1}$.  Then the Bread Lemma gives us involutions on 
$G_i-S_{i+1}$.  Having eliminated all of
the ``primed'' sets, we continue sandwiching  in since the involutions we
find this way still satisfy the hypotheses of the Bread Lemma, until we
have found a fixed-point-free, sign-reversing involution on the set
$T_0-S_n$.  Since all elements of $T_0$ and $S_n$ are positive, the only
possible involution is one that matches each tree to a code.  Furthermore,
since we used the Involution Principle, we have an algorithm that allows
us to find the code belonging to any given tree (or vice versa).
\end{comment}
The sets and involutions satisfy the conditions of Lemma \ref{general}.
Thus we can construct the bijection between the set of trees and the set of
codes.
\end{proof}

\subsection{An example}
To clarify the method, we
use the matrix method to construct the Blob Code for the tree
$2\to 1\to 3\to 0\in G_0$.  It will help to remember that elements in
each signed set can only have one of the involutions of types 
$\rho_i,\rho_i',\kappa_i,\kappa_i',\mu_i,\mu_i'$ applied to them,
and we always alternate between negative identity maps and our defined
involutions 
($\mu$, $\kappa$, and $\rho$, with indices and with or without primes).

\begin{center}
\begin{tabular}{c|c}
Involution & Acts on\\
\hline
$\rho_i$ & $G_{i-1}'-S_i$\\
$\rho_i'$ & $S_i-S_i'$\\
$\kappa_i$ & $S_i'-T_i$\\
$\kappa_i'$ & $T_i-T_i'$\\
$\mu_i$ & $T_i'-G_i$\\
$\mu_i'$ & $G_i-G_i'$
\end{tabular}
\end{center}

This particular example is a sort of ``worst case scenario'' for a
small graph, but for larger $n$ such a long process
would be more likely.  Coding trees
with no inversions is much easier, as is coding any tree that satisfies the 
condition that for every 
$i$ whose path to $0$ goes through some $j>i$, it also holds that 
$\verb+succ+(i)>i$.

\subsubsection{Step 1}
First we apply $\mu_0'$ to get an element of $-G_0'$:
\[
\mu_0'\left(2\to 1\to 3\to 0\right)=-\left[\begin{smallmatrix}
B_3 & &\\ & B_1 &\\ & & B_0
\end{smallmatrix}\right]\in -G_0'.
\]
Next we apply $\overline{I_{G_0'}}$:
\[
\overline{-I_{G_0'}}\left(-\left[\begin{smallmatrix}
B_3 & &\\ & B_1 &\\ & & B_0
\end{smallmatrix}\right]\right)=\left[\begin{smallmatrix}
B_3 & &\\ & B_1 &\\ & & B_0
\end{smallmatrix}\right]\in G_0'
\]
As in the example for the Happy Code, we alternate between the involutions
we defined, and the negative identity maps.  
\[
\overline{-I_{S_1}}\circ\rho_1\left(\left[\begin{smallmatrix}
B_3 & &\\ & B_1 &\\ & & B_0
\end{smallmatrix}\right]\right)=\left[\begin{smallmatrix}
B_3 & &\\ & B_1 &\\ & & B_0
\end{smallmatrix}\right]\in S_1
\]
\[
\overline{-I_{S_1'}}\circ\rho_1'\left(\left[\begin{smallmatrix}
B_3 & &\\ & B_1 &\\ & & B_0
\end{smallmatrix}\right]\right)=\left[\begin{smallmatrix}
B_3 & &\\ & B_1 &\\ & & B_0
\end{smallmatrix}\right]\in S_1'
\]
\[
\overline{-I_{T_1}}\circ\kappa_1\left(\left[\begin{smallmatrix}
B_3 & &\\ & B_1 &\\ & & B_0
\end{smallmatrix}\right]\right)=\left[\begin{smallmatrix}
B_3 & &\\ & B_1 &\\ & & B_0
\end{smallmatrix}\right]\in T_1
\]
\[
\overline{-I_{T_1'}}\circ\kappa_1'\left(\left[\begin{smallmatrix}
B_3 & &\\ & B_1 &\\ & & B_0
\end{smallmatrix}\right]\right)=\left[\begin{smallmatrix}
B_3 & &\\ & B_1 &\\ & & B_0
\end{smallmatrix}\right]\in T_1'
\]
Here is the first time the involutions do anything interesting:
\[
\overline{-I_{T_1'}}\circ\mu_1\left(\left[\begin{smallmatrix}
B_3 & &\\ & B_1 &\\ & & B_0
\end{smallmatrix}\right]\right)=-\left[\begin{smallmatrix}
&-b_3& \\ -b_1 &&\\ & & B_0
\end{smallmatrix}\right]\in -T_1'
\]
Since $\mu_1$ doesn't move us out of $T_1'$, the negative identity map
results in a move to $-T_1'$.  Note that any time we are in a negative
set, we are moving ``up'' the sequence of matrices (or stalled where
we are).  Because the array
did not correspond to a tree in the graph where $2$ and $3$ are identified,
the basic effect of $\mu_1$ at that step was 
to find the array with off-diagonal entries that cancels it in the matrix.  
From $-T_1'$, we apply $\kappa_i'$.
\[
\overline{-I_{T_1}}\circ\kappa_1'\left(-\left[\begin{smallmatrix}
&-b_3 &\\ -b_1 &&\\ & & B_0
\end{smallmatrix}\right]\right)=-\left[\begin{smallmatrix}
&-b_3& \\ -b_1 &&\\ & & B_0
\end{smallmatrix}\right]\in -T_1
\]
The next few involutions have the effect of switching columns:
\[
\overline{-I_{T_1}}\circ\kappa_1\left(-\left[\begin{smallmatrix}
&-b_3& \\ -b_1 &&\\ & & B_0
\end{smallmatrix}\right]\right)=\left[\begin{smallmatrix}
&&-b_3\\ -b_1 &&\\ & B_0&
\end{smallmatrix}\right]\in T_1
\]
\[
\overline{-I_{T_1}}\circ\kappa_1'\left(\left[\begin{smallmatrix}
&&-b_3\\ -b_1 &&\\ & B_0&
\end{smallmatrix}\right]\right)=-\left[\begin{smallmatrix}
&&-b_3\\ -b_1 &&\\ & -B_0&
\end{smallmatrix}\right]\in -T_1
\]
\[
\overline{-I_{S_1'}}\circ\kappa_1\left(-\left[\begin{smallmatrix}
&&-b_3\\ -b_1 &&\\ & -B_0&
\end{smallmatrix}\right]\right)=-\left[\begin{smallmatrix}
&&-b_3\\ -b_1 &&\\ & -B_0&
\end{smallmatrix}\right]\in -S_1'
\]
\[
\overline{-I_{S_1}}\circ\rho_1'\left(-\left[\begin{smallmatrix}
&&-b_3\\ -b_1 &&\\ & -B_0&
\end{smallmatrix}\right]\right)=-\left[\begin{smallmatrix}
&&-b_3\\ -b_1 &&\\ & -B_0&
\end{smallmatrix}\right]\in -S_1
\]
And switching rows:
\[
\overline{-I_{S_1}}\circ\rho_1\left(-\left[\begin{smallmatrix}
&&-b_3\\ -b_1 &&\\ & -B_0&
\end{smallmatrix}\right]\right)=\left[\begin{smallmatrix}
&&-b_3\\ &B_0&\\ b_1&&
\end{smallmatrix}\right]\in S_1
\]
\[
\overline{-I_{S_1}}\circ\rho_1\left(\left[\begin{smallmatrix}
&&-b_3\\ &B_0&\\ b_1&&
\end{smallmatrix}\right]\right)=-\left[\begin{smallmatrix}
&&-b_3\\ &B_0&\\ -b_1&&
\end{smallmatrix}\right]\in -S_1
\]
\[
\overline{-I_{G_0'}}\circ\rho_1\left(-\left[\begin{smallmatrix}
&&-b_3\\ &B_0&\\ -b_1&&
\end{smallmatrix}\right]\right)=-\left[\begin{smallmatrix}
&&-b_3\\ &B_0&\\ -b_1&&
\end{smallmatrix}\right]\in -G_0'
\]
We have defined the involutions in such a way that there is no passing
the set $G_i$ when moving up; we apply the Matrix Tree Theorem again
(the effect, in this case, of $\mu_0'$):
\[
\overline{-I_{G_0'}}\circ\mu_0'\left(-\left[\begin{smallmatrix}
&&-b_3\\ &B_0&\\ -b_1&&
\end{smallmatrix}\right]\right)=\left[\begin{smallmatrix}
B_3 & &\\ & B_0 &\\ & & B_1
\end{smallmatrix}\right]\in G_0'
\]
And from here on, it's easy for the rest of the step:
\[
\overline{-I_{S_1}}\circ\rho_1\left(\left[\begin{smallmatrix}
B_3 & &\\ & B_0 &\\ & & B_1
\end{smallmatrix}\right]\right)=\left[\begin{smallmatrix}
B_3 & &\\ & B_0 &\\ & & B_1
\end{smallmatrix}\right]\in S_1
\]
\[
\overline{-I_{S_1'}}\circ\rho_1'\left(\left[\begin{smallmatrix}
B_3 & &\\ & B_0 &\\ & & B_1
\end{smallmatrix}\right]\right)=\left[\begin{smallmatrix}
B_3 & &\\ & B_0 &\\ & & B_1
\end{smallmatrix}\right]\in S_1'
\]
\[
\overline{-I_{T_1}}\circ\kappa_1\left(\left[\begin{smallmatrix}
B_3 & &\\ & B_0 &\\ & & B_1
\end{smallmatrix}\right]\right)=\left[\begin{smallmatrix}
B_3 & &\\ & B_0 &\\ & & B_1
\end{smallmatrix}\right]\in T_1
\]
\[
\overline{-I_{T_1'}}\circ\kappa_1'\left(\left[\begin{smallmatrix}
B_3 & &\\ & B_0 &\\ & & B_1
\end{smallmatrix}\right]\right)=\left[\begin{smallmatrix}
B_3 & &\\ & B_0 &\\ & & B_1
\end{smallmatrix}\right]\in T_1'
\]
\[
\overline{-I_{G_1}}\circ\mu_1\left(\left[\begin{smallmatrix}
B_3 & &\\ & B_0 &\\ & & B_1
\end{smallmatrix}\right]\right)=\left(\begin{picture}(55,20)
\put(0,0){1}
\put(5,3){\vector(1,0){20}}
\put(28,10){\oval(9,21)}
\put(25,0){3}
\put(25,12){2}
\put(33,5){\vector(1,0){15}}
\put(50,0){0}
\end{picture},(1)\right)\in G_1
\]
Since we've gotten to $G_1$ and have a tree and a partial code, we are
done with this step. 

%whew!
\subsubsection{Step 2}
Starting where we left off,
\[
\overline{-I_{G_1'}}\circ\mu_1'\left(\begin{picture}(55,20)
\put(0,0){1}
\put(5,3){\vector(1,0){20}}
\put(28,10){\oval(9,21)}
\put(25,0){3}
\put(25,12){2}
\put(33,5){\vector(1,0){15}}
\put(50,0){0}
\end{picture},(1)\right)=
\left[\begin{smallmatrix}
B_3 & &\\ & B_0 &\\ & & B_1
\end{smallmatrix}\right]\in G_1'
\]
\[
\overline{-I_{S_2}}\circ\rho_2\left(\left[\begin{smallmatrix}
B_3 & &\\ & B_0 &\\ & & B_1
\end{smallmatrix}\right]\right)=\left[\begin{smallmatrix}
B_3 & &\\ & B_0 &\\ & & B_1
\end{smallmatrix}\right]\in S_2
\]
\[
\overline{-I_{S_2'}}\circ\rho_2'\left(\left[\begin{smallmatrix}
B_3 & &\\ & B_0 &\\ & & B_1
\end{smallmatrix}\right]\right)=\left[\begin{smallmatrix}
B_3 & &\\ & B_0 &\\ & & B_1
\end{smallmatrix}\right]\in S_2'
\]
\[
\overline{-I_{T_2}}\circ\kappa_2\left(\left[\begin{smallmatrix}
B_3 & &\\ & B_0 &\\ & & B_1
\end{smallmatrix}\right]\right)=\left[\begin{smallmatrix}
B_3 & &\\ & B_0 &\\ & & B_1
\end{smallmatrix}\right]\in T_2
\]
Now is the first time in Step 2 that we cannot move on to the next set, because
for each of the above applications of involutions 
there was a corresponding element in the next
set.  The elements of $T_2$ can only be acted on by $\kappa_2'$.
\[
\kappa_2'\left(\left[\begin{smallmatrix}
B_3 & &\\ & B_0 &\\ & & B_1
\end{smallmatrix}\right]\right)=\left[\begin{smallmatrix}
-b_3 & &\\ & B_0 &\\ & & B_1
\end{smallmatrix}\right]\in T_2.
\]
\[
\overline{-I_{T_2}}\left(\left[\begin{smallmatrix}
-b_3 & &\\ & B_0 &\\ & & B_1
\end{smallmatrix}\right]\right)=-\left[\begin{smallmatrix}
-b_3 & &\\ & B_0 &\\ & & B_1
\end{smallmatrix}\right]\in -T_2.
\]
From $-T_2$, the involution $\kappa_2$ will either take us to $S_2$ or 
else leave us in $-T_2$ (in this case, the latter):
\[
\kappa_2\left(-\left[\begin{smallmatrix}
-b_3 & &\\ & B_0 &\\ & & B_1
\end{smallmatrix}\right]\right)=-\left[\begin{smallmatrix}
& -b_3 &\\ B_0 & &\\ & & B_1
\end{smallmatrix}\right]\in -T_2.
\]
Another application of a negative identity map is now required as part of the
algorithm of the Involution Principle.
\[
\overline{-I_{T_2}}\left(-\left[\begin{smallmatrix}
&-b_3 &\\ B_0 &&\\ & & B_1
\end{smallmatrix}\right]\right)=\left[\begin{smallmatrix}
& -b_3 &\\ B_0 & &\\ & & B_1
\end{smallmatrix}\right]\in T_2.
\]
Now we go back to the appropriate involution, $\kappa_2'$ in this case:
\[
\kappa_2'\left(\left[\begin{smallmatrix}
& -b_3 &\\ B_0 & &\\ & & B_1
\end{smallmatrix}\right]\right)=\left[\begin{smallmatrix}
& -b_3 &\\ -B_0 & &\\ & & B_1
\end{smallmatrix}\right]\in T_2.
\]
\[
\overline{-I_{T_2}}\left(\left[\begin{smallmatrix}
& -b_3 &\\ -B_0 & &\\ & & B_1
\end{smallmatrix}\right]\right)=-\left[\begin{smallmatrix}
& -b_3 &\\ -B_0 & &\\ & & B_1
\end{smallmatrix}\right]\in -T_2.
\]
\[
\overline{-I_{S_2'}}\circ\kappa_2\left(-\left[\begin{smallmatrix}
& -b_3 &\\ -B_0 & &\\ & & B_1
\end{smallmatrix}\right]\right)=-\left[\begin{smallmatrix}
& -b_3 &\\ -B_0 & &\\ & & B_1
\end{smallmatrix}\right]\in -S_2'.
\]
\[
\overline{-I_{S_2}}\circ\rho_2'\left(-\left[\begin{smallmatrix}
& -b_3 &\\ -B_0 & &\\ & & B_1
\end{smallmatrix}\right]\right)=-\left[\begin{smallmatrix}
& -b_3 &\\ -B_0 & &\\ & & B_1
\end{smallmatrix}\right]\in -S_2.
\]
Again we get stuck at a set.  The involution $\rho_2$ should either send us
to $T_1'$ or leave us where we are, and it is the latter that occurs.
\[
\rho_2\left(-\left[\begin{smallmatrix}
& -b_3 &\\ -B_0 & &\\ & & B_1
\end{smallmatrix}\right]\right)=-\left[\begin{smallmatrix} B_0 & &\\
& b_3 & \\ & & B_1\end{smallmatrix}\right]\in -S_2
\]
It is time for another negative identity map:
\[
\overline{-I_{S_2}}\left(-\left[\begin{smallmatrix} B_0 & &\\
& b_3 & \\ & & B_1\end{smallmatrix}\right]\right)=
\left[\begin{smallmatrix} B_0 & &\\
& b_3 & \\ & & B_1\end{smallmatrix}\right]\in S_2.
\]
Since we are back in $S_2$, we apply $\rho_2'$ followed by a negative
identity map:
\[
\overline{-I_{S_2'}}\circ\rho_2'\left(\left[\begin{smallmatrix} B_0 & &\\
& b_3 & \\ & & B_1\end{smallmatrix}\right]\right)=
-\left[\begin{smallmatrix} B_0 & &\\
& -b_3 & \\ & & B_1\end{smallmatrix}\right]\in -S_2
\]
\[
\overline{-I_{G_1'}}\circ\rho_2\left(-\left[\begin{smallmatrix}
B_0&&\\&-b_3&\\&&B_1\end{smallmatrix}\right]\right)=-\left[\begin{smallmatrix}
B_0&&\\&-b_3&\\&&B_1\end{smallmatrix}\right]\in -G_1'
\]
This array does not correspond to a tree because there is a loop 
$\verb+blob+\to 3$.  So we toggle the diagonality of the cycle.
\[
\overline{-I_{G_1'}}\circ\mu_1'\left(-\left[\begin{smallmatrix}
B_0&&\\&-b_3&\\&&B_1\end{smallmatrix}\right]\right)=\left[\begin{smallmatrix}
B_0&&\\&B_3&\\&&B_1\end{smallmatrix}\right]\in G_1'
\]
Now we are all set to go through to the end of the step:
\[
\overline{-I_{S_2}}\circ\rho_2\left(\left[\begin{smallmatrix}
B_0&&\\&B_3&\\&&B_1\end{smallmatrix}\right]\right)=\left[\begin{smallmatrix}
B_0&&\\&B_3&\\&&B_1\end{smallmatrix}\right]\in S_2
\]
\[
\overline{-I_{S_2'}}\circ\rho_2'\left(\left[\begin{smallmatrix}
B_0&&\\&B_3&\\&&B_1\end{smallmatrix}\right]\right)=\left[\begin{smallmatrix}
B_0&&\\&B_3&\\&&B_1\end{smallmatrix}\right]\in S_2'
\]
\[
\overline{-I_{T_2}}\circ\kappa_2\left(\left[\begin{smallmatrix}
B_0&&\\&B_3&\\&&B_1\end{smallmatrix}\right]\right)=\left[\begin{smallmatrix}
B_0&&\\&B_3&\\&&B_1\end{smallmatrix}\right]\in T_2
\]
And we continue:
\[
\overline{-I_{T_2'}}\circ\kappa_2'\left(\left[\begin{smallmatrix} B_0 & &\\
& B_3 & \\ & & B_1\end{smallmatrix}\right]\right)=
\left[\begin{smallmatrix} B_0 & &\\
& B_3 & \\ & & B_1\end{smallmatrix}\right]\in T_2'
\]
\[
\overline{-I_{G_2}}\circ\mu_2\left(\left[\begin{smallmatrix}
B_0 &&\\& B_3 &\\&&B_1\end{smallmatrix}\right]\right)=
\left(\begin{picture}(40,20)
\put(0,5){1}
\put(7,3){2}
\put(5,12){3}
\put(17,5){\vector(1,0){15}}
\put(35,0){0}
\put(8,10){\circle{20}}
\end{picture},(3,1)\right)\in G_2.
\]
We are almost done, because the last step is always considerably shorter.
%whew!
\subsubsection{Step 3}
From here we have
\[
\overline{-I_{G_2'}}\circ\mu_2'\left(\begin{picture}(40,20)
\put(0,5){1}
\put(7,3){2}
\put(5,12){3}
\put(17,5){\vector(1,0){15}}
\put(35,0){0}
\put(8,10){\circle{20}}
\end{picture},(3,1)\right)=
\left[\begin{smallmatrix} b_0 & &\\
& b_3 & \\ & & b_1\end{smallmatrix}\right]\in G_2', 
\]
and finally,
\[
\rho_3\left(\left[\begin{smallmatrix} b_0 & &\\
& b_3 & \\ & & b_1\end{smallmatrix}\right]\right)=-(b_0, b_3, b_1)\in -S_3.
\]
Thus, the Blob Code for the tree $2\to 1\to 3\to 0$ is (3,1).

Notice how the Blob Code differs from the Happy Code:  we are constantly
referring back to the altered graph.  It turns out we need not use matrices 
at all.

\chapter{Tree Surgery for the Blob Code}

A related algorithm for finding the Blob Code for a tree
involves progressively identifying vertices, starting at $n$ and ending with
a blob-vertex consisting of all the vertices from $1$ to $n$.  As the blob
grows, so does the code; meanwhile, the number of edges shrinks.  
The idea, as in the matrix method, 
is that if we consider our tree to be a spanning tree within the
complete directed graph (with loops), every pair of vertices is identifiable.
We keep track of the tree in the new graph that would correspond 
to our original tree.  The difference is that now we ignore the matrices.

\section{Tree Surgery Algorithm}
The algorithm takes as its input a rooted tree
(as a set of edges) whose vertices are the labels $\{0,1,\dots,n\}$.  The 
algorithm uses 
a function $\verb+path+(x)$ that finds the path (an ordered list of 
vertices) from $x$ to 0, that is,  
\[
\verb+path+(x)=(x,\verb+succ+(x),\verb+succ+(\verb+succ+(x)),\dots,0).
\]
Other procedures used are ``remove edge''
and ``add edge.''

\begin{tabbing}
Tree Surgery algorithm for the Blob Code\\
{\bf begin}\=\\
\>$\verb+blob+\leftarrow \{n\}$\\
\>$\verb+code+\leftarrow ()$\\
\>$i\leftarrow 1$\\
\>{\bf repeat}\= \\
\>\>{\bf if} \=$\verb+path+(n-i)\cap\verb+blob+\neq\emptyset$ {\bf then}\\
\>\>\>$\verb+code+\leftarrow(\verb+succ+(n-i),\verb+code+)$\\
\>\>\>remove edge $(n-i)\to\verb+succ+(n-i)$\\
\>\>\>$\verb+blob+\leftarrow\verb+blob+\cup\{n-i\}$\\
\>\>{\bf else}\\
\>\>\>$\verb+code+\leftarrow(\verb+succ+(\verb+blob+),\verb+code+)$\\
\>\>\>remove edge $\verb+blob+\to\verb+succ+(\verb+blob+)$\\
\>\>\>add edge $\verb+blob+\to\verb+succ+(n-i)$\\
\>\>\>remove edge $(n-i)\to\verb+succ+(n-i)$\\
\>\>\>$\verb+blob+\leftarrow\verb+blob+\cup\{n-i\}$\\
\>\>$i\leftarrow i+1$\\
\>{\bf until} $i=n$\\
{\bf end.}
\end{tabbing}

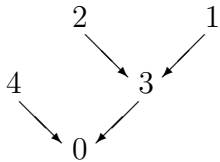
\begin{ex}

\begin{picture}(100,75)
\put (50,0){0}
\put (25,25){4}
\put(30,22){\vector(1,-1){16}}
\put(75,25){3}
\put(75,22){\vector(-1,-1){16}}
\put(50,50){2}
\put(100,50){1}
\put(55,47){\vector(1,-1){16}}
\put(100,47){\vector(-1,-1){16}}
\end{picture}

\noindent
Beginning with this tree, we create a \verb+blob+ containing a single
vertex (the one with the largest label).

\vspace{10pt}
\subsubsection{Step 1}

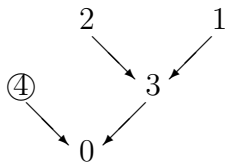
\begin{picture}(90,75)
\put (50,0){0}
\put (25,25){4}
\put(28,28){\circle{10}}
\put(30,22){\vector(1,-1){16}}
\put(75,25){3}
\put(75,22){\vector(-1,-1){16}}
\put(50,50){2}
\put(100,50){1}
\put(55,47){\vector(1,-1){16}}
\put(100,47){\vector(-1,-1){16}}
\end{picture}

\noindent
The blob contains only the vertex $4$; $n-i=3$ and $\verb+code+ =()$.  
Does the path from $3$ to $0$ go 
through the blob?  No.  So we follow the
{\bf then} instructions.  We take \verb+succ(blob)+, which is 0, 
and put it 
at the beginning of the code, then delete that edge and add an edge from
\verb+blob+ to \verb+succ(3)+ (which is 0).  Then we delete the edge from
$3$ to $0$ and put $3$ into the blob.  The new tree is:

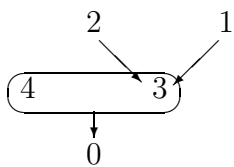
\begin{picture}(75,50)
\put(25,-25){0}
\put(0,0){4}
\put(28,2){\oval(65,15)}
\put(50,0){3}
\put(28,-5){\vector(0,-1){10}}
\put(25,25){2}
\put(75,25){1}
\put(30,22){\vector(1,-1){16}}
\put(75,22){\vector(-1,-1){17}}
\end{picture}

\vspace{10pt}
\subsubsection{Step 2}

$n-i=2$ and $\verb+code+ = (0)$.  Since $i<n$, we continue.  Does the path
from $2$ to $0$ go through the blob?  Yes.  We follow the {\bf else} 
in the algorithm.  Put \verb+succ(2)+, which is $3$, at the beginning 
of the code, get rid of that edge and put $2$ in the blob.

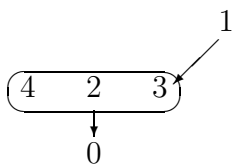
\begin{picture}(75,50)
\put(25,-25){0}
\put(0,0){4}
\put(50,0){3}
\put(28,-5){\vector(0,-1){10}}
\put(28,2){\oval(65,15)}
\put(25,0){2}
\put(75,25){1}
\put(75,22){\vector(-1,-1){17}}
\end{picture}

\vspace{10pt}
\subsubsection{Step 3}

Now $n-i=1$ and $\verb+code+ =(3,0)$.  Since $i>0$, we continue.  Does the path 
from $1$ to $0$ go through the blob?  Yes.  Prepend \verb+succ(1)+, which
is $3$ again, to the code, get rid of that edge and put $1$ in the blob.

Now we are done.  $i=n$ and $\verb+code+ = (3,3,0)$, and we stop.  
Here is the new tree:

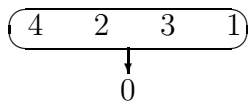
\begin{picture}(100,50)
\put (50,0){0}
\put (15,25){4}
\put(65,25){3}
\put(53,20){\vector(0,-1){10}}
\put(53,27){\oval(90,15)}
\put(40,25){2}
\put(90,25){1}
\end{picture}

\end{ex}

To see what the tree algorithm (which doesn't even refer to matrices at all) 
has to do with the matrix method, we note that the $i^{\rm th}$ row of
the initial matrix $C_0'$ represents the possible edges out of $i$.  Thus,
a row operation that cancels most of the entries of that row obliterates
the information of what the edge out of $i$ was.  This resembles the
placing of $i$ into the blob--since there is only one edge leaving the 
blob, we no
longer know where the individual vertex $i$ was pointing. However, the
information is not entirely lost because the code-in-progress is still in
the matrix. In fact, the row operation followed by the column operation 
corresponds directly to the blobbing of vertices and adding to the code.

More specifically,
the relationship between the tree method and the matrix method is as follows:
At the end of step $i$, we are in the set $G_i$.  The matrix $C_i'$
represents the graph with vertices $n-i,\dots,n$ in the blob.  The upper-left
corner with $n-i$ rows and columns is the Matrix Tree Theorem matrix for
that graph, and the $i$ rows with nothing except $B$ on the diagonal represent
the set of possible codes-in-progress.  If the path from $n-i$ to 0 does
not pass through the \verb+blob+, we follow the {\bf else} at step 
$i$ in the tree surgery algorithm, which 
corresponds to getting to pass through matrices easily from $G_{i-1}'$ to 
$G_i$.  If it does (ie, we follow the {\bf then} at step $i$ in the tree
surgery algorithm), the matrix method will involve several bounces
up and down within the matrices between $S_i$ and $T_i'$.

\section{Tree Surgery Is A Bijection}\label{backwards}
The tree surgery method is reversible.  The inverse algorithm takes a
code $(c_1,c_2,\dots c_{n-1})$ and finds the corresponding tree:

\begin{tabbing}
Algorithm to go from Blob Code to Tree\\
{\bf begin}\=\\
\>$i\leftarrow 0$\\
\>$\verb+blob+=\{1,\dots,n\}$\\
\>$\verb+edges+=\{\verb+blob+\to 0\}$\\
\>{\bf repeat}\=\\
\>\>$i\leftarrow i+1$\\
\>\>$\verb+blob+\leftarrow\verb+blob+\setminus\{i\}$\\
\>\>{\bf if} \= $\verb+path+(c_1)\cap\verb+blob+\neq\emptyset$ {\bf then}\\
\>\>\>add edge $i\to c_1$\\
\>\>{\bf else}\\
\>\>\>add edge $i\to\verb+succ+(\verb+blob+)$\\
\>\>\>remove edge $\verb+blob+\to\verb+succ+(\verb+blob+)$\\
\>\>\>add edge $\verb+blob+\to c_1$\\
\>\>behead \verb+code+\\
\>{\bf until} $i=n-1$\\
{\bf end.}
\end{tabbing}

It is easy to check that this algorithm undoes the Blob Code algorithm,
one step at a time.  

\section{The Two Methods Give the Same Blob Code}

\begin{theorem}
The matrix method and the tree surgery method give the same Blob Code.
\end{theorem}
\begin{proof}
We assume constant $n$ and proceed by induction on the number of steps
$i$ taken so far.  The base case is $i=0$, the zeroth step.  
Before we do anything (using either method), we have a tree and an 
empty code.  We consider the vertex $n$ to be a \verb+blob+ containing
only one label ($n$).  At the end of the $0^{\rm th}$ step, both methods
have the same code-in-progress (namely, an empty code) and the same tree.

Now we assume that at the end of the $(i-1)^{\rm th}$ step, the two
methods result in the same tree and code-in-progress.

At the beginning of step $i$, each method has a pair consisting of a
tree with a \verb+blob+ as one of the vertices and a partial code of length
$(i-1)$.  The \verb+blob+ contains $n-i+1,n-i+2,\dots,n$, so its size is $i$. 

The matrix method requires following the involutions through sets of
arrays.  Rows $n-i$ through $n-i+1$ look like this in the
sequence of matrices:
\begin{multline*}
C_{i-1}'=\\
\left[\begin{array}{r|ccccccc} 
&1&\dots&n-i&n-i+1&n-i+2&\dots&n\\
\hline\\
\vdots&\vdots&\ddots&\vdots&\vdots&\vdots&&\vdots\\
n-i&-b_1&\dots&B-b_{n-i}&-\displaystyle{\sum_{n-i+1}^n b_k}&
 -\displaystyle{\sum_{n-i+2}^n b_k}&\dots&-b_n\\
n-i+1&-b_1&\dots&-b_{n-i}&B-\displaystyle{\sum_{n-i+1}^n b_k}&
 -\displaystyle{\sum_{n-i+2}^n b_k}&\dots&-b_n\\
n-i+2&0&\dots&0&0&B&\dots&0\\
\vdots&\vdots&&\vdots&\vdots&\vdots&\ddots&\vdots
\end{array}\right]
\end{multline*}
\begin{multline*}
R_i=\\
{\footnotesize\begin{bmatrix}
&\vdots&\vdots&\vdots&&\vdots\\
\dots&B-b_{n-i}&-\displaystyle{\sum_{n-i+1}^n b_k}&
 -\displaystyle{\sum_{n-i+2}^n b_k}&\dots&-b_n\\
\dots&-b_{n-i}-B+b_{n-i}&B-\displaystyle{\sum_{n-i+1}^n b_k}
 +\displaystyle{\sum_{n-i+1}^n b_k}&-\displaystyle{\sum_{n-i+2}^n b_k}+
 \displaystyle{\sum_{n-i+2}^n b_k}&\dots&-b_n+b_n\\
\dots&0&0&B&\dots&0\\
&\vdots&\vdots&\vdots&&\vdots
\end{bmatrix}}
\end{multline*}
\[
R_i'=\begin{bmatrix}
\vdots&&\vdots&\vdots&\vdots&&\vdots\\
-b_1&\dots&B-b_{n-i}&-\displaystyle{\sum_{n-i+1}^n b_k}&
 -\displaystyle{\sum_{n-i+2}^n b_k}&\dots&-b_n\\
0&\dots&-B&B&0&\dots&0\\
0&\dots&0&0&B&\dots&0\\
\vdots&&\vdots&\vdots&\vdots&&\vdots
\end{bmatrix}
\]
\[
C_i=\begin{bmatrix}
\vdots&&\vdots&\vdots&\vdots&&\vdots\\
-b_1&\dots&B-b_{n-i}-\displaystyle{\sum_{n-i+1}^n b_k}&
 -\displaystyle{\sum_{n-i+1}^n b_k}&
 -\displaystyle{\sum_{n-i+2}^n b_k}&\dots&-b_n\\
0&\dots&-B+B&B&0&\dots&0\\
0&\dots&0&0&B&\dots&0\\
\vdots&&\vdots&\vdots&\vdots&&\vdots
\end{bmatrix}
\]
\[
C_i'=\begin{bmatrix}
\vdots&&\vdots&\vdots&\vdots&&\vdots\\
-b_1&\dots&B-\displaystyle{\sum_{n-i}^n b_k}&
 -\displaystyle{\sum_{n-i+1}^n b_k}&
 -\displaystyle{\sum_{n-i+2}^n b_k}&\dots&-b_n\\
0&\dots&0&B&0&\dots&0\\
0&\dots&0&0&B&\dots&0\\
\vdots&&\vdots&\vdots&\vdots&&\vdots
\end{bmatrix}
\]
At step $i$ in the matrix method, we are dealing with the sets $G_{i-1}'$ (the
set of trees (with a \verb+blob+ containing $i$ labels) and partial codes of
length $i-1$), $S_i,S_i',T_i,T_i'$ (the sets of arrays in the matrices above, 
respectively), and $G_i$ (the set of trees with a \verb+blob+ containing $i+1$
labels together with partial codes of length $i$).

Suppose we are at the start of step $i$.  
This means that no matter which method
we are using, we have a tree and a partial code.  
Let $\verb+succ+(\verb+blob+)=l$ and $\verb+succ+(n-i)=k$.  
Also suppose that the first element in the partial code is
$b_m$.  Note that since we have a tree, $l\leq n-i$ because all
vertices with labels greater than $n-i$ are in the \verb+blob+.

An application of $\overline{-I_{S_i}}\circ\rho_i$ leaves us with 
$\left[\begin{smallmatrix}\ddots&&&&\\&B_k&&&\\&&B_l&&\\&&&B_m&\\&&&&\ddots
\end{smallmatrix}\right]\in S_i$.
\[
\overline{-I_{S_i'}}\circ\rho_i'\left(
\left[\begin{smallmatrix}\ddots&&&&\\&B_k&&&\\&&B_l&&\\&&&B_m&\\&&&&\ddots
\end{smallmatrix}\right]\right)=
\left[\begin{smallmatrix}\ddots&&&&\\&B_k&&&\\&&B_l&&\\&&&B_m&\\&&&&\ddots
\end{smallmatrix}\right]\in S_i'.
\]
\[
\overline{-I_{T_i}}\circ\kappa_i\left(
\left[\begin{smallmatrix}\ddots&&&&\\&B_k&&&\\&&B_l&&\\&&&B_m&\\&&&&\ddots
\end{smallmatrix}\right]\right)=
\left[\begin{smallmatrix}\ddots&&&&\\&B_k&&&\\&&B_l&&\\&&&B_m&\\&&&&\ddots
\end{smallmatrix}\right]\in T_i.
\]
Note that these positive capitalized entries 
on the diagonal do not disappear from the 
matrices.
\[
\overline{-I_{T_i'}}\circ\kappa_i'\left(
\left[\begin{smallmatrix}\ddots&&&&\\&B_k&&&\\&&B_l&&\\&&&B_m&\\&&&&\ddots
\end{smallmatrix}\right]\right)=
\left[\begin{smallmatrix}\ddots&&&&\\&B_k&&&\\&&B_l&&\\&&&B_m&\\&&&&\ddots
\end{smallmatrix}\right]\in T_i'.
\]
Next we will be applying $\overline{-I_{G_i}}\circ\mu_i$, and there are 
two possible outcomes.

\underline{Case 1}
Consider the case where the path from $n-i$ to 0 does not go through the
\verb+blob+ (that is, $k$ is not inverted in the original tree).  
If the path from $k$ to $0$ does not pass through the \verb+blob+, then 
$\overline{-I_{G_i}}\circ\mu_i=(\tau,\gamma)$ where $\gamma$ 
is the code from $G_{i-1}$
with $b_l$ prepended to it and $\tau$ is a tree containing the
same edges as the tree from $G_{i-1}$ with the following exceptions:
$n-i\to k$ has been deleted, $n-i$ has been added to the \verb+blob+, and
the edge $\verb+blob+\to l$ has been replaced by the edge $\verb+blob+\to k$.
This is a tree because if the path from $n-i$ to 0 does not pass through
the \verb+blob+, then moving the \verb+blob+ to the position where $n-i$
was does not create a cycle.

Note that the effect is exactly the same as the result of
the tree surgery method.  Tree surgery would have removed and
added exactly those same edges, and prepended the same label to the code.

\underline{Case 2}
This is the more complicated case.  Here, when we apply $\mu_i$, we don't
get a tree because a cycle would be created (the path from $n-i$ to 0 goes
through the \verb+blob+, but now $n-i$ should be in the \verb+blob+ with
$\verb+succ+(\verb+blob+)=k$.  Hence there is a cycle containing \verb+blob+
and other vertices all of whose labels are less than $n-i$).  Thus,
$\overline{-I_{T_i'}}\circ\mu_i\left(\left[\begin{smallmatrix}\ddots&&&&\\
&B_k&&&\\&&B_l&&\\&&&B_m&\\&&&&\ddots\end{smallmatrix}\right]\right)$ is 
a negative element in $T_i'$ 
with all entries that correspond to edges in the cycle moved off 
the diagonal.  In this matrix, row $n-i$ contains $-b_k$ in the 
$k^{\rm th}$ column;
the rest of the off-diagonal entries are higher up in the matrix, including
some unique
entry in the $n-i$ column (say $b_r$, where $r\geq n-i$; this corresponds
to an edge into \verb+blob+).  If $k>n-i$ (that is, 
$\verb+succ+(n-i)\in\verb+blob+$), then the matrix will look a little 
different than the one below; we will deal with that case later.

\underline{Case 2a} If $k<n-i$ ($k\neq n-i$ because then we would have a loop
in the tree at the start of the step), we have    
\[
\overline{-I_{T_i'}}\circ\mu_i\left(\left[\begin{smallmatrix}\ddots&&&&\\
&B_k&&&\\&&B_l&&\\&&&B_m&\\&&&&\ddots\end{smallmatrix}\right]\right)=
-\left[\begin{smallmatrix}
\ddots&&&&&&\\
&&&-b_r&&&\\
&&\ddots&&&&\\
&-b_k&&&&\\
&&&&B_l&&\\
&&&&&B_m&\\
&&&&&&\ddots
\end{smallmatrix}\right]\in -T_i'
\]
Note that this $-b_r$ represents an edge into the \verb+blob+ and thus
$r$ can be any label greater than or equal to $n-i$.  Also, there may
be many vertices in the cycle that is now off the diagonal.

$\overline{-I_{T_i}}\circ\kappa_i'$ of this
gives the same array in $-T_i$. However, $\overline{-I_{T_i}}\circ\kappa_i$ of
{\em that} switches the entries in columns $n-i$ and $n-i+1$, leaving us 
with an element of $T_i$
because $R_i'$ only has $-b_{n-i}$ above the diagonal in column $n-i$.  
In some row above $n-i$, our array in $T_i$ has $-b_r$ in column 
$n-i+1$; it also has $B_l$ in the $(n-i+1,n-i)$ position; nothing else has 
moved (the $(n-i,k)$ position contains $-b_k$).  
\[
\overline{-I_{T_i}}\circ\kappa_i\left(-\left[
\begin{smallmatrix}
\ddots&&&&&&\\
&&&-b_r&&&\\
&&\ddots&&&&\\
&-b_k&&&&\\
&&&&B_l&&\\
&&&&&B_m&\\
&&&&&&\ddots
\end{smallmatrix}\right]\right)=
\left[\begin{smallmatrix}
\ddots&&&&&&\\
&&&&-b_r&&\\
&&\ddots&&&&\\
&-b_k&&&&\\
&&&B_l&&&\\
&&&&&B_m&\\
&&&&&&\ddots
\end{smallmatrix}\right]
\]
$\overline{-I_{T_i}}\circ\kappa_i'$ changes the sign of the $B_l$ in row 
$n-i+1$, leaving us in $-T_i$.  This new array appears in $-S_i'$ and $-S_i$
too: $\overline{-I_{S_i'}}\circ\kappa_i$ takes us to $-S_i'$ and 
$\overline{-I_{S_i}}\circ\rho_i'$ takes us to $-S_i$.  
\begin{multline*}
\overline{-I_{T_i}}\circ\kappa_i'\left(\left[\begin{smallmatrix}
\ddots&&&&&&\\
&&&&-b_r&&\\
&&\ddots&&&&\\
&-b_k&&&&\\
&&&B_l&&&\\
&&&&&B_m&\\
&&&&&&\ddots
\end{smallmatrix}\right]\right)=\\-
\left[\begin{smallmatrix}
\ddots&&&&&&\\
&&&&-b_r&&\\
&&\ddots&&&&\\
&-b_k&&&&\\
&&&-B_l&&&\\
&&&&&B_m&\\
&&&&&&\ddots
\end{smallmatrix}\right]\in -T_i;
\end{multline*}
\begin{multline*}
\overline{-I_{S_i'}}\circ\kappa_i\left(-\left[\begin{smallmatrix}
\ddots&&&&&&\\
&&&&-b_r&&\\
&&\ddots&&&&\\
&-b_k&&&&\\
&&&-B_l&&&\\
&&&&&B_m&\\
&&&&&&\ddots
\end{smallmatrix}\right]\right)=\\-
\left[\begin{smallmatrix}
\ddots&&&&&&\\
&&&&-b_r&&\\
&&\ddots&&&&\\
&-b_k&&&&\\
&&&-B_l&&&\\
&&&&&B_m&\\
&&&&&&\ddots
\end{smallmatrix}\right]\in -S_i';
\end{multline*}
and
\begin{multline*}
\overline{-I_{S_i}}\circ\rho_i'\left(-\left[\begin{smallmatrix}
\ddots&&&&&&\\
&&&&-b_r&&\\
&&\ddots&&&&\\
&-b_k&&&&\\
&&&-B_l&&&\\
&&&&&B_m&\\
&&&&&&\ddots
\end{smallmatrix}\right]\right)=\\
-\left[\begin{smallmatrix}
\ddots&&&&&&\\
&&&&-b_r&&\\
&&\ddots&&&&\\
&-b_k&&&&\\
&&&-B_l&&&\\
&&&&&B_m&\\
&&&&&&\ddots
\end{smallmatrix}\right]\in -S_i.
\end{multline*}
Now we will end
up switching the entries in rows $n-i$ and $n-i+1$:
$\overline{-I_{S_i}}\circ\rho_i$ has this effect, with the result that our
new array in $S_i$ has $B_l$ in the $(n-1)^{\rm th}$ diagonal entry and $b_k$
in the $(n-i+1,k)$ position.  
\[
\overline{-I_{S_i}}\circ\rho_i\left(-\left[\begin{smallmatrix}
\ddots&&&&&&\\
&&&&-b_r&&\\
&&\ddots&&&&\\
&-b_k&&&&\\
&&&-B_l&&&\\
&&&&&B_m&\\
&&&&&&\ddots
\end{smallmatrix}\right]\right)=
\left[\begin{smallmatrix}
\ddots&&&&&&\\
&&&&-b_r&&\\
&&\ddots&&&&\\
&&&B_l&&\\
&b_k&&&&&\\
&&&&&B_m&\\
&&&&&&\ddots
\end{smallmatrix}\right]\in S_i.
\]
An application of $\overline{-I_{S_i}}\circ\rho_i'$
changes the sign of the $b_k$ in row $n-i+1$, putting us in $-S_i$:  
\[
\overline{-I_{S_i}}\circ\rho_i'\left(\left[\begin{smallmatrix}
\ddots&&&&&&\\
&&&&-b_r&&\\
&&\ddots&&&&\\
&&&B_l&&\\
&b_k&&&&&\\
&&&&&B_m&\\
&&&&&&\ddots
\end{smallmatrix}\right]\right)=-\left[\begin{smallmatrix}
\ddots&&&&&&\\
&&&&-b_r&&\\
&&\ddots&&&&\\
&&&B_l&&\\
&-b_k&&&&&\\
&&&&&B_m&\\
&&&&&&\ddots
\end{smallmatrix}\right]\in -S_i.
\]
This same
array appears in $-G_{i-1}'$ and is what we get by applying 
$\overline{-I_{G_{i-1}'}}\circ\rho_i$. Now when we apply 
$\overline{-I_{G_{i-1}'}}\circ\mu_{i-1}'$ we have a different cycle.
Here, the graph in question has edges $\verb+blob+\to k$ and $(n-i)\to l$
instead of vice versa.  The off-diagonal entries must correspond to a 
cycle, so we move the cycle back onto the diagonal, landing in $G_{i-1}'$.
\[
\overline{-I_{G_{i-1}'}}\circ\mu_{i-1}'\left(-\left[\begin{smallmatrix}
\ddots&&&&&&\\
&&&&-b_r&&\\
&&\ddots&&&&\\
&&&B_l&&\\
&-b_k&&&&&\\
&&&&&B_m&\\
&&&&&&\ddots
\end{smallmatrix}\right]\right)=
\left[\begin{smallmatrix}
\ddots&&&&\\
&B_l&&&\\
&&B_k&&\\
&&&B_m&\\
&&&&\ddots
\end{smallmatrix}\right].
\]

Note that the only way this array differs from the one we started with at the
very beginning of step $i$ is that the entries in rows $n-i$ and $n-i+1$ have
been interchanged.  

Now when we apply $\rho_i,\rho_i',\kappa_i,\kappa_i'$ 
with the appropriate negative identity maps in between, we eventually reach
\[
\left[\begin{smallmatrix}\ddots&&&&\\&B_l&&&\\&&B_k&&\\&&&B_m&\\&&&&\ddots
\end{smallmatrix}\right]\in T_i',
\]
and then
\[
\overline{-I_{G_i}}\circ\mu_i\left(
\left[\begin{smallmatrix}\ddots&&&&\\&B_l&&&\\&&B_k&&\\&&&B_m&\\&&&&\ddots
\end{smallmatrix}\right]\right),
\]
which is a tree with edge $\verb+blob+\to l$ (where $n-i$ is now in 
the \verb+blob+) 
together with a code beginning with $(b_k,b_m,\dots)$.
Since the tree surgery method would have deleted the edge from $n-i$ to $k$,
placed $n-i$ in the \verb+blob+, prepended $b_k$ to the code
and left the edge $\verb+blob+\to l$, the
matrix method had exactly the same effect.  

\underline{Case 2b}
Here we treat separately the case where $\verb+succ+(n-i)\in\verb+blob+$.
In this case, we have
\[
\overline{-I_{T_i'}}\circ\mu_i\left(\left[\begin{smallmatrix}\ddots&&&&\\
&B_k&&&\\&&B_l&&\\&&&B_m&\\&&&&\ddots\end{smallmatrix}\right]\right)=
-\left[\begin{smallmatrix}
\ddots&&&&\\&-b_k&&&\\&&B_l&&\\&&&B_m&\\&&&&\ddots
\end{smallmatrix}\right]\in -T_i'.
\]
\[
\overline{-I_{T_i}}\circ\kappa_i'\left(-\left[\begin{smallmatrix}
\ddots&&&&\\&-b_k&&&\\&&B_l&&\\&&&B_m&\\&&&&\ddots
\end{smallmatrix}\right]\right)=
-\left[\begin{smallmatrix}
\ddots&&&&\\&-b_k&&&\\&&B_l&&\\&&&B_m&\\&&&&\ddots
\end{smallmatrix}\right]\in -T_i.
\]
Now $\kappa_i$ will switch the columns of two of the entries.  
\[
\overline{-I_{T_i}}\circ\kappa_i\left(-\left[\begin{smallmatrix}
\ddots&&&&\\&-b_k&&&\\&&B_l&&\\&&&B_m&\\&&&&\ddots
\end{smallmatrix}\right]\right)=
\left[\begin{smallmatrix}
\ddots&&&&\\&&-b_k&&\\&B_l&&&\\&&&B_m&\\&&&&\ddots
\end{smallmatrix}\right]\in T_i.
\]
\[
\overline{-I_{T_i}}\circ\kappa_i'\left(\left[\begin{smallmatrix}
\ddots&&&&\\&&-b_k&&\\&B_l&&&\\&&&B_m&\\&&&&\ddots
\end{smallmatrix}\right]\right)=
-\left[\begin{smallmatrix}
\ddots&&&&\\&&-b_k&&\\&-B_l&&&\\&&&B_m&\\&&&&\ddots
\end{smallmatrix}\right]\in -T_i.
\]
$\kappa_i'$ is defined to change the sign of the entry $B_l$ in row
$n-i+1$, but nothing else in the array changes.  This element also 
occurs in the sets $-S_i'$ and $-S_i$, so
\[
\overline{-I_{S_i}}\circ\rho_i'\overline{-I_{S_i'}}\circ\kappa_i\left(
-\left[\begin{smallmatrix}
\ddots&&&&\\&&-b_k&&\\&-B_l&&&\\&&&B_m&\\&&&&\ddots
\end{smallmatrix}\right]\right)=-\left[\begin{smallmatrix}
\ddots&&&&\\&&-b_k&&\\&-B_l&&&\\&&&B_m&\\&&&&\ddots
\end{smallmatrix}\right]\in -S_i.
\]
In Case 2a we actually made it all the way up to the set
$G_{i-1}'$, but this time we do not; the next thing that happens
is that the entries in rows $n-i$ and $n-i+1$ are interchanged,
with the requisite sign changes:
\[
\overline{-I_{S_i}}\circ\rho_i\left(-\left[\begin{smallmatrix}
\ddots&&&&\\&&-b_k&&\\&-B_l&&&\\&&&B_m&\\&&&&\ddots
\end{smallmatrix}\right]\right)=\left[\begin{smallmatrix}
\ddots&&&&\\&B_l&&&\\&&b_k&&\\&&&B_m&\\&&&&\ddots
\end{smallmatrix}\right]\in S_i.
\]
\[
\overline{-I_{S_i}}\circ\rho_i'\left(\left[\begin{smallmatrix}
\ddots&&&&\\&B_l&&&\\&&b_k&&\\&&&B_m&\\&&&&\ddots
\end{smallmatrix}\right]\right)=-\left[\begin{smallmatrix}
\ddots&&&&\\&B_l&&&\\&&-b_k&&\\&&&B_m&\\&&&&\ddots
\end{smallmatrix}\right]\in -S_i.
\]
\[
\overline{-I_{G_{i-1}'}}\circ\rho_i\left(-\left[\begin{smallmatrix}
\ddots&&&&\\&B_l&&&\\&&-b_k&&\\&&&B_m&\\&&&&\ddots
\end{smallmatrix}\right]\right)=-\left[\begin{smallmatrix}
\ddots&&&&\\&B_l&&&\\&&-b_k&&\\&&&B_m&\\&&&&\ddots
\end{smallmatrix}\right]\in -G_{i-1}'
\]
This array does not correspond to a tree because there is a loop
$\verb+blob+\to k$
\[
\overline{-I_{G_{i-1}'}}\circ\mu_{i-1}'\left(-\left[\begin{smallmatrix}
\ddots&&&&\\&B_l&&&\\&&-b_k&&\\&&&B_m&\\&&&&\ddots
\end{smallmatrix}\right]\right)=\left[\begin{smallmatrix}
\ddots&&&&\\&B_l&&&\\&&B_k&&\\&&&B_m&\\&&&&\ddots
\end{smallmatrix}\right]\in G_{i-1}'
\]

Now we can go ahead and apply (with the obvious negative identity maps
in between) $\rho_i, \rho_i', \kappa_i,$ and $\kappa_i'$, eventually ending up
with this same array in the set $T_i'$.  All of this had exactly the
same effect that the manipulations in Case 2a did--namely, we interchanged
the two entries on the diagonal, switching $b_k$ with $b_l$.  The same 
argument we used above shows that this had the same effect as the tree
surgery method.

Since we have accounted for all possible cases, we conclude that these two 
methods give the same code at step $i$.  Thus at step $n$ the effect of the
two methods is the same, so by induction the Blob Code can be found using
either method.
\end{proof}
%whew!!

\chapter{Tree Surgery for the Happy Code}

Considering that the matrix method did not refer back to the graph at
each step, it is surprising that there is a purely bijective method for
finding the Happy Code.
In fact, we do have another form of tree surgery for the Happy Code, so we
can avoid resorting to matrices and involutions.

\section{Tree Surgery Algorithm}
Begin by finding the path from $1$ to 0.  The
method consists of deleting $\verb+succ+(1)$ from the path and moving it to
a separate connected component of the graph, and forming a cycle with it, then
repeating the process.  The algorithm corresponds directly to
the matrix/involution algorithm of chapter \ref{happyhell}. 
%\begin{defn}
%If $v$ is a $k$-tuple then $v_i$ is the $i^{\rm th}$ entry in $v$, and
%{\bf behead}$(v)$ is the $(k-1)$-tuple found by deleting $v_1$.
%\end{defn}
The algorithm below takes as its input a tree in the form of a set of
edges.

\begin{tabbing}
The Tree Surgery Algorithm for Happy Code\\
{\bf begin}\=\\
\>$J\leftarrow\verb+succ+(1)$\\
\>{\bf if} $J\neq 0$ \={\bf then}\\
\>\>{\bf repeat}\=\\
\>\>\>$j\leftarrow\verb+succ+(1)$\\
\>\>\>remove edge $1\to j$\\
\>\>\>add edge $1\to\verb+succ+(j)$\\
\>\>\>{\bf if} $j\geq J$ \= {\bf then}\\
\>\>\>\>add edge $j\to j$\\
\>\>\>\>$J\leftarrow j$\\
\>\>\>{\bf else}\\
\>\>\>\>add edge $j\to\verb+succ+(J)$\\
\>\>\>\>remove edge $J\to\verb+succ+(J)$\\
\>\>\>\>add edge $J\to j$\\
\>\>{\bf until} $\verb+succ+(1)=0$\\
\>{\bf else}\\
\>\>\{the Happy Code is practically the same as the na\"\i ve code\}\\
\>$\verb+code+\leftarrow(\verb+succ+(2),\verb+succ+(3),\dots,\verb+succ+(n))$\\
{\bf end.}
\end{tabbing}

\noindent
This algorithm turns out to be essentially equivalent to the matrix method
shown in Chapter \ref{happyhell}.  

\begin{comment}
The bijection we came up with using matrices and involutions
can be split into steps.  At the first
step, we follow the involutions around until we arrive at the set $A_{n+1}$.
If the first item in the product is $B_0$ then we are essentially done and 
can move on to $A_{n+1}'$, leaving the factors in the same order.  
This corresponds to the case where $\verb+succ+(1)
=0$; in which case the code is simply given by taking $(\verb+succ+(2),
\verb+succ+(3),\dots,\verb+succ+(n))$.

If the first item in the product is not $B_0$, then $\verb+succ+(1)
\neq 0$, so we let $i=\verb+succ+(1)$, put it in the set of $C$ of cycles,
and change the edge $1\to\verb+succ+(1)$ to $1\to\verb+succ+(i)$.  We let
$J=\displaystyle{\max_{j\in C}\{j\}}$.  
If $i=J$ then we create a loop at $i$, 
so that the new $\verb+succ+(i)$ is $i$.  Otherwise, we insert $i$ into the
cycle after $J$ by assigning a new value to $\verb+succ+(i)$ as 
$\verb+succ+(J)$ and to $\verb+succ+(J)$ as $i$.  All of this corresponds 
to what would happen in the involutions above:  we switch $B_{j_1}$ with
$B_{j_{j_1}}$ (sending $1$ to $\verb+succ+(i)$ and creating a loop at $i$),
then as we follow the involutions around, eventually we end up back at
$A_0'$ where the Matrix Tree Theorem (in the form of $\phi_0$) 
allows us to alter the cycle with the
largest element in it.  By the time we find ourselves back at $A_{n+1}$, 
we have inserted $i$ into the appropriate cycle.  
\end{comment}

\section{An Example}
Consider the tree $1\to 3\to 2\to 4\to 0$.  Step 1:  pull $\verb+succ+(1)=3$ 
out of the
path from $1$ to $0$ and put it in a cycle.

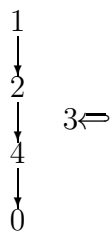
\vspace{10pt}
\begin{picture}(50,75)
\put(20,75){1}
\put(23,73){\vector(0,-1){15}}
\put(20,50){2}
\put(23,48){\vector(0,-1){15}}
\put(20,25){4}
\put(23,23){\vector(0,-1){15}}
\put(20,0){0}
\put(40,38){3}
\put(45,38){$\hookleftarrow$}
\put(49,43.5){\line(1,0){7}}
\end{picture}

One nice thing about the Happy Code is that we don't have to keep track of
the code as we go; we just read it off at the end.  Step 2:  pull 2 
(the new $\verb+succ+(1)$) out of
the path from $1$ to $0$ and put it in a cycle.  Since it is not the largest
vertex in a cycle, we insert it after the largest (which is $3$).

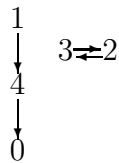
\begin{picture}(60,65)
\put(20,50){1}
\put(23,48){\vector(0,-1){15}}
\put(20,25){4}
\put(23,23){\vector(0,-1){15}}
\put(20,0){0}
\put(38,38){3}
\put(55,38){2}
\put(44,42){\vector(1,0){10}}
\put(55,39){\vector(-1,0){10}}
\end{picture}

The last step is to pull $4$ out of the path from $1$ to $0$; it gets a loop
because it is the largest element of the cycles.

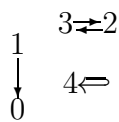
\begin{picture}(60,65)
\put(20,30){1}
\put(23,28){\vector(0,-1){15}}
\put(20,5){0}
\put(40,15){4}
\put(45,15){$\hookleftarrow$}
\put(49,20.5){\line(1,0){7}}
\put(38,38){3}
\put(55,38){2}
\put(44,42){\vector(1,0){10}}
\put(55,39){\vector(-1,0){10}}
\end{picture}

Now we can write down, in order, the successors of $2,3,4$ to find the code:
$(3,2,4)$.  Notice how much faster the tree surgery procedure is!  Also, it
is nice to know that it would be even faster if the path from $1$ to $0$ were
shorter.  Another nice feature of this method is that we no longer have to
keep track of the code as we go; instead, we find it directly once we have
finished performing surgery on the tree. The weight of the happy functional
digraph at the end of the process is equal to the weight of the original 
tree.  

If the tree were branchier, the method would not be any more complicated.
Edges that are not part of the path from 1 to 0 are not affected by tree
surgery; at the end of the surgical procedures the code is the list of the
respective successors of all vertices $\geq 2$.

This tree surgery method is related to Joyal's proof that there are 
$(n+1)^{n-1}$ trees.  See \S\ref{joy} for a discussion.
\section{Tree surgery is a bijection}
Again, there is a simple inverse for the Happy Code tree surgery.  We assume
that we have a procedure that figures out which vertices are in cycles.  The
input is a code $(c_1, c_2,\dots,c_{n-1})$.

\begin{tabbing}
Algorithm to go from Happy Code to Tree\\
{\bf begin}\=\\
\>$\verb+edges+=\{1\to 0\}$\\
\>{\bf for} \= $i=2$ {\bf to} $n$ {\bf do}\\
\>\>add edge $i\to c_{i-1}$\\
\>{\bf while} $\verb+cycles+\neq\emptyset$ {\bf do}\\
\>\>$J\leftarrow \displaystyle{\max_{j\in\verb+cycles+}j}$\\
\>\>$k\leftarrow\verb+succ+(J)$\\
\>\>add edge $J\to\verb+succ+(k)$\\
\>\>remove edge $J\to k$\\
\>\>add edge $k\to \verb+succ+(1)$\\
\>\>remove edge $1\to\verb+succ+(1)$\\
\>\>add edge $1\to k$\\
{\bf end.}
\end{tabbing}

It is clear that this algorithm undoes the Happy Code tree surgery, one
step at a time.

\section{The Two Methods Give the Same Happy Code}
\subsection{A Lemma}
In order to prove that the tree surgery method gives the same code as the
matrix method, we will need the following lemma.  The notion of a cycle being
``active'' or ``inactive'' is content-free.  A cycle is ``active'' if we
label it as active, and inactive otherwise.  Actually we will see later that
``active'' corresponds to appearing off the diagonal in the matrix, and 
``inactive'' corresponds to being on the diagonal.
\begin{lemma}\label{Escher}
The input for the following algorithm is an active loop at vertex $L$ 
and an active
cycle (which may also be a loop) containing at least one vertex greater than
$L$. Let $J$ be the largest element in the cycle.
Then the output is the original cycle, now inactive, with $L$ 
inserted between $J$ and $\verb+succ+(J)$.
\end{lemma}
\begin{tabbing}
{\bf begin}\=\\
\>{\bf repeat}\=\\
\>\>$p\leftarrow$ largest vertex in an active cycle\\
\>\>$q\leftarrow$ second-largest vertex in an active cycle\\
\>\>$m\leftarrow\verb+succ+(q)$\\
\>\>add edge $q\to\verb+succ+(p)$\\
\>\>remove edge $p\to\verb+succ+(p)$\\
\>\>remove edge $q\to m$\\
\>\>add edge $p\to m$\\
\>\>toggle ``activity'' of the cycle containing $J$\\
\>{\bf until} there are no active cycles.\\
{\bf end.}
\end{tabbing}
\begin{proof}
We begin by noting that for a cycle of length $c$, the worst-case scenario
is that each edge (other than $J\to\verb+succ+(J)$) is an ascent and all 
vertices are larger than the one in the loop.  For such
a cycle, the algorithm terminates after $2^c-1$ iterations.  In fact, 
in this situation the
iterative algorithm above is actually equivalent to a recursive algorithm.
This is proven by induction.  

The base case is that the cycle is 
a loop at $J$.  This is an Escher cycle of length $c=1$.  
The algorithm sets $p=J$,
$q=$ the vertex of the loop, and $m=q$.  It removes 
the loops and adds edges $J\to q$ and $q\to J$, then toggles the activity
of the cycle containing $J$.  There are no more active cycles and the
algorithm
has inserted $q$ directly after $J$ in its cycle.  Furthermore it has
taken $2^1-1=1$ step.

The induction hypothesis is that it takes $2^{c-1}-1$ steps to complete
the algorithm if the cycle is an Escher cycle of length $c-1$, and that 
the result is that of
inserting the loop vertex after $J$ in the cycle.  

Now we consider an Escher
cycle of length $c$ containing only vertices larger than the loop vertex.
Since each vertex of the cycle is larger than the loop vertex, the only
way to be able to change the edge from the loop vertex (call it $L$) is
to make all but one of the vertices in the cycle inactive.  

This is a slow
process.  The first step of the algorithm removes $J$ from the cycle, forming
a loop which becomes inactive.  Next, the second-largest vertex is removed,
and $J$ becomes active again.  The following step will form a 2-cycle with
these two vertices and make it inactive.  The procedure continues until only
the smallest vertex from the cycle (the original $\verb+succ+(J)$) 
is left in a loop, with $J$ and the rest of the vertices in an inactive 
cycle. By the induction hypothesis, 
this takes $2^{c-1}-1$ steps because it is precisely the reverse of
adding that smallest vertex to the cycle.  The next step of the iterative
algorithm switches the successors of $L$ and the old $\verb+succ+(J)$ and
makes the rest of the vertices active again.  The remaining steps merely
undo all of the previous steps, with the exception that $L$
has been inserted before the old $\verb+succ+(J)$ in all the cycles 
containing it.  The number of steps before we finish is thus 
$2\cdot(2^{c-1}-1)+1=2^c-1$.  Furthermore, since $L$ has been inserted before
$\verb+succ+(J)$, in the final cycle it appears right after $J$.

Thus we have the result in the case where the cycle is an Escher
cycle all of whose vertices are greater than $L$.  However, in fact
any cycle reduces to an Escher cycle of vertices greater than
$L$ in the following way:  any vertices smaller than $L$ will never be
affected by the edge switching, because $L$ is active until the bitter end
and is never the largest active vertex.  So these vertices can be considered
to be chained to their successors and thus do not effect the length of time
the algorithm takes nor its effect.  Furthermore, any vertices that fall
in between a vertex and its nearest greater neighbor are also chained to
their successors.  
\end{proof}
\begin{ex}

\vspace{7pt}

\noindent
\begin{picture}(250,110)
\put(25,75){9}
\put(30,78){\vector(1,0){20}}
\put(50,75){2}
\put(55,75){\vector(1,-1){20}}
\put(75,50){3}
\put(75,50){\vector(-1,-1){20}}
\put(50,25){6}
\put(50,28){\vector(-1,0){20}}
\put(25,25){8}
\put(25,30){\vector(-1,1){20}}
\put(0,50){7}
\put(5,58){\vector(1,1){19}}
\put(110,50){4}
\put(115,50){$\hookleftarrow$}
\put(119,55.5){\line(1,0){7}}
\put(0,100){{\bf Step 0}}
\put(50,100){\underline{Active}}
\put(200,100){\underline{Inactive}}
\end{picture}

\noindent
When we switch the successors of the two largest vertices, we replace
the edges $9\to 2$ and $8\to 7$ by the edges $9\to 7$ and $8\to 2$.  
This breaks our cycle into two cycles, one of which is inactive:

\vspace{15pt}
\noindent
\begin{picture}(250,100)
\put(0,100){{\bf Step 1}}
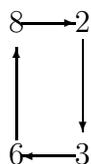
\put(50,100){\underline{Active}}
\put(200,100){\underline{Inactive}}
\put(25,75){8}
\put(30,78){\vector(1,0){20}}
\put(50,75){2}
\put(53,72){\vector(0,-1){35}}
\put(50,25){3}
\put(50,28){\vector(-1,0){20}}
\put(25,25){6}
\put(28,34){\vector(0,1){35}}
\put(110,50){4}
\put(115,50){$\hookleftarrow$}
\put(119,55.5){\line(1,0){7}}
\put(204,75){9}
\put(209,80){\vector(1,0){20}}
\put(229,75){7}
\put(229,75){\vector(-1,0){20}}
\end{picture}

\noindent 
We repeat.  We replace the edges $8\to 2$ and $6\to 8$ by the edges
$6\to 2$ and $8\to 8$ (a loop) and reactivate the cycle containing 9.

\vspace{15pt}
\noindent
\begin{picture}(250,100)
\put(0,100){{\bf Step 2}}
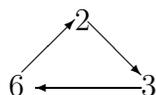
\put(50,100){\underline{Active}}
\put(200,100){\underline{Inactive}}
\put(110,25){8}
\put(115,25){$\hookleftarrow$}
\put(119,30.5){\line(1,0){7}}
\put(50,50){2}
\put(55,53){\vector(1,-1){20}}
\put(75,25){3}
\put(75,28){\vector(-1,0){40}}
\put(25,25){6}
\put(30,34){\vector(1,1){20}}
\put(110,50){4}
\put(115,50){$\hookleftarrow$}
\put(119,55.5){\line(1,0){7}}
\put(104,75){9}
\put(109,80){\vector(1,0){20}}
\put(129,75){7}
\put(129,75){\vector(-1,0){20}}
\end{picture}

\noindent
The two largest active vertices are 8 and 9, so 8 is inserted into
9's cycle.

\vspace{15pt}
\noindent
\begin{picture}(250,100)
\put(0,100){{\bf Step 3}}
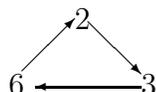
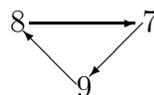
\put(50,100){\underline{Active}}
\put(200,100){\underline{Inactive}}
\put(200,75){8}
\put(207,78){\vector(1,0){40}}
\put(250,75){7}
\put(250,78){\vector(-1,-1){20}}
\put(225,50){9}
\put(225,55){\vector(-1,1){20}}
\put(50,50){2}
\put(55,53){\vector(1,-1){20}}
\put(75,25){3}
\put(75,28){\vector(-1,0){40}}
\put(25,25){6}
\put(30,34){\vector(1,1){20}}
\put(110,50){4}
\put(115,50){$\hookleftarrow$}
\put(119,55.5){\line(1,0){7}}
\end{picture}

\noindent
Now that the loop vertex, 4, is the second-largest active vertex,
it gets inserted into the other active cycle.  Note that it ends
up inserted just before $\verb+succ+(9)$.  This marks the approximate
halfway point of the process.  From now on we basically undo everything
we did.

\vspace{15pt}
\noindent
\begin{picture}(250,100)
\put(0,100){{\bf Step 4}}
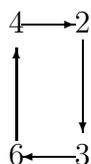
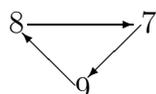
\put(50,100){\underline{Active}}
\put(200,100){\underline{Inactive}}
\put(100,75){8}
\put(107,78){\vector(1,0){40}}
\put(150,75){7}
\put(150,78){\vector(-1,-1){20}}
\put(125,50){9}
\put(125,55){\vector(-1,1){20}}
\put(25,75){4}
\put(30,78){\vector(1,0){20}}
\put(50,75){2}
\put(53,72){\vector(0,-1){35}}
\put(50,25){3}
\put(50,28){\vector(-1,0){20}}
\put(25,25){6}
\put(28,34){\vector(0,1){35}}
\end{picture}

\vspace{-5pt}\noindent
Now we switch the edges from 8 and 9, which has the effect of removing
8 from 9's cycle.  Step 5 corresponds to Step 2, only with 4 inserted
before 2 and the activity of 9's cycle toggled.

\vspace{15pt}
\noindent
\begin{picture}(250,100)
\put(0,100){{\bf Step 5}}
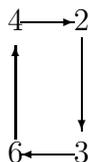
\put(50,100){\underline{Active}}
\put(200,100){\underline{Inactive}}
\put(25,75){4}
\put(30,78){\vector(1,0){20}}
\put(50,75){2}
\put(53,72){\vector(0,-1){35}}
\put(50,25){3}
\put(50,28){\vector(-1,0){20}}
\put(25,25){6}
\put(28,34){\vector(0,1){35}}
\put(204,75){9}
\put(209,80){\vector(1,0){20}}
\put(229,75){7}
\put(229,75){\vector(-1,0){20}}
\put(110,50){8}
\put(115,50){$\hookleftarrow$}
\put(119,55.5){\line(1,0){7}}
\end{picture}

\vspace{-5pt}\noindent
Now 8 will get inserted into the larger cycle, and 9's cycle is reactivated.
Step 6 corresponds to Step 1, except that 4 has been inserted before 2
and the cycle containing 9 has the opposite activity.

\vspace{15pt}
\noindent
\begin{picture}(250,100)
\put(0,100){{\bf Step 6}}
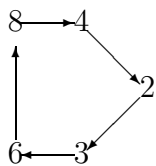
\put(50,100){\underline{Active}}
\put(200,100){\underline{Inactive}}
\put(25,75){8}
\put(30,78){\vector(1,0){20}}
\put(50,75){4}
\put(55,75){\vector(1,-1){20}}
\put(75,50){2}
\put(75,50){\vector(-1,-1){20}}
\put(50,25){3}
\put(50,28){\vector(-1,0){20}}
\put(25,25){6}
\put(28,34){\vector(0,1){35}}
\put(104,75){9}
\put(109,80){\vector(1,0){20}}
\put(129,75){7}
\put(129,75){\vector(-1,0){20}}
\end{picture}

\vspace{-5pt}\noindent 
Step 7 corresponds to Step 0.

\vspace{15pt}
\noindent
\begin{picture}(250,100)
\put(0,100){{\bf Step 7}}
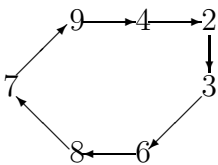
\put(50,100){\underline{Active}}
\put(200,100){\underline{Inactive}}
\put(225,75){9}
\put(230,78){\vector(1,0){20}}
\put(250,75){4}
\put(255,78){\vector(1,0){20}}
\put(275,75){2}
\put(278,73){\vector(0,-1){14}}
\put(275,50){3}
\put(275,50){\vector(-1,-1){20}}
\put(250,25){6}
\put(250,28){\vector(-1,0){20}}
\put(225,25){8}
\put(225,30){\vector(-1,1){20}}
\put(200,50){7}
\put(205,58){\vector(1,1){19}}
\end{picture}

\noindent
The end result is that of inserting 4 into the cycle, 
right after 9.  It took, in this case, $7=2^3-1$ steps, because the
cycle we started with 
is ``equivalent'' to the following Escher cycle with vertices
larger than 4:

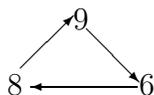
\vspace{15pt}
\begin{picture}(50,25)
\put(25,25){9}
\put(30,25){\vector(1,-1){20}}
\put(50,0){6}
\put(49,3){\vector(-1,0){40}}
\put(0,0){8}
\put(5,10){\vector(1,1){20}}
\end{picture}

\noindent
Note that 2 and 3 (the two vertices less than 4, our loop vertex)
are ``chained'' together and to 6, and their 
outgoing edges never change.  
Meanwhile, 7 is ``chained'' to its successor, 9, 
because the edge into 7 is not an ascent.
\end{ex}
\subsection{The proof}
\begin{theorem}
The tree surgery method gives the same Happy Code as the matrix method.
\end{theorem}
\begin{proof}
We assume constant $n$ and proceed by induction on the length of the
path from 1 to 0.  The base case is the case where the tree includes
the edge $1\to 0$.  In that case, the tree surgery method doesn't have
to go through the {\bf repeat} loop at all and the code is given by
$(\verb+succ+(2),\verb+succ+(3),\dots,\verb+succ+(n))$.  The matrix
method goes as follows for the base case:  first, an application of
$\overline{-I_{A_0'}}\circ\phi_0$ gives us an array with 
$B_{\verb+succ+(i)}$ in the $i^{\rm th}$ diagonal position.  
Let $j_i=\verb+succ+(i)$.  Next,
\[
\overline{-I_{A_1}}\circ\phi_0'\left(\begin{bmatrix}
\lambda &&&&\\
&B_0&&&\\
&&B_{j_2}&&\\
&&&\ddots&\\
&&&&B_{j_n}
\end{bmatrix}\right)=\begin{bmatrix}
\lambda &&&&\\
&B_0&&&\\
&&B_{j_2}&&\\
&&&\ddots&\\
&&&&B_{j_n}
\end{bmatrix}\in A_1.
\]
Since the involutions have been defined in such a way that none of these
diagonal entries ever get cancelled by a matrix operation, we have after 
many similar applications
\[
\overline{-I_{A_n'}}\circ\phi_n\left(\begin{bmatrix}
\lambda &&&&\\
&B_0&&&\\
&&B_{j_2}&&\\
&&&\ddots&\\
&&&&B_{j_n}
\end{bmatrix}\right)=\begin{bmatrix}
\lambda &&&&\\
&B_0&&&\\
&&B_{j_2}&&\\
&&&\ddots&\\
&&&&B_{j_n}
\end{bmatrix}\in A_n'.
\]
Now we have 
\[
\overline{-I_{A_{n+1}}}\circ\phi_n'\left(\begin{bmatrix}
\lambda &&&&\\
&B_0&&&\\
&&B_{j_2}&&\\
&&&\ddots&\\
&&&&B_{j_n}
\end{bmatrix}\right)=
B_0 B_{j_2} \dots B_{j_n}\in A_{n+1}.
\]
\[
\phi_{n+1}\left(B_0 B_{j_2} \dots B_{j_n}\right)=-b_0 B_{j_2} \dots B_{j_n}
\in -A_{n+1}'.
\]
Here, since $j_i=\verb+succ+(i)$, we end up with the same code we got by
tree surgery.  Thus the base case is true.

Our induction hypothesis is that the two methods give the same code for
all happy functional digraphs 
where the path from 1 to 0 is of length $i-1$.
We show that if we start with a functional digraph whose path from 1 to 0 is
of length $i$, both methods will manipulate the graph into one with a
shorter path from 1 to 0.  

The length of the path from 1 to 0 is $i$.  As we start, we have an array
with all entries on the diagonal.  We will automatically
(as in the base case) make it down to $A_{n+1}$ by a sequence of involutions
with no complications, because none of these diagonal entries get cancelled
in the row operation arithmetic.
Let $\verb+succ+(i)=j_i$ and $j_1=r$.  Then
\begin{multline*}
\overline{-I_{A_{n+1}}}\circ\phi_n'\circ\dots\circ\overline{-I_{A_0'}}\circ
\phi_0\left(\begin{bmatrix}
\lambda&&&&\\
&B_r&&&\\
&&B_{j_2}&&\\
&&&\ddots&\\
&&&&B_{j_n}
\end{bmatrix}\right)\\
=B_r B_{j_2} \dots B_{j_n}\in A_{n+1}.
\end{multline*}
Now since $r\neq 0$, 
\[\phi_{n+1}(B_r B_{j_2} \dots B_{j_n})=-B_{j_r}B_{j_2}\dots b_r\dots B_{j_n}
\in A_{n+1},
\]
and
\begin{multline*}
\overline{-I_{A_n}'}\circ\phi_n'\circ\overline{-I_{A_{n+1}}}\left(
-B_{j_r}B_{j_2}\dots b_r\dots B_{j_n}\right)=\\
-\begin{bmatrix}
&&&-b_r&&\\
&B_{j_r}&&&&\\
&&\ddots&&&\\
-\lambda&&&&&\\
&&&&\ddots&\\
&&&&&B_{j_n}
\end{bmatrix}\in -A_n'.
\end{multline*}
There will be no problem in applying involutions and we will move
swiftly through the sequence of sets $-A_n',-A_n,-A_{n-1}',-A_{n-1},\dots$
until we reach the one where the $\lambda$
first appears in this (the $r^{\rm th}$) row.
\begin{multline*}
\overline{-I_{A_{n-r+1}}}\circ\phi_{n-r}'\left(-\left[\begin{smallmatrix}
&&&-b_r&&\\
&B_{j_r}&&&&\\
&&\ddots&&&\\
-\lambda&&&&&\\
&&&&\ddots&\\
&&&&&B_{j_n}
\end{smallmatrix}\right]\right)=\\
\left[\begin{smallmatrix}
\lambda&&&&&\\
&B_{j_r}&&&&\\
&&\ddots&&&\\
&&&+b_r&&\\
&&&&\ddots&\\
&&&&&B_{j_n}
\end{smallmatrix}\right]\in A_{n-r+1}.
\end{multline*}
Now that we are in $A_{n-r+1}$, we apply $\phi_{n-r+1}$:
\begin{multline*}
\overline{-I_{A_{n-r+1}}}\circ\phi_{n-r+1}\left(\left[\begin{smallmatrix}
\lambda&&&&&\\
&B_{j_r}&&&&\\
&&\ddots&&&\\
&&&+b_r&&\\
&&&&\ddots&\\
&&&&&B_{j_n}
\end{smallmatrix}\right]\right)=\\-\left[\begin{smallmatrix}
\lambda&&&&&\\
&B_{j_r}&&&&\\
&&\ddots&&&\\
&&&-b_r&&\\
&&&&\ddots&\\
&&&&&B_{j_n}
\end{smallmatrix}\right]\in -A_{n-r+1}.
\end{multline*}
This array appears in all of the previous matrices, so we get all the
way back up to $A_0'$.  $\phi_0$ toggles the diagonality of the cycle
with the largest element.  Note that so far, what has happened is that 
we have switched the successors for 1 and $r$.  In other words, we have
removed $r$ from the path from 1 to 0, and created a loop at $r$; 1
now points directly at what used to be after $r$ on the path to 0.

\underline{Case 1}
If $r$ is the largest vertex in a cycle, 
\[
\overline{-I_{A_0'}}\circ\phi_0\left(-\left[\begin{smallmatrix}
\lambda&&&&&\\
&B_{j_r}&&&&\\
&&\ddots&&&\\
&&&-b_r&&\\
&&&&\ddots&\\
&&&&&B_{j_n}
\end{smallmatrix}\right]\right)=\left[\begin{smallmatrix}
\lambda&&&&&\\
&B_{j_r}&&&&\\
&&\ddots&&&\\
&&&B_r&&\\
&&&&\ddots&\\
&&&&&B_{j_n}
\end{smallmatrix}\right]\in A_0'.
\]
We are done
because the effect of the tree surgery method would have been
exactly the same:  we would have removed $r$ from the path from
1 to 0, and created a loop on it.  By the induction hypothesis,
the two methods will give the same code because the path from 1
to 0 now has length $i-1$.

%I might need this next stuff later!
\begin{comment}
Note that now that we have all capital entries on the diagonal,
we can zip right on down through the involutions until we reach
\[
\overline{-I_A_{n+1}\circ\phi_n'\left(\left[\begin{smallmatrix}
\lambda&&&&&\\
&B_{j_r}&&&&\\
&&\ddots&&&\\
&&&B_r&&\\
&&&&\ddots&\\
&&&&&B_{j_n}
\end{smallmatrix}\right]\right)=B_{j_r} B_{j_2} \dots B_r\dots B_{j_n}
\in A_{n+1}.
\]
\end{comment}
%Now, back to your regularly scheduled program.

Otherwise, we probably still have a long way to go. 

\underline{Case 2}
If the largest element in a cycle is not $r$, then tree surgery has the effect
of inserting $r$ after the largest element in a cycle, $J$.

In this case, the application
of $\phi_0$ will move another cycle off the diagonal.  Our new element
of $A_0'$ looks like this:
\[
\begin{bmatrix}
\lambda&&&&&&&\\
&B_{j_r}&&&&&&\\
&&\ddots&&&&&\\
&&&-b_r&&&&\\
&&&&\ddots&&&\\
&&-b_k&&&&&\\
&&&&&&\ddots&\\
&&&&&&&B_{j_n}
\end{bmatrix},
\]
where $k$ is one of the vertices in the new off-diagonal cycle,
and $k=\verb+succ+(J)$ where $J$ is the largest vertex in a cycle.
($-b_k$ is in row J.)  Nothing interesting
happens with the involutions until we reach $A_{n-J+1}$:
\begin{multline*}
\overline{-I_{A_{n-J+1}}}\circ\phi_{n-J+1}\left(\left[\begin{smallmatrix}
\lambda&&&&&&&\\
&B_{j_r}&&&&&&\\
&&\ddots&&&&&\\
&&&-b_r&&&&\\
&&&&\ddots&&&\\
&&-b_k&&&&&\\
&&&&&&\ddots&\\
&&&&&&&B_{j_n}
\end{smallmatrix}\right]\right)=\\-\left[\begin{smallmatrix}
\lambda&&&&&&&\\
&B_{j_r}&&&&&&\\
&&\ddots&&&&&\\
&&&-b_r&&&&\\
&&&&\ddots&&&\\
&&+b_k&&&&&\\
&&&&&&\ddots&\\
&&&&&&&B_{j_n}
\end{smallmatrix}\right]\in -A_{n-J+1}.
\end{multline*}
\begin{multline*}
\overline{-I_{A_{n-J+1}}}\circ\phi_{n-J}'\left(\left[\begin{smallmatrix}
\lambda&&&&&&&\\
&B_{j_r}&&&&&&\\
&&\ddots&&&&&\\
&&&-b_r&&&&\\
&&&&\ddots&&&\\
&&+b_k&&&&&\\
&&&&&&\ddots&\\
&&&&&&&B_{j_n}
\end{smallmatrix}\right]\right)=\\
\left[\begin{smallmatrix}
&&-b_k&&&&&\\
&B_{j_r}&&&&&&\\
&&\ddots&&&&&\\
&&&-b_r&&&&\\
&&&&\ddots&&&\\
-\lambda&&&&&&&\\
&&&&&&\ddots&\\
&&&&&&&B_{j_n}
\end{smallmatrix}\right]\in A_{n-J+1}.
\end{multline*}
From here we can move through $A_{n-J+1},A_{n-J+1}',\dots$ 
until we reach the next set where an off-
diagonal entry disappears.  The next time it happens depends on how
far down in the matrix $M_{n_J+1}$
the row with the next off-diagonal entry appears.

\underline{Case 2a}
If $r$ is the second-largest vertex in an off-diagonal cycle, 
then we're in business.
\begin{multline*}
\overline{-I_{A_{n-r+1}}}\circ\phi_{n-r+1}\left(\left[\begin{smallmatrix}
&&-b_k&&&&&\\
&B_{j_r}&&&&&&\\
&&\ddots&&&&&\\
&&&-b_r&&&&\\
&&&&\ddots&&&\\
-\lambda&&&&&&&\\
&&&&&&\ddots&\\
&&&&&&&B_{j_n}
\end{smallmatrix}\right]\right)=\\
-\left[\begin{smallmatrix}
&&-b_k&&&&&\\
&B_{j_r}&&&&&&\\
&&\ddots&&&&&\\
&&&+b_r&&&&\\
&&&&\ddots&&&\\
-\lambda&&&&&&&\\
&&&&&&\ddots&\\
&&&&&&&B_{j_n}
\end{smallmatrix}\right]\in -A_{n-r+1}.
\end{multline*}
Now applications of involutions will switch the entries in row 0 and row $r$:
\begin{multline*}
\overline{-I_{A_{n-r+1}}}\circ\phi_{n-r}'\left(-\left[\begin{smallmatrix}
&&-b_k&&&&&\\
&B_{j_r}&&&&&&\\
&&\ddots&&&&&\\
&&&+b_r&&&&\\
&&&&\ddots&&&\\
-\lambda&&&&&&&\\
&&&&&&\ddots&\\
&&&&&&&B_{j_n}
\end{smallmatrix}\right]\right)=\\
\left[\begin{smallmatrix}
&&&-b_r&&&&\\
&B_{j_r}&&&&&&\\
&&\ddots&&&&&\\
&&+b_k&&&&&\\
&&&&\ddots&&&\\
-\lambda&&&&&&&\\
&&&&&&\ddots&\\
&&&&&&&B_{j_n}
\end{smallmatrix}\right]\in A_{n-r+1}.
\end{multline*}
But we still don't get to move on to a matrix set with a larger subscript:
\begin{multline*}
\overline{-I_{A_{n-r+1}}}\circ\phi_{n-r+1}\left(\left[\begin{smallmatrix}
&&&-b_r&&&&\\
&B_{j_r}&&&&&&\\
&&\ddots&&&&&\\
&&+b_k&&&&&\\
&&&&\ddots&&&\\
-\lambda&&&&&&&\\
&&&&&&\ddots&\\
&&&&&&&B_{j_n}
\end{smallmatrix}\right]\right)=\\-\left[\begin{smallmatrix}
&&&-b_r&&&&\\
&B_{j_r}&&&&&&\\
&&\ddots&&&&&\\
&&-b_k&&&&&\\
&&&&\ddots&&&\\
-\lambda&&&&&&&\\
&&&&&&\ddots&\\
&&&&&&&B_{j_n}
\end{smallmatrix}\right]\in -A_{n-r+1}
\end{multline*}
Note that now we are headed up (toward $A_0'$) again.  The next
interesting involution occurs when we again have an element of 
the set where $-\lambda$
first appears in the $J^{\rm th}$ row.
\begin{multline*}
\overline{-I_{A_{n-J+1}}}\circ\phi_{n-J}'\left(-\left[\begin{smallmatrix}
&&&-b_r&&&&\\
&B_{j_r}&&&&&&\\
&&\ddots&&&&&\\
&&-b_k&&&&&\\
&&&&\ddots&&&\\
-\lambda&&&&&&&\\
&&&&&&\ddots&\\
&&&&&&&B_{j_n}
\end{smallmatrix}\right]\right)=\\
\left[\begin{smallmatrix}
\lambda&&&&&&&\\
&B_{j_r}&&&&&&\\
&&\ddots&&&&&\\
&&-b_k&&&&&\\
&&&&\ddots&&&\\
&&&+b_r&&&&\\
&&&&&&\ddots&\\
&&&&&&&B_{j_n}
\end{smallmatrix}\right]\in A_{n-J+1}.
\end{multline*}
Now $b_r$ is in the $J^{\rm th}$ row.  After changing its sign we will
continuing applying involutions whose images are in sets with decreasing
subscripts:
\begin{multline*}
\overline{-I_{A_{n-J+1}}}\circ\phi_{n-J+1}\left(\left[\begin{smallmatrix}
\lambda&&&&&&&\\
&B_{j_r}&&&&&&\\
&&\ddots&&&&&\\
&&-b_k&&&&&\\
&&&&\ddots&&&\\
&&&+b_r&&&&\\
&&&&&&\ddots&\\
&&&&&&&B_{j_n}
\end{smallmatrix}\right]\right)=\\
-\left[\begin{smallmatrix}
\lambda&&&&&&&\\
&B_{j_r}&&&&&&\\
&&\ddots&&&&&\\
&&-b_k&&&&&\\
&&&&\ddots&&&\\
&&&-b_r&&&&\\
&&&&&&\ddots&\\
&&&&&&&B_{j_n}
\end{smallmatrix}\right]\in -A_{n-J+1}
\end{multline*}
and now we'll make it all the way back up to $A_0$ without interruption.
When we get there, we note that now the only difference in our graph is
that $\verb+succ+(J)=r$ instead of $k$, and $k$ is now in row $r$
so $\verb+succ+(r)=k$.  In fact, we have inserted $r$ after
the largest vertex in a cycle without changing anything else about the
graph--exactly what would've happened in the tree surgery method.  Since
all off-diagonal entries are now in the same cycle with $J$, we have
\begin{multline*}
\overline{-I_{A_0'}}\circ\phi_0\left(-\left[\begin{smallmatrix}
\lambda&&&&&&&\\
&B_{j_r}&&&&&&\\
&&\ddots&&&&&\\
&&-b_k&&&&&\\
&&&&\ddots&&&\\
&&&-b_r&&&&\\
&&&&&&\ddots&\\
&&&&&&&B_{j_n}
\end{smallmatrix}\right]\right)=\\
\left[\begin{smallmatrix}
\lambda&&&&&&&\\
&B_{j_r}&&&&&&\\
&&\ddots&&&&&\\
&&&B_k&&&&&\\
&&&&\ddots&&&\\
&&&&&B_r&&\\
&&&&&&\ddots&\\
&&&&&&&B_{j_n}
\end{smallmatrix}\right]\in A_0'.
\end{multline*}
By the induction hypothesis, from here (a graph where the path from
1 to 0 is of length $(i-1)$) we know that the two methods give the same
code.

\underline{Case 2b}
However, if $r$ is not the second-largest vertex in a cycle, the procedure
is a bit longer.  In general, Case 2 started with 
\[
\begin{bmatrix}
\lambda&&&&&&&\\
&B_{j_r}&&&&&&\\
&&\ddots&&&&&\\
&&&-b_r&&&&\\
&&&&\ddots&&&\\
&&-b_k&&&&&\\
&&&&&&\ddots&\\
&&&&&&&B_{j_n}
\end{bmatrix}\in A_0'.
\]
The result of applying the first bunch of involutions before we end up
back at $A_0'$ again is to switch the rows of the lowest (meaning their
row indices are largest) two off-diagonal 
entries.  Let $l$ be the second largest vertex in the cycle containing $J$,
and let $m=\verb+succ+(l)$.  
So our starting matrix actually looks something like this (although it
is possible that $\verb+succ+(l)=m=J$):
\[
\begin{bmatrix}
\lambda&&&&&&&\\
&B_{j_r}&&&&&&\\
&&\ddots&&&-b_l&&\\
&&&-b_r&&&&\\
&&&&\ddots&&-b_J&\\
&&-b_m&&&&&\\
&&&&&\ddots&&\\
&&&&-b_k&&&\\
&&&&&&\ddots&\\
&&&&&&&B_{j_n}
\end{bmatrix}\in A_0'.
\]
As before, we can get down to $A_{n-J+1}$ uneventfully, but then interesting
things happen:
\begin{multline*}
\overline{-I_{A_{n-J+1}}}\circ\phi_{n-J+1}\left(\left[\begin{smallmatrix}
\lambda&&&&&&&\\
&B_{j_r}&&&&&&\\
&&\ddots&&&-b_l&&\\
&&&-b_r&&&&\\
&&&&\ddots&&-b_J&\\
&&-b_m&&&&&\\
&&&&&\ddots&&\\
&&&&-b_k&&&\\
&&&&&&\ddots&\\
&&&&&&&B_{j_n}
\end{smallmatrix}\right]\right)=\\
-\left[\begin{smallmatrix}
\lambda&&&&&&&\\
&B_{j_r}&&&&&&\\
&&\ddots&&&-b_l&&\\
&&&-b_r&&&&\\
&&&&\ddots&&-b_J&\\
&&-b_m&&&&&\\
&&&&&\ddots&&\\
&&&&+b_k&&&\\
&&&&&&\ddots&\\
&&&&&&&B_{j_n}
\end{smallmatrix}\right]\in -A_{n-J+1}.
\end{multline*}
\begin{multline*}
\overline{-I_{A_{n-J+1}}}\circ\phi_{n-J}'\left(\left[\begin{smallmatrix}
\lambda&&&&&&&\\
&B_{j_r}&&&&&&\\
&&\ddots&&&-b_l&&\\
&&&-b_r&&&&\\
&&&&\ddots&&-b_J&\\
&&-b_m&&&&&\\
&&&&&\ddots&&\\
&&&&+b_k&&&\\
&&&&&&\ddots&\\
&&&&&&&B_{j_n}
\end{smallmatrix}\right]\right)=\\
\left[\begin{smallmatrix}
&&&&-b_k&&&\\
&B_{j_r}&&&&&&\\
&&\ddots&&&-b_l&&\\
&&&-b_r&&&&\\
&&&&\ddots&&-b_J&\\
&&-b_m&&&&&\\
&&&&&\ddots&&\\
-\lambda&&&&&&&\\
&&&&&&\ddots&\\
&&&&&&&B_{j_n}
\end{smallmatrix}\right]\in A_{n-J+1}.
\end{multline*}
Now we apply $\phi_{n-J+1}$, and have no further interruptions until we
reach $A_{n-l+1}$:
\begin{multline*}
\overline{-I_{A_{n-l+1}}}\circ\phi_{n-l+1}\left(\left[\begin{smallmatrix}
&&&&-b_k&&&\\
&B_{j_r}&&&&&&\\
&&\ddots&&&-b_l&&\\
&&&-b_r&&&&\\
&&&&\ddots&&-b_J&\\
&&-b_m&&&&&\\
&&&&&\ddots&&\\
-\lambda&&&&&&&\\
&&&&&&\ddots&\\
&&&&&&&B_{j_n}
\end{smallmatrix}\right]\right)=\\-\left[\begin{smallmatrix}
&&&&-b_k&&&\\
&B_{j_r}&&&&&&\\
&&\ddots&&&-b_l&&\\
&&&-b_r&&&&\\
&&&&\ddots&&-b_J&\\
&&+b_m&&&&&\\
&&&&&\ddots&&\\
-\lambda&&&&&&&\\
&&&&&&\ddots&\\
&&&&&&&B_{j_n}
\end{smallmatrix}\right]\in -A_{n-l+1}.
\end{multline*}
Here, the lowest off-diagonal entry in the matrix was $-b_m$ in row $l$, 
so it changed sign; now we apply $\phi_{n-l}'$.
\begin{multline*}
\overline{-I_{A_{n-l+1}}}\circ\phi_{n-l}'\left(-\left[\begin{smallmatrix}
&&&&-b_k&&&\\
&B_{j_r}&&&&&&\\
&&\ddots&&&-b_l&&\\
&&&-b_r&&&&\\
&&&&\ddots&&-b_J&\\
&&+b_m&&&&&\\
&&&&&\ddots&&\\
-\lambda&&&&&&&\\
&&&&&&\ddots&\\
&&&&&&&B_{j_n}
\end{smallmatrix}\right]\right)=\\
\left[\begin{smallmatrix}
&&-b_m&&&&&\\
&B_{j_r}&&&&&&\\
&&\ddots&&&-b_l&&\\
&&&-b_r&&&&\\
&&&&\ddots&&-b_J&\\
&&&&+b_k&&&\\
&&&&&\ddots&&\\
-\lambda&&&&&&&\\
&&&&&&\ddots&\\
&&&&&&&B_{j_n}
\end{smallmatrix}\right]\in A_{n-l+1}.
\end{multline*}
Positive off-diagonal entries never survive.  We apply $\phi_{n-l+1}$: 
\begin{multline*}
\overline{-I_{A_{n-l+1}}}\circ\phi_{n-l+1}\left(\left[\begin{smallmatrix}
&&-b_m&&&&&\\
&B_{j_r}&&&&&&\\
&&\ddots&&&-b_l&&\\
&&&-b_r&&&&\\
&&&&\ddots&&-b_J&\\
&&&&+b_k&&&\\
&&&&&\ddots&&\\
-\lambda&&&&&&&\\
&&&&&&\ddots&\\
&&&&&&&B_{j_n}
\end{smallmatrix}\right]\right)=\\-\left[\begin{smallmatrix}
&&-b_m&&&&&\\
&B_{j_r}&&&&&&\\
&&\ddots&&&-b_l&&\\
&&&-b_r&&&&\\
&&&&\ddots&&-b_J&\\
&&&&-b_k&&&\\
&&&&&\ddots&&\\
-\lambda&&&&&&&\\
&&&&&&\ddots&\\
&&&&&&&B_{j_n}
\end{smallmatrix}\right]\in -A_{n-l+1}.
\end{multline*}
This array will take us back up to $A_{n-J+1}$ (this should remind you of 
what happened in Case 2a).
\begin{multline*}
\overline{-I_{A_{n-J+1}}}\circ\phi_{n-J}'\left(-\left[\begin{smallmatrix}
&&-b_m&&&&&\\
&B_{j_r}&&&&&&\\
&&\ddots&&&-b_l&&\\
&&&-b_r&&&&\\
&&&&\ddots&&-b_J&\\
&&&&-b_k&&&\\
&&&&&\ddots&&\\
-\lambda&&&&&&&\\
&&&&&&\ddots&\\
&&&&&&&B_{j_n}
\end{smallmatrix}\right]\right)=\\
\left[\begin{smallmatrix}
\lambda&&&&&&&\\
&B_{j_r}&&&&&&\\
&&\ddots&&&-b_l&&\\
&&&-b_r&&&&\\
&&&&\ddots&&-b_J&\\
&&&&-b_k&&&\\
&&&&&\ddots&&\\
&&+b_m&&&&&\\
&&&&&&\ddots&\\
&&&&&&&B_{j_n}
\end{smallmatrix}\right]\in A_{n-J+1}.
\end{multline*}
And our last little side trip:
\begin{multline*}
\overline{-I_{A_{n-J+1}}}\circ\phi_{n-J+1}\left(\left[\begin{smallmatrix}
\lambda&&&&&&&\\
&B_{j_r}&&&&&&\\
&&\ddots&&&-b_l&&\\
&&&-b_r&&&&\\
&&&&\ddots&&-b_J&\\
&&&&-b_k&&&\\
&&&&&\ddots&&\\
&&+b_m&&&&&\\
&&&&&&\ddots&\\
&&&&&&&B_{j_n}
\end{smallmatrix}\right]\right)=\\
-\left[\begin{smallmatrix}
\lambda&&&&&&&\\
&B_{j_r}&&&&&&\\
&&\ddots&&&-b_l&&\\
&&&-b_r&&&&\\
&&&&\ddots&&-b_J&\\
&&&&-b_k&&&\\
&&&&&\ddots&&\\
&&-b_m&&&&&\\
&&&&&&\ddots&\\
&&&&&&&B_{j_n}
\end{smallmatrix}\right]\in -A_{n-J+1}.
\end{multline*}
This last matrix appears in all of the previous sets.
Thus we pass through
a number of sets, finally arriving at
\[-
\begin{bmatrix}
\lambda&&&&&&&\\
&B_{j_r}&&&&&&\\
&&\ddots&&&-b_l&&\\
&&&-b_r&&&&\\
&&&&\ddots&&-b_J&\\
&&&&-b_k&&&\\
&&&&&\ddots&&\\
&&-b_m&&&&&\\
&&&&&&\ddots&\\
&&&&&&&B_{j_n}
\end{bmatrix}\in -A_0'.
\]
This is where we apply $\phi_0$.  Unfortunately, this time we are not
as lucky as in Case 2a, where everything moved back on diagonal.  Note
that this new matrix corresponds to a graph that differs from the one
at the start of Case 2 by only 2 edges--namely, we have switched the
successors of $J$ and $l$, the two largest vertices in the cycle containing
$J$.  Necessarily we now have three cycles; everything in between $J$ and
$l$ has been shorted out and forms its own cycle.  Now when we apply $\phi_0$,
we move the cycle containing $J$ onto the diagonal.  This corresponds to
the cycle containing $J$ being considered inactive.  
\begin{multline*}
\overline{-I_{A_0'}}\circ\phi_0\left(-\left[\begin{smallmatrix}
\lambda&&&&&&&\\
&B_{j_r}&&&&&&\\
&&\ddots&&&-b_l&&\\
&&&-b_r&&&&\\
&&&&\ddots&&-b_J&\\
&&&&-b_k&&&\\
&&&&&\ddots&&\\
&&-b_m&&&&&\\
&&&&&&\ddots&\\
&&&&&&&B_{j_n}
\end{smallmatrix}\right]\right)=\\
\left[\begin{smallmatrix}
\lambda&&&&&&&&\\
&B_{j_r}&&&&&&&\\
&&\ddots&&&&&-b_l&&&\\
&&&-b_r&&&&\\
&&&&\ddots&&&&\\
&&&&&B_J&&&\\
&&&&&&\ddots&&\\
&&&&-b_k&&&\\
&&&&&&&\ddots&&\\
&&&&&&&&B_m&&\\
&&&&&&&&&\ddots&\\
&&&&&&&&&&B_{j_n}
\end{smallmatrix}\right]\in A_0'
\end{multline*}
The cycle containing $J$ is no longer off the diagonal, so nothing will
happen to it as we move down the sequence of matrices (hence the notion
of it being inactive).  If we let $p$ be
the largest vertex in a cycle that appears off the diagonal at this stage, 
and $q$ be the 
second-largest, then
essentially the same procedure we just finished will be duplicated, only
with $p$ as the lowest row with an off-diagonal entry.  Each time we do
this, the lowest two off-diagonal entries in the array are switched, and
we toggle the diagonality (``activity'') of the cycle containing $J$ (which
will change as we go).  But the effect of this switching in the matrix is
the trading of successors for $p$ and $q$ at each stage, and thus this 
matrix process is equivalent to the graph surgery from Lemma \ref{Escher}

Now we appeal to the lemma.  The effect of this huge process is to insert
$r$ after $J$ in the cycle containing $J$.  Furthermore, since $r$ has
been removed from the path joining 1 to 0, the new happy functional digraph
has a shorter path.  Thus, by induction, the tree surgery method and
the matrix method give the same code.
\end{proof}

\chapter{The Dandelion Code}\label{dandy}

The method for this code
is sort of a m\'elange of the methods of the Happy Code and the
Blob Code.  As we did for the Blob Code, we consider the $n\times n$ submatrix 
obtained from $\hat{D}-\hat{A}$ by crossing out the zeroth row and column
and apply the Matrix Tree Theorem at
every possible opportunity.  However, following the method of the Happy Code,
we only do row operations and we always subtract the top row.  We will again
use $B$ to denote $\sum_0^n B_j$.

The matrix we start with, with rows and columns indexed from 1 to $n$, is
\[
N_0'=\begin{bmatrix}
B-b_1 & -b_2 & \dots & -b_n\\
-b_1 & B-b_2 & \dots & -b_n\\
\vdots & \vdots & \ddots & \vdots\\
-b_1 & -b_2 & \dots & B-b_n
\end{bmatrix}.
\]
We will subtract the first row from each of the other rows in turn, in the
usual way:  from the bottom up, with cancellation being a separate step.
\[
N_1=\begin{bmatrix}
B-b_1 & -b_2 & \dots & -b_n\\
-b_1 & B-b_2 & \dots & -b_n\\
\vdots & \vdots & \ddots & \vdots\\
-b_1-B+b_1 & -b_2+b_2 & \dots & B-b_n+b_n
\end{bmatrix},
\]
and we cancel only terms in the $n^{\rm th}$ row at this point.
\[
N_1'=\begin{bmatrix}
B-b_1 & -b_2 & \dots & -b_n\\
-b_1 & B-b_2 & \dots & -b_n\\
\vdots & \vdots & \ddots & \vdots\\
-B & 0 & \dots & B
\end{bmatrix}.
\]
Next we subtract the first row from the $(n-1)^{\rm th}$ row, and
we continue; at step $i$, we subtract row 1 from row $(n-i+1)$, until 
we reach the last matrix:
\[
N_{n-1}'=\begin{bmatrix}
B-b_1 & -b_2 & -b_3 &\dots & -b_n\\
-B & B & 0 & \dots & 0\\
\vdots & \vdots &\vdots & \ddots & \vdots\\
-B & 0 & 0 &\dots & B
\end{bmatrix}.
\]
It may not be clear by inspection what $\det N_{n-1}'$ is, but we already
know the answer because of the Matrix Tree Theorem.

\section{The Sets}
The sequence of sets is $F_0,F_0',D_1,D_1',F_1,\dots,F_{n-1}'$.  $F_0$
is the set of trees in the original graph; $F_0'$ is the set of arrays
from $N_0'$.  For $1\leq i\leq n-1$, 
the set $F_i$ is the set of spanning trees in an 
altered graph.  The altered graph at step $i$ has the same edges out of
$1,2,\dots,n-i$ as the original graph, and each of the 
vertices $n-i+1,\dots,n$ has multiple edges pointing to $1$ with 
certain weights, but no edges to any other vertex.  Specifically, at step $i$,
we replace the edge $n-i+1\to j$ in the graph at step $i$
by an edge $n-i+1\to 1$ with weight $B_j$, for
each $j$.  After all, if we apply the Matrix Tree Theorem to $N_1'$ (for
example), we see that row $n$ represents edges $n\to j$ and has mostly
zeroes, implying that there is no edge from $n$ to any vertex besides $1$.
If an off-diagonal entry is a sum, it
corresponds to multiple edges, each with monomial weight.  So for
$1\leq i\leq n-1$, $F_i$ is the
set of spanning trees in the altered 
graph corresponding to the matrix at step $N_i'$.
$D_i$ is the set of arrays from $N_i$ and $D_i'$ is the
set of arrays from $N_i'$.  For $1\leq i\leq n-2$,
$F_i'=D_i'$.  Finally, $F_{n-1}'$ is the set of codes.

Note that the last graph, whose spanning trees make up $F_{n-1}$,
has $n+1$ monomial-weighted edges of the form 
$k\to 1$ for each $k=2,\dots,n$, and one edge $1\to k$
for each $k=0,1,2,\dots,n$.  However, since any spanning tree rooted at $0$
must contain an edge into $0$, we know that the only possible edge out of
$1$ that can occur in a spanning tree is the edge $1\to 0$, so that $F_{n-1}$
can also be thought of as the set of spanning trees of the graph below.

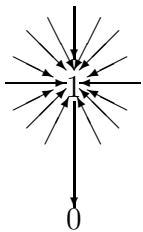
\begin{picture}(100,100)
\put(53,48){\vector(0,-1){40}}
\put(50,0){0}
\put(50,50){1}
\put(35,38){\vector(1,1){15}}
\put(27,55){\vector(1,0){23}}
\put(35,75){\vector(1,-1){15}}
\put(53,84){\vector(0,-1){20}}
\put(70,75){\vector(-1,-1){15}}
\put(78,55){\vector(-1,0){23}}
\put(70,38){\vector(-1,1){15}}
\put(30,45){\vector(2,1){15}}
\put(30,65){\vector(2,-1){15}}
\put(43,80){\vector(1,-2){10}}
\put(63,80){\vector(-1,-2){10}}
\put(73,65){\vector(-2,-1){15}}
\put(73,45){\vector(-2,1){15}}
\put(42,32){\vector(1,2){9}}
\put(63,32){\vector(-1,2){9}}
\end{picture}

\vspace{5pt}\noindent
This picture should enlighten the reader as to the name for this Code.

\section{The Involutions}
The involutions are defined very similarly to the ones for the Blob Code.  

As usual, the first involution,
$\mu_0':F_0-F_0'\to F_0-F_0'$, takes each tree in the original graph to 
the corresponding array in $N_0'$, and pairs up the extra arrays according
to toggling the diagonality of the cycle containing the greatest element.

For $1\leq i\leq n-1$,
$\mu_i:D_i'-F_i\to D_i'-F_i$ is the involution of the bijective proof of the
Matrix Tree Theorem, which
matches each positive array from $N_i'$ (that is, each element
of $(D_i')^+$) to a negative tree in $-F_i$.  
Meanwhile, for $0\leq i \leq n-2$, 
$\mu_i':F_i-F_i'\to F_i-F_i'$ is essentially
the negative of map $\mu_i$; it matches trees and arrays 
in the same way but with opposite signs.

For $1\leq i\leq n-1$,
$\xi_i:F_{i-1}'-D_i\to F_{i-1}'-D_i$ is the involution that matches  
arrays to one another
according to the row operation.  Thus if $a\in -D_i$ and the entry
in row $n-i+1$ is $+b_j$ or $-B_j$, then $\xi_i(a)=a'\in -D_i$ where $a'$
is obtained by interchanging and negating rows 1 and $n-i+1$.  For all 
other $a\in F_{i-1}'-D_i$, $\xi_i(a)=-a$ (in $-D_i$ if $a\in F_{n-1}'$ and
vice versa).

For $1\leq i\leq n-1$,
$\xi_i':D_i-D_i'\to D_i-D_i'$ is the involution that matches up arrays 
according to the arithmetic within entries in row $n-i+1$.  If $a\in D_i$
and the entry in row $n-i+1$ is $\pm b_j$, then $\xi_i'(a)=a'\in D_i$
where $a'$ is obtained from $a$ by changing the sign of the entry in row
$n-i+1$; for all other $a$, $\xi_i'(a)=-a$ (in $D_i$ if $a\in -D_i'$ and
vice versa).

The final involution $\hat{\mu}_{n-1}:F_{n-1}-F_{n-1}'\to F_{n-1}-F_{n-1}'$ 
matches trees to codes.  The code for a tree is given by the weights of the
outgoing edges from vertices $2,3,\dots,n$ in order.  Thus if the weight
of the edge $i\to 1$ in the tree $\tau$ is $w_i$ for each $i=2,3,\dots n$, 
then $\hat{\mu}_{n-1}(\tau)=(w_2,w_3,\dots,w_n)$.

\section{How to Find the Dandelion Code}
\begin{theorem}
Given the sets $F_0,F_0',D_1,D_1',F_1,\dots,F_{n-1}'$ 
and the sign-reversing involutions
$\mu_0',\xi_1,\xi_1',\mu_1,\mu_1',\dots,\xi_{n-1}',\hat{\mu}_{n-1}$,
we can construct the
bijection between $F_0$ (trees in our original graph) and $D_n$ (codes).
\end{theorem}
\begin{proof}
Again, our sets and involutions satisfy the hypotheses of Lemma
\ref{general}, so we can construct the bijection.
\end{proof}
\subsection{An example}
For $n=4$, consider the tree $1\to 3\to 4\to 2\to 0\in F_0$.  
First, the Matrix Tree Theorem tells us what array
corresponds to this tree.
\[
\mu_0'(1\to 3\to 4\to 2\to 0)=-\left[\begin{smallmatrix}
B_3 &&&\\&B_0&&\\&&B_4&\\&&&B_2\end{smallmatrix}\right]\in -F_0'.
\]
And the obligatory negative identity map:
\[
\overline{-I_{F_0'}}\left(-\left[\begin{smallmatrix}
B_3 &&&\\&B_0&&\\&&B_4&\\&&&B_2\end{smallmatrix}\right]\right)=
\left[\begin{smallmatrix}
B_3 &&&\\&B_0&&\\&&B_4&\\&&&B_2\end{smallmatrix}\right]\in F_0'.
\]
By now this is child's play.
\[
\overline{-I_{D_1}}\circ\xi_1\left(\left[\begin{smallmatrix}
B_3 &&&\\&B_0&&\\&&B_4&\\&&&B_2\end{smallmatrix}\right]\right)=
\left[\begin{smallmatrix}
B_3 &&&\\&B_0&&\\&&B_4&\\&&&B_2\end{smallmatrix}\right]\in D_1.
\]
\[
\overline{-I_{D_1'}}\circ\xi_1'\left(\left[\begin{smallmatrix}
B_3 &&&\\&B_0&&\\&&B_4&\\&&&B_2\end{smallmatrix}\right]\right)=
\left[\begin{smallmatrix}
B_3 &&&\\&B_0&&\\&&B_4&\\&&&B_2\end{smallmatrix}\right]\in D_1'.
\]
Now it gets slightly tricky.  This array does not correspond to
a tree in the graph where $4$ only has edges pointing at $1$, because
it represents the following functional digraph:

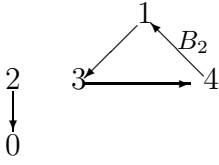
\begin{picture}(100,75)
\put(0,0){0}
\put(0,25){2}
\put(3,24){\vector(0,-1){15}}
\put(25,25){3}
\put(75,25){4}
\put(50,50){1}
\put(50,49){\vector(-1,-1){20}}%1->3
\put(30,27){\vector(1,0){40}}%3->4
\put(75,30){\vector(-1,1){20}}%4->1
\put(65,40){\footnotesize{$B_2$}}
\end{picture}

\noindent
The next step is to apply $\mu_1$ to the array; the cycle 
gets moved off the diagonal:
\[
\overline{-I_{D_1'}}\circ\mu_1\left(\left[\begin{smallmatrix}
B_3 &&&\\&B_0&&\\&&B_4&\\&&&B_2\end{smallmatrix}\right]\right)=
-\left[\begin{smallmatrix} &&-b_3&\\&B_0&&\\&&&-b_4\\-B_2&&&
\end{smallmatrix}\right]\in -D_1'
\]
\[
\overline{-I_{D_1}}\circ\xi_1'\left(-\left[\begin{smallmatrix}
&&-b_3&\\&B_0&&\\&&&-b_4\\-B_2&&&
\end{smallmatrix}\right]\right)=
-\left[\begin{smallmatrix} 
&&-b_3&\\&B_0&&\\&&&-b_4\\-B_2&&&
\end{smallmatrix}\right]\in -D_1
\]
\[
\overline{-I_{D_1}}\circ\xi_1\left(-\left[\begin{smallmatrix}
&&-b_3&\\&B_0&&\\&&&-b_4\\-B_2&&&
\end{smallmatrix}\right]\right)=
\left[\begin{smallmatrix} B_2&&&\\&B_0&&\\&&&-b_4\\&&+b_3&
\end{smallmatrix}\right]\in D_1
\]
\[
\overline{-I_{D_1}}\circ\xi_1'\left(\left[
\begin{smallmatrix} B_2&&&\\&B_0&&\\&&&-b_4\\&&+b_3&
\end{smallmatrix}\right]\right)=-\left[
\begin{smallmatrix} B_2&&&\\&B_0&&\\&&&-b_4\\&&-b_3&
\end{smallmatrix}\right]\in -D_1
\]
\[
\overline{-I_{F_0'}}\circ\xi_1\left(-\left[
\begin{smallmatrix} B_2&&&\\&B_0&&\\&&&-b_4\\&&-b_3&
\end{smallmatrix}\right]\right)=-\left[
\begin{smallmatrix} B_2&&&\\&B_0&&\\&&&-b_4\\&&-b_3&
\end{smallmatrix}\right]\in -F_0'.
\]
Since we still have a cycle, the effect of $\mu_0'$ will be to put
it back on the diagonal.
\[
\overline{-I_{F_0'}}\circ\mu_0'\left(-\left[
\begin{smallmatrix} B_2&&&\\&B_0&&\\&&&-b_4\\&&-b_3&
\end{smallmatrix}\right]\right)=\left[
\begin{smallmatrix} B_2&&&\\&B_0&&\\&&B_4&\\&&&B_3
\end{smallmatrix}\right]\in F_0'
\]
\[
\overline{-I_{D_1}}\circ\xi_1\left(-\left[
\begin{smallmatrix} B_2&&&\\&B_0&&\\&&B_4&\\&&&B_3
\end{smallmatrix}\right]\right)=\left[
\begin{smallmatrix} B_2&&&\\&B_0&&\\&&B_4&\\&&&B_3
\end{smallmatrix}\right]\in D_1
\]
\[
\overline{-I_{D_1'}}\circ\xi_1'\left(-\left[
\begin{smallmatrix} B_2&&&\\&B_0&&\\&&B_4&\\&&&B_3
\end{smallmatrix}\right]\right)=\left[
\begin{smallmatrix} B_2&&&\\&B_0&&\\&&B_4&\\&&&B_3
\end{smallmatrix}\right]\in D_1'
\]
Now $\mu_1$ will give us a tree:

\begin{picture}(300,100)
\put(50,50){$\overline{-I_{F_1}}\circ\mu_1\left(\left[
\begin{smallmatrix} B_2&&&\\&B_0&&\\&&B_4&\\&&&B_3
\end{smallmatrix}\right]\right)=$}
\put(200,75){3}
\put(205,80){\vector(1,0){20}}
\put(225,75){4}
\put(230,80){\vector(1,0){20}}
\put(250,75){1}
\put(235,85){\footnotesize{$B_3$}}
\put(253,74){\vector(0,-1){15}}
\put(250,50){2}
\put(253,49){\vector(0,-1){15}}
\put(250,25){0}
\end{picture}

\noindent
When we apply $\overline{-I_{F_1'}}\circ\mu_1'$ to this, 
we get back the same array we left in $D_1'$, only now we are in $F_1'$.  
We continue the same process.
\[
\overline{-I_{D_2}}\circ\xi_2\left(\left[
\begin{smallmatrix} B_2&&&\\&B_0&&\\&&B_4&\\&&&B_3
\end{smallmatrix}\right]\right)=\left[
\begin{smallmatrix} B_2&&&\\&B_0&&\\&&B_4&\\&&&B_3
\end{smallmatrix}\right]\in D_2;
\]
\[
\overline{-I_{D_2'}}\circ\xi_2'\left(\left[
\begin{smallmatrix} B_2&&&\\&B_0&&\\&&B_4&\\&&&B_3
\end{smallmatrix}\right]\right)=\left[
\begin{smallmatrix} B_2&&&\\&B_0&&\\&&B_4&\\&&&B_3
\end{smallmatrix}\right]\in D_2'.
\]
Once again we use the Matrix Tree Theorem, this time in the form of $\mu_2$,
to find out if we have a tree in the digraph where 3 and 4 have multiple
outgoing edges to 1.

\begin{picture}(300,100)
\put(60,50){$\overline{-I_{F_2}}\circ\mu_2\left(\left[
\begin{smallmatrix} B_2&&&\\&B_0&&\\&&B_4&\\&&&B_3\end{smallmatrix}\right]
\right)=$}
\put(250,0){0}
\put(250,25){2}
\put(253,24){\vector(0,-1){15}}
\put(250,50){1}
\put(253,49){\vector(0,-1){15}}
\put(225,50){4}
\put(230,54){\vector(1,0){20}}
\put(233,45){\footnotesize{$B_3$}}
\put(250,75){3}
\put(253,74){\vector(0,-1){15}}
\put(255,62){\footnotesize{$B_4$}}
\end{picture}

\noindent
We do, so we move on. $\mu_2'$ followed by $\overline{-I_{F_2'}}$ gives
us the same array we left behind before reaching that tree. 
\[
\overline{-I_{D_3}}\circ\xi_3\left(\left[
\begin{smallmatrix} B_2&&&\\&B_0&&\\&&B_4&\\&&&B_3
\end{smallmatrix}\right]\right)=\left[
\begin{smallmatrix} B_2&&&\\&B_0&&\\&&B_4&\\&&&B_3
\end{smallmatrix}\right]\in D_3;
\]
\[
\overline{-I_{D_3'}}\circ\xi_3'\left(\left[
\begin{smallmatrix} B_2&&&\\&B_0&&\\&&B_4&\\&&&B_3
\end{smallmatrix}\right]\right)=\left[
\begin{smallmatrix} B_2&&&\\&B_0&&\\&&B_4&\\&&&B_3
\end{smallmatrix}\right]\in D_3'.
\]
Now we run into trouble again.  This array corresponds to a graph with
a cycle between 1 and 2, so $\mu_3$ has the following effect:
\[
\overline{-I_{D_3'}}\circ\mu_3\left(\left[
\begin{smallmatrix} B_2&&&\\&B_0&&\\&&B_4&\\&&&B_3
\end{smallmatrix}\right]\right)=-\left[
\begin{smallmatrix} &-b_2&&\\-B_0&&&\\&&B_4&\\&&&B_3
\end{smallmatrix}\right]\in -D_3'.
\]
\[
\overline{-I_{D_3}}\circ\xi_3'\left(-\left[
\begin{smallmatrix} &-b_2&&\\-B_0&&&\\&&B_4&\\&&&B_3
\end{smallmatrix}\right]\right)=-\left[
\begin{smallmatrix} &-b_2&&\\-B_0&&&\\&&B_4&\\&&&B_3
\end{smallmatrix}\right]\in -D_3
\]
\[
\overline{-I_{D_3}}\circ\xi_3\left(-\left[
\begin{smallmatrix} &-b_2&&\\-B_0&&&\\&&B_4&\\&&&B_3
\end{smallmatrix}\right]\right)=\left[
\begin{smallmatrix} B_0&&&\\&+b_2&&\\&&B_4&\\&&&B_3
\end{smallmatrix}\right]\in D_3
\]
\[
\overline{-I_{D_3}}\circ\xi_3'\left(\left[
\begin{smallmatrix} B_0&&&\\&+b_2&&\\&&B_4&\\&&&B_3
\end{smallmatrix}\right]\right)=-\left[
\begin{smallmatrix} B_0&&&\\&-b_2&&\\&&B_4&\\&&&B_3
\end{smallmatrix}\right]\in -D_3
\]
This does not correspond to a tree in the graph where 3 and 4 have 
edges only to one, because there is a loop at the vertex 2.
\[
\overline{-I_{D_3}}\circ\xi_3\left(-\left[
\begin{smallmatrix} B_0&&&\\&-b_2&&\\&&B_4&\\&&&B_3
\end{smallmatrix}\right]\right)=\left[
\begin{smallmatrix} B_0&&&\\&B_2&&\\&&B_4&\\&&&B_3
\end{smallmatrix}\right]\in D_3,
\]
\[
\overline{-I_{D_3'}}\circ\xi_3'\left(\left[
\begin{smallmatrix} B_0&&&\\&B_2&&\\&&B_4&\\&&&B_3
\end{smallmatrix}\right]\right)=\left[
\begin{smallmatrix} B_0&&&\\&B_2&&\\&&B_4&\\&&&B_3
\end{smallmatrix}\right]\in D_3',
\]
and finally

\begin{picture}(300,100)
\put(0,50){$\overline{-I_{F_3}}\circ\mu_3\left(\left[
\begin{smallmatrix} B_0&&&\\&B_2&&\\&&B_4&\\&&&B_3
\end{smallmatrix}\right]\right)=$}
\put(175,0){0}
\put(175,50){1}
\put(178,49){\vector(0,-1){38}}
\put(150,50){4}
\put(155,54){\vector(1,0){20}}
\put(158,45){\footnotesize{$B_3$}}
\put(175,75){3}
\put(178,74){\vector(0,-1){15}}
\put(180,62){\footnotesize{$B_4$}}
\put(200,50){2}
\put(200,54){\vector(-1,0){20}}
\put(185,45){\footnotesize{$B_2$}}
\end{picture}

\noindent
At this point we can read the code off from the picture by looking at
the weights of the edges coming out of vertices $2,3,4$ in order.
$\overline{-I_{F_{n-1}'}}\circ\hat{\mu}_{n-1}$ of this tree is 
its dandelion code: $(B_2,B_4,B_3)\equiv(2,4,3)$.
Note that although our ending tree looks different from the original tree,
its total weight is equal to the weight of the original tree.

\section{Tree Surgery Method}
The same code can be found by skipping the matrix steps in between, since
we can predict their effect.

The plan is this:  We take the tree, and at step $i$ we remove the edge
$n-i+1\to\verb+succ+(n-i+1)$ 
and instead put in an edge $n-i+1\to 1$ with weight
$B_{\verb+succ+(n-i+1)}$.  If no cycle is created in the process, then we
move on to the next step.  If there is a cycle, we have to do something
about it:  we remove the edges $1\to\verb+succ+(1)$ and $n-i+1\to 1$ and
replace them by edges $1\to\verb+succ+(n-i+1)$ and $n-i+1\to 1$, this last edge
having weight $B_{\verb+succ+(1)}$.  At the end, we read off a version of
the na\"\i ve code from the vertices $2,\dots,n$ (instead of the successors
of each vertex (since each points at 1 now), we look at the weights of these
edges).  The algorithm takes as its input a tree as a set of edges.   

\begin{tabbing}
Tree Surgery Method for Dandelion Code\\
{\bf begin}\=\\
\>{\bf for} \= $i=1$ \= {\bf to} $n-1$ {\bf do}\\
\>\>$m\leftarrow\verb+succ+(n-i+1)$\\
\>\>$k\leftarrow\verb+succ+(1)$\\
\>\>remove edge $(n-i+1)\to m$\\
\>\>add edge $(n-i+1)\to 1$ with weight $B_m$\\
\>\>{\bf if} a cycle has been created {\bf then}\\
\>\>\>remove edge $1\to k$\\
\>\>\>remove edge $(n-i+1)\to 1$\\
\>\>\>add edge $1\to m$ \\
\>\>\>add edge $(n-i+1)\to 1$ with weight $B_k$\\
\>{\bf for} $j=2$ {\bf to} $n$ {\bf do}\\
\>\>$w_j\leftarrow$ the weight of the edge $j\to 1$\\
\>$\verb+code+\leftarrow(w_2,w_3,\dots,w_n)$\\
{\bf end.}
\end{tabbing}

In section \S\ref{joy} we will discuss 
the relationship between the Dandelion Code and
Joyal's proof of the formula for the number of labelled trees in \cite{J}
as well as the bijection in \cite{E-R}.
In fact this algorithm turns out to differ only in notation from one given by 
E{\u{g}}ecio{\u{g}}lu and Remmel in \cite{E-R}.  What is beautiful is the
fact that the matrix method and the tree surgery method result in this
same bijection.  Using our method, we can see
the underlying relationship of the tree surgical bijection 
with linear algebra and the Matrix Tree
Theorem. 

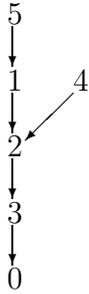
\begin{ex}

\begin{picture}(25,120)
\put(0,100){5}
\put(2,99){\vector(0,-1){15}}
\put(0,75){1}
\put(2,74){\vector(0,-1){15}}
\put(25,75){4}
\put(25,74){\vector(-1,-1){18}}
\put(0,50){2}
\put(2,49){\vector(0,-1){15}}
\put(0,25){3}
\put(2,24){\vector(0,-1){15}}
\put(0,0){0}
\end{picture}

\noindent  
The first step is to remove the edge $5\to 1$ and replace it by an
edge $5\to 1$ of weight $B_1$.  This is a bit redundant.  The point is
that whatever the successor of 5 is becomes the subscript of the weight
of the edge $5\to 1$.

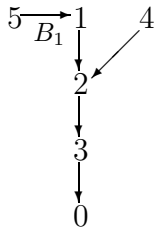
\begin{picture}(50,100)
\put(0,75){5}
\put(5,80){\vector(1,0){19}}
\put(10,70){\footnotesize{$B_1$}}
\put(25,75){1}
\put(27,74){\vector(0,-1){15}}
\put(50,75){4}
\put(50,74){\vector(-1,-1){18}}
\put(25,50){2}
\put(27,49){\vector(0,-1){15}}
\put(25,25){3}
\put(27,24){\vector(0,-1){15}}
\put(25,0){0}
\end{picture}

\noindent
The next step removes the edge $4\to 2$ and replaces it by an edge $4\to 1$
with weight $B_2$.  This does not create a cycle, so this is another quick
step.

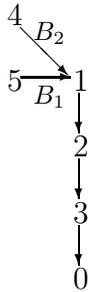
\begin{picture}(50,120)
\put(0,75){5}
\put(5,80){\vector(1,0){19}}
\put(10,70){\footnotesize{$B_1$}}
\put(25,75){1}
\put(27,74){\vector(0,-1){15}}
\put(0,100){4}
\put(5,99){\vector(1,-1){18}}
\put(10,94){\footnotesize{$B_2$}}
\put(25,50){2}
\put(27,49){\vector(0,-1){15}}
\put(25,25){3}
\put(27,24){\vector(0,-1){15}}
\put(25,0){0}
\end{picture}

\noindent
Now we remove the edge $3\to 0$ and replace it by an edge $3\to 1$ of 
weight $B_0$.

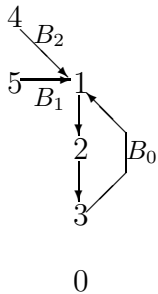
\begin{picture}(50,120)
\put(0,75){5}
\put(5,80){\vector(1,0){19}}
\put(10,70){\footnotesize{$B_1$}}
\put(25,75){1}
\put(27,74){\vector(0,-1){15}}
\put(0,100){4}
\put(5,99){\vector(1,-1){18}}
\put(10,94){\footnotesize{$B_2$}}
\put(25,50){2}
\put(27,49){\vector(0,-1){15}}
\put(25,25){3}
\put(30,30){\line(1,1){15}}
\put(45,45){\line(0,1){15}}
\put(45,60){\vector(-1,1){15}}
\put(45,50){\footnotesize{$B_0$}}
\put(25,0){0}
\end{picture}

\noindent
We have created a cycle, so we'd better fix it.  We remove the edge $1\to 2$
and replace it by the edge $1\to 0$, and replace the edge $3\to 1$ with weight
$B_0$ by an edge $3\to 1$ with weight $B_2$.  

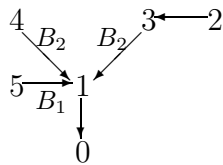
\begin{picture}(75,70)
\put(0,25){5}
\put(5,30){\vector(1,0){19}}
\put(10,20){\footnotesize{$B_1$}}
\put(25,25){1}
\put(27,24){\vector(0,-1){15}}
\put(0,50){4}
\put(5,49){\vector(1,-1){18}}
\put(10,44){\footnotesize{$B_2$}}
\put(25,0){0}
\put(50,50){3}
\put(50,49){\vector(-1,-1){18}}
\put(33,44){\footnotesize{$B_2$}}
\put(75,50){2}
\put(75,55){\vector(-1,0){20}}
\end{picture}

\noindent
The last step is to replace the edge $2\to 3$ by an edge $2\to 1$ of 
weight $B_3$.

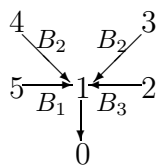
\begin{picture}(75,70)
\put(0,25){5}
\put(5,30){\vector(1,0){19}}
\put(10,20){\footnotesize{$B_1$}}
\put(25,25){1}
\put(27,24){\vector(0,-1){15}}
\put(0,50){4}
\put(5,49){\vector(1,-1){18}}
\put(10,44){\footnotesize{$B_2$}}
\put(25,0){0}
\put(50,50){3}
\put(50,49){\vector(-1,-1){18}}
\put(33,44){\footnotesize{$B_2$}}
\put(50,25){2}
\put(50,30){\vector(-1,0){20}}
\put(33,20){\footnotesize{$B_3$}}
\end{picture}

\noindent
Now we look at the weights.  The code is $(B_3,B_2,B_2,B_1)$.
\end{ex}

The inverse algorithm is fairly self-explanatory.  It takes a code 
$(c_1, c_2,\dots, c_n)$ and finds the corresponding tree.

\begin{tabbing}
Algorithm to go from Dandelion Code to Tree\\
{\bf begin}\=\\
\>$\verb+edges+\leftarrow\{1\to 0\}$\\
\>{\bf for} \= $i=2$ {\bf to} $n$ {\bf do}\\
\>\>add edge $i\to 1$ of weight $c_{i-1}$\\
\>{\bf for} $i=2$ {\bf to} $n$ {\bf do}\\
\>\>$k\leftarrow$ the subscript of the weight of the edge $i\to 1$\\
\>\>remove edge $i\to 1$\\
\>\>add edge $i\to k$\\
\>\>{\bf if} \= $\verb+cycles+\neq\emptyset$ {\bf then}\\
\>\>\>$m\leftarrow\verb+succ+(1)$\\
\>\>\>remove edge $i\to k$\\
\>\>\>add edge $1\to k$\\
\>\>\>remove edge $1\to m$\\
\>\>\>add edge $i\to m$\\
{\bf end.}
\end{tabbing}

\section{The Two Methods Give the Dandelion Code}
\begin{theorem}
The tree surgery method gives the same Dandelion Code as the matrix method.
\end{theorem}
\begin{proof}
Again, we assume constant $n$ and proceed by induction on step $i$.  The
base case is $i=0$.  At the start of the
zeroth step, using either method, we have a tree in this original graph.

At the end of the $i^{\rm th}$ step, which is 
the start of the $(i+1)^{\rm th}$ step, we assume that both methods have
led to the same tree in which all vertices $j\geq n-i+1$ have weighted
edges with heads at 1.

The matrices are
\[
N_{i}'=\begin{bmatrix}
B-b_1 & -b_2 & \dots & -b_{n-i} &-b_{n-i+1}& \dots & -b_n\\
-b_1&B-b_2&\dots& -b_{n-i} &-b_{n-i+1}& \dots & -b_n\\
\vdots&\vdots& \ddots & \vdots & \vdots&\dots & \vdots\\
-b_1 & -b_2 & \dots & B-b_{n-i} & -b_{n-i+1}&\dots& -b_n\\
-B & 0 & \dots & 0 &B&\dots & 0\\
\vdots &\vdots &  &\vdots& \vdots&\ddots&\vdots\\
-B & 0 &\dots & 0&0& \dots & B
\end{bmatrix},
\]
\[
N_{i+1}=\left[\begin{smallmatrix}
B-b_1 & -b_2 & \dots & -b_{n-i} &-b_{n-i+1}& \dots & -b_n\\
-b_1&B-b_2&\dots& -b_{n-i} &-b_{n-i+1}& \dots & -b_n\\
\vdots&\vdots& \ddots & \vdots & \vdots&\dots & \vdots\\
-b_1-B+b_1 & -b_2+b_2 & \dots & B-b_{n-i}+b_{n-i} & -b_{n-i+1}+b_{n-i+1}
 &\dots& -b_n+b_n\\
-B & 0 & \dots & 0 &B&\dots & 0\\
\vdots &\vdots &  &\vdots& \vdots&\ddots&\vdots\\
-B & 0 &\dots & 0&0& \dots & B
\end{smallmatrix}\right],
\]
and
\[
N_{i+1}'=\begin{bmatrix}
B-b_1 & -b_2 & \dots & -b_{n-i} &-b_{n-i+1}& \dots & -b_n\\
-b_1&B-b_2&\dots& -b_{n-i} &-b_{n-i+1}& \dots & -b_n\\
\vdots&\vdots& \ddots & \vdots & \vdots&\dots & \vdots\\
-B & 0 & \dots & B & 0 &\dots& 0\\
-B & 0 & \dots & 0 &B&\dots & 0\\
\vdots &\vdots &  &\vdots& \vdots&\ddots&\vdots\\
-B & 0 &\dots & 0&0& \dots & B
\end{bmatrix}.
\]
Let $w_j$ represent the weight of the edge $j\to 1$ 
for these vertices (remember that for each $j$, $w_j=b_r$ for some $r$), 
and let $m_k=\verb+succ+(k)$ for $1\leq k\leq n-i$.
In the matrix method, the tree is an element of
$F_{i}$.  When we apply $\overline{-I_{F_{i}'}}\circ\mu_{i+1}'$, we get
\[
\begin{bmatrix}
B_{m_1}&&&&&\\
&\ddots&&&&\\
&&B_{m_{n-i}}&&&\\
&&&w_{n-i+1}&&\\
&&&&\ddots&\\
&&&&&w_n
\end{bmatrix}\in F_i'.
\]
Now we proceed as usual for the matrix method:
\begin{multline*}
\overline{-I_{D_{i+1}}}\circ\xi_{i+1}\left(\left[\begin{smallmatrix}
B_{m_1}&&&&&\\
&\ddots&&&&\\
&&B_{m_{n-i}}&&&\\
&&&w_{n-i+1}&&\\
&&&&\ddots&\\
&&&&&w_n
\end{smallmatrix}\right]\right)=\\
\left[\begin{smallmatrix}
B_{m_1}&&&&&\\
&\ddots&&&&\\
&&B_{m_{n-i}}&&&\\
&&&w_{n-i+1}&&\\
&&&&\ddots&\\
&&&&&w_n
\end{smallmatrix}\right]\in D_{i+1}, 
\end{multline*}
and
\begin{multline*}
\overline{-I_{D_{i+1}'}}\circ\xi_{i+1}'\left(\left[\begin{smallmatrix}
B_{m_1}&&&&&\\
&\ddots&&&&\\
&&B_{m_{n-i}}&&&\\
&&&w_{n-i+1}&&\\
&&&&\ddots&\\
&&&&&w_n
\end{smallmatrix}\right]\right)=\\
\left[\begin{smallmatrix}
B_{m_1}&&&&&\\
&\ddots&&&&\\
&&B_{m_{n-i}}&&&\\
&&&w_{n-i+1}&&\\
&&&&\ddots&\\
&&&&&w_n
\end{smallmatrix}\right]\in D_{i+1}'.
\end{multline*}
Now we note that the next step depends on the status of our tree in the
new graph.  

\underline{Case 1}
Suppose that
\[
\overline{-I_{F_{i+1}}}\circ\mu_{i+1}\left(\left[\begin{smallmatrix}
B_{m_1}&&&&&\\
&\ddots&&&&\\
&&B_{m_{n-i}}&&&\\
&&&w_{n-i+1}&&\\
&&&&\ddots&\\
&&&&&w_n
\end{smallmatrix}\right]\right)
\]
is a tree in $F_{i+1}$.  We have removed the
edge $(n-i)\rightarrow m_{n-i}$, and added an edge $(n-i)\to 1$ with
weight $B_{m_{n-i}}$. We set $w_{n-i}=B_{m_{n-i}}$, and are finished
with this step.  Clearly we have the same tree we would have if we had
used the tree surgery method.

\underline{Case 2}
Suppose that 
\[
\mu_{i+1}\left(\left[\begin{smallmatrix}
B_{m_1}&&&&&\\
&\ddots&&&&\\
&&B_{m_{n-i}}&&&\\
&&&w_{n-i+1}&&\\
&&&&\ddots&\\
&&&&&w_n
\end{smallmatrix}\right]\right)
\]
is another array in $D_{i+1}'$.  The only
way for this to happen is if this array does not correspond
to a tree in the graph where all of $(n-i)$'s edges point to 1.
Since we started at a tree where all the vertices greater than $n-i$
point at 1, the only possibility is that
there is a cycle including both $(n-i)$ and 1.  None of the vertices 
$j\geq n-i+1$ can appear in this cycle since it includes only the 
vertices on the path from 1 to $(n-i)$ and these vertices all point to
1; hence, the bottom portion of the matrix
is not affected.  Thus,
\begin{multline*}
\overline{-I_{D_{i+1}'}}\circ\mu_{i+1}\left(\left[\begin{smallmatrix}
B_{m_1}&&&&&\\
&\ddots&&&&\\
&&B_{m_{n-i}}&&&\\
&&&w_{n-i+1}&&\\
&&&&\ddots&\\
&&&&&w_n
\end{smallmatrix}\right]\right)=\\
-\left[\begin{smallmatrix}
&-b_{m_1}&&&&\\
&\ddots&&&&\\
-B_{m_{n-i}}&&&&&\\
&&&w_{n-i+1}&&\\
&&&&\ddots&\\
&&&&&w_n
\end{smallmatrix}\right]\in -D_{i+1}',
\end{multline*}
with as many off-diagonal entries above row $n-i+1$
as there are vertices in the cycle
being moved off the diagonal.  These entries appear in all $(i,j)$ positions
satisfying the condition that $i\to j$ is an edge in the cycle.
\begin{multline*}
\overline{-I_{D_{i+1}}}\circ\xi_{i+1}'\left(-\left[\begin{smallmatrix}
&-b_{m_1}&&&&\\
&\ddots&&&&\\
-B_{m_{n-i}}&&&&&\\
&&&w_{n-i+1}&&\\
&&&&\ddots&\\
&&&&&w_n
\end{smallmatrix}\right]\right)=\\
-\left[\begin{smallmatrix}
&-b_{m_1}&&&&\\
&\ddots&&&&\\
-B_{m_{n-i}}&&&&&\\
&&&w_{n-i+1}&&\\
&&&&\ddots&\\
&&&&&w_n
\end{smallmatrix}\right]\in -D_{i+1},
\end{multline*}
since all entries in $N_{i+1}'$ appear also in $N_{i+1}$.  However, next 
we switch entries in rows $n-i$ and 1:
\begin{multline*}
\overline{-I_{D_{i+1}}}\circ\xi_{i+1}\left(-\left[\begin{smallmatrix}
&-b_{m_1}&&&&\\
&\ddots&&&&\\
-B_{m_{n-i}}&&&&&\\
&&&w_{n-i+1}&&\\
&&&&\ddots&\\
&&&&&w_n
\end{smallmatrix}\right]\right)=\\
\left[\begin{smallmatrix}
B_{m_{n-i}}&&&&&&\\
&\ddots&&\hdots&&&\\
&&+b_{m_1}&&&&\\
&&&&w_{n-i+1}&&\\
&&&&&\ddots&\\
&&&&&&w_n
\end{smallmatrix}\right]\in D_{i+1}.
\end{multline*}
Note that this will not take everything back to the diagonal ($b_{m_1}$
is not on the diagonal in these next few arrays). 
\begin{multline*}
\overline{-I_{D_{i+1}}}\circ\xi_{i+1}'\left(\left[\begin{smallmatrix}
B_{m_{n-i}}&&&&&&\\
&\ddots&&\hdots&&&\\
&&+b_{m_1}&&&&\\
&&&&w_{n-i+1}&&\\
&&&&&\ddots&\\
&&&&&&w_n
\end{smallmatrix}\right]\right)=\\
-\left[\begin{smallmatrix}
B_{m_{n-i}}&&&&&&\\
&\ddots&&\hdots&&&\\
&&-b_{m_1}&&&&\\
&&&&w_{n-i+1}&&\\
&&&&&\ddots&\\
&&&&&&w_n
\end{smallmatrix}\right]\in -D_{i+1}
\end{multline*}
\begin{multline*}
\overline{-I_{F_i'}}\circ\xi_{i+1}\left(-\left[\begin{smallmatrix}
B_{m_{n-i}}&&&&&&\\
&\ddots&&\hdots&&&\\
&&-b_{m_1}&&&&\\
&&&&w_{n-i+1}&&\\
&&&&&\ddots&\\
&&&&&&w_n
\end{smallmatrix}\right]\right)=\\
-\left[\begin{smallmatrix}
B_{m_{n-i}}&&&&&&\\
&\ddots&&\hdots&&&\\
&&-b_{m_1}&&&&\\
&&&&w_{n-i+1}&&\\
&&&&&\ddots&\\
&&&&&&w_n
\end{smallmatrix}\right]\in -F_i'.
\end{multline*}
Things finally get straightened out in the next step; all the off-diagonal 
entries are returned to the diagonal because there is still
only one cycle:
\begin{multline*}
\overline{-I_{F_i'}}\circ\mu_{i+1}'\left(-\left[\begin{smallmatrix}
B_{m_{n-i}}&&&&&&\\
&\ddots&&\hdots&&&\\
&&-b_{m_1}&&&&\\
&&&&w_{n-i+1}&&\\
&&&&&\ddots&\\
&&&&&&w_n
\end{smallmatrix}\right]\right)=\\
\left[\begin{smallmatrix}
B_{m_{n-i}}&&&&&\\
&\ddots&&&&\\
&&B_{m_1}&&&&\\
&&&w_{n-i+1}&&\\
&&&&\ddots&\\
&&&&&w_n
\end{smallmatrix}\right]\in F_i',
\end{multline*}
where now all entries are on the diagonal.  This holds because the
result of switching the entries in rows $n-i$ and 1 is to get rid
of the edges from those two vertices and replace them by the edges
$1\to m_{n-i}$ and $(n-i)\to m_1$.  Since there was a cycle containing
these vertices before (the cycle was $1\to m_1\to\dots\to (n-i)\to 1$,
where the last edge had weight $B_{m_{n-i}}$), what we have done is to
remove 1 from the cycle and pull the cycle out of the tree; every vertex
that was in the cycle has been removed from the path joining 1 to 0, and
1 is in the component of the graph that is still a tree.  This lone cycle
has to be returned to the diagonal.  Now we can follow the involutions
joyfully back down the sequence of matrices:
\begin{multline*}
\overline{-I_{D_{i+1}'}}\circ\xi_{i+1}'\circ\overline{-I_{D_{i+1}}}\circ\xi_{i+1}\left(
\left[\begin{smallmatrix}
B_{m_{n-i}}&&&&&\\
&\ddots&&&&\\
&&B_{m_1}&&&&\\
&&&w_{n-i+1}&&\\
&&&&\ddots&\\
&&&&&w_n
\end{smallmatrix}\right]\right)=\\
\left[\begin{smallmatrix}
B_{m_{n-i}}&&&&&\\
&\ddots&&&&\\
&&B_{m_1}&&&&\\
&&&w_{n-i+1}&&\\
&&&&\ddots&\\
&&&&&w_n
\end{smallmatrix}\right]\in D_{i+1}'.
\end{multline*}
(There are no interesting steps in between.)  When we apply 
$\mu_{i+1}$ to this matrix, we get a tree in the graph where $n-i$ has
edges only to 1.  This is simply because $n-i$ is no longer on
the path from 1 to 0.  We note that the weight of the new edge 
$(n-i)\to 1$ is $B_{m_1}$, so we set $w_{n-i}=B_{m_1}$ and are finished
with this step.
This is exactly the same as the tree surgery result.

Having accounted for all the cases, we see that at the end of step $i+1$,
the tree surgery method and the matrix method give the same weights of
edges for vertices $j\geq n-i$.  By induction, we conclude that at the
end of step $n-1$, both methods give the same weights and that consequently,
the Dandelion code found will be the same using each method.
\end{proof}

It is interesting to note that although the Dandelion matrix method
seems more closely related to the Blob matrix method than the Happy one,
the tree surgery algorithm is closer to the Happy Code.

\chapter{Permutations of the na\"\i ve code}
\section{The Happy Code: an easier method}
We have a third method for the Happy Code, that depends only on taking
the na\"\i ve code (the input is in the form 
$p=(p_1,p_2,\dots p_n)=(B_{j_1},B_{j_2},\dots,B_{j_n})$)
and permuting it according to the following algorithm:

\begin{tabbing}
Fast algorithm for Happy Code\\
{\bf begin}\=\\
\>{\bf while} \= $p_1\neq B_0$ {\bf do}\\
\>\>$a\leftarrow$ subscript of $p_1$\\
\>\>$t\leftarrow p_a$\\
\>\>$p_a\leftarrow b_a$\\
\>\>$p_1\leftarrow t$\\
\>\>$k\leftarrow n$\\
\>\>{\bf while} \= $k>a$ {\bf and} $\forall j$, $p_k\neq b_j$ {\bf do}\\
\>\>\>$k\leftarrow k-1$\\
\>\>$t\leftarrow p_a$\\
\>\>$p_a\leftarrow p_k$\\
\>\>$p_k\leftarrow t$\\ 
\>$\verb+happycode+\leftarrow$ the subscripts of $(p_2,p_3,\dots,p_n)$,
in order\\
{\bf end.}
\end{tabbing}

\noindent
For example, if we start with the tree

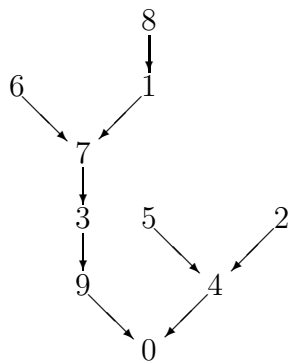
\vspace{15pt}
\begin{picture}(100,125)
\put(50,0){0}
\put(25,25){9}
\put(30,25){\vector(1,-1){17}}
\put(75,25){4}
\put(75,25){\vector(-1,-1){16}}
\put(25,50){3}
\put(28,48){\vector(0,-1){14}}
\put(50,50){5}
\put(55,50){\vector(1,-1){17}}
\put(100,50){2}
\put(100,50){\vector(-1,-1){16}}
\put(25,75){7}
\put(28,73){\vector(0,-1){14}}
\put(0,100){6}
\put(5,100){\vector(1,-1){17}}
\put(50,100){1}
\put(50,100){\vector(-1,-1){16}}
\put(50,125){8}
\put(53,123){\vector(0,-1){14}}
\end{picture}

\vspace{5pt}
\noindent then the procedure goes as follows:  
First we note that the na\"\i ve code is\\ 
$B_7 B_4 B_9 B_0 B_4 B_7 B_3 B_1 B_0$.

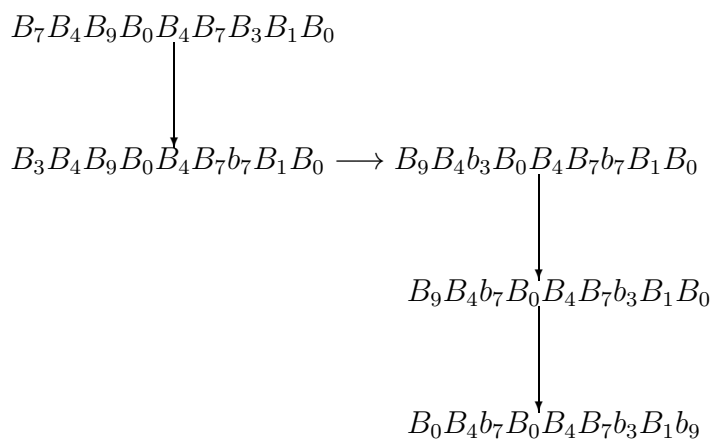
\begin{picture}(400,170)
\put(0,150){$B_7 B_4 B_9 B_0 B_4 B_7 B_3 B_1 B_0$}
\put(62,148){\vector(0,-1){40}}
\put(0,100){$B_3 B_4 B_9 B_0 B_4 B_7 b_7 B_1 B_0\longrightarrow B_9 B_4
b_3 B_0 B_4 B_7 b_7 B_1 B_0$}
\put(200,98){\vector(0,-1){40}}
\put(150,50){$B_9 B_4 b_7 B_0 B_4 B_7 b_3 B_1 B_0$}
\put(200,48){\vector(0,-1){40}}
\put(150,0){$B_0 B_4 b_7 B_0 B_4 B_7 b_3 B_1 b_9$}
\end{picture}

\vspace{5pt}
\noindent We find the code by looking at the subscripts after the initial
$B_0$: the Happy Code for the tree shown above is (4,7,0,4,7,3,1,9).

\begin{theorem}
The algorithm above gives the Happy Code as defined in previous sections.
\end{theorem}
\begin{proof}
At any stage in this algorithm, a lower-case entry indicates membership
in a cycle.  We think of the $i^{\rm th}$ entry as having $\verb+succ+(i)$
as its subscript.
All this method does at each step is to change $\verb+succ+(1)$ (the
first entry) to $\verb+succ+(\verb+succ+(1))$ and insert $\verb+succ+(1)$
after the largest vertex in a cycle, since the largest vertex will be the
furthest lower-case entry to the right.  This is exactly what tree surgery
accomplishes, so this is essentially a shorthand notation for tree surgery.
Note that this also proves that the algorithm terminates.
\end{proof}

\section{The Dandelion Code: an easier method}\label{joy}
An even faster method exists for the Dandelion Code.  The algorithm
has as its input a tree as a set of edges. It uses the previously mentioned
function $\verb+path+(x)$ which finds the path from $x$ to 0, returning a
list of vertices $(x,\verb+succ+(x),\dots,0)$.

\begin{tabbing}
Fast algorithm for Dandelion Code\\
{\bf begin}\=\\
\>$p\leftarrow\verb+path+(1)$\\
\>$m\leftarrow$ length of $p$\\
\>$p\leftarrow(p_2,\dots,p_{m-1})$\\
\>$m\leftarrow m-2$\\
\>{\bf repeat}\=\\
\>\>$a\leftarrow$ the position of the maximum element of $p$\\
\>\>$(p_1,p_2,\dots,p_a)$ becomes a cycle\\ 
\>\>$p\leftarrow(p_{a+1},\dots,p_m)$\\ 
\>{\bf until} $p=\emptyset$\\
\>rewrite the resulting collection 
of cycles as a permutation in 2-line notation\\
\>this permutation gives the new 
   \verb+succ+ function on the vertices on the path\\
\>$\verb+code+\leftarrow(\verb+succ+(2),\verb+succ+(3),\dots,\verb+succ+(n))$\\
{\bf end.}
\end{tabbing}

\begin{ex}
We begin with the following tree:

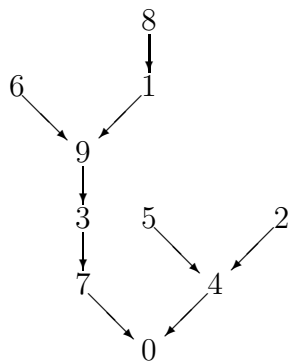
\begin{picture}(100,150)
\put(50,0){0}
\put(25,25){7}
\put(30,25){\vector(1,-1){17}}
\put(75,25){4}
\put(75,25){\vector(-1,-1){16}}
\put(25,50){3}
\put(28,48){\vector(0,-1){14}}
\put(50,50){5}
\put(55,50){\vector(1,-1){17}}
\put(100,50){2}
\put(100,50){\vector(-1,-1){16}}
\put(25,75){9}
\put(28,73){\vector(0,-1){14}}
\put(0,100){6}
\put(5,100){\vector(1,-1){17}}
\put(50,100){1}
\put(50,100){\vector(-1,-1){16}}
\put(50,125){8}
\put(53,123){\vector(0,-1){14}}
\end{picture}

\vspace{5pt}
\noindent and the procedure goes as follows:

First we note that the path from 1 to 0 is (9,3,7).  We want to 
write this as cycles according to the algorithm.  9 is the largest
thing on the path, so we end a cycle after it.  Then 3 is not the
largest remaining label on the path, so we don't end a cycle after
it, but 7 is, so we do.  Then we have to include the successors of
the other vertices:
\[
(9,3,7)\longrightarrow(9)(37)\longrightarrow\binom{379}{739}\longrightarrow
\binom{23456789}{47049319}
\]
The code is given by the bottom line: (4,7,0,4,9,3,1,9).

Another example:  If we start with the tree
\[
1\to 6\to 4\to 9\to 8\to 3\to 2\to 5\to 7\to 0,
\] 
then the procedure is as follows:
\[
(6,4,9,8,3,2,5,7)\longrightarrow(649)(8)(3257)\longrightarrow
\binom{23456789}{52974386}\longrightarrow\binom{23456789}{52974386}
\]
(Here, the path consisted of all the other vertices in the graph, so the
last 2 steps look identical.)  So the Dandelion Code 
for this tree is (5,2,9,7,4,3,8,6).
\end{ex}

At first glance it may not be clear that this algorithm is even a bijection.
However, it is.  We will need the following:
\begin{defn}
If $S$ is a set of disjoint
cycles, let $\preceq$ be the partial ordering on $S$ defined
by $C_1\preceq C_2$ if and only if the largest vertex in $C_1$ is less
than the largest vertex in $C_2$.
\end{defn}
\begin{theorem}
The fast algorithm for the Dandelion Code has as its inverse the following
algorithm:
\end{theorem}
\begin{tabbing}
Fast Algorithm to go from Dandelion Code to Tree\\
{\bf begin}\=\\
\>$\verb+edges+\leftarrow\{1\to 0\}$\\
\>{\bf for} \= $i=2$ {\bf to} $n$ {\bf do}\\
\>\>add edge $i\to c_{i-1}$\\
\>write cycles as permutations in cycle notation\\
\>write them in descending order according to $\preceq$\\
\>within each cycle, cyclically reorder so that the largest element appears
last\\
\>$s\leftarrow$ the permutation with the parentheses ignored, as a list\\
\>prepend 1 to $s$\\
\>append 0 to $s$\\
\>{\bf for} \= $j=1$ {\bf to} $|s|-1$ {\bf do}\\
\>\>remove edge $s_j\to\verb+succ+(s_j)$\\
\>\>add edge $s_j\to s_{j+1}$\\
{\bf end.}
\end{tabbing}

This is clearly the inverse of the Fast Algorithm for the Dandelion Code.
These algorithms, in slightly different form, were previously discovered
by E{\u{g}}ecio{\u{g}}lu and Remmel
\cite{E-R}, apparently using some version of the Involution Principle
\cite{R}.  Their bijection $\theta_{n+1}$ is isomorphic to our bijection
as follows.  Starting with a tree whose vertices are labelled 
$\{1,2,\dots,n+1\}$, we subtract from $n+1$ the
labels of all vertices besides 1 on the path from 1 to $n+1$.  Then we 
apply the fast
algorithm for the Dandelion Code, and then subtract from $n+1$
the labels of all vertices in cycles.  The result is the same 
functional digraph that E{\u{g}}ecio{\u{g}}lu and Remmel 
produced, except that we also have an edge $1\to 0$ and the
vertex $n+1$ has been relabelled with 0.

The Dandelion Code 
is reminiscent of Joyal's proof of the formula for the number of labelled
trees \cite{J}.  His argument rested on the fact that the 
number of linear orderings of a set is the same as the number of collections
of cycles from that set, and on the notion that an undirected tree, two of 
whose vertices 
are ``special,'' should correspond to a functional digraph found by taking
the linear ordering of the vertices between the two ``special vertices'' and
using the corresponding collection of disjoint cycles.  

The obvious bijection between linear orderings (of the vertices on the path
from one special vertex to the other) and collections of cycles is to consider
the linear ordering to be the second line of the 2-line notation for 
permutations, and the collection of cycles to be the permutation.  Although
this is probably what Joyal had in mind, it is somewhat unnatural in that
it usually preserves very few of the original edges in the tree.

The relationship between Joyal's proof and the Dandelion Code is that for our
purposes, the ``special'' vertices are always 1 and 0, and we are specific
about the bijection between linear orderings and collections of disjoint
cycles.  The bijection we choose (namely, the one where 1 and 0 are ignored
and the path between them is broken into cycles according to the algorithm 
above) is more natural than the obvious one 
because it preserves nearly all of the original edges of the tree.FIND

Essentially, the Dandelion Code is an implementation of Joyal's argument,
where we consider only functional digraphs where there is a loop at 0 and a
loop at 1, 1 is considered to be the largest vertex, and we use the algorithm
of the fast Dandelion Code and its inverse as the bijection between linear
orderings and collections of cycles.  In Chapter \ref{yellow} we discuss
the relationship between the Dandelion Code and the Happy Code, which means
that the Happy Code is a different implementation of Joyal's argument.

\begin{theorem}
The fast algorithm above 
gives the Dandelion Code as defined in Chapter \ref{dandy}.
\end{theorem}
\begin{proof}
We note that the tree surgery algorithm has $n-1$ steps, whereas the 
fast algorithm has an unclear number that is usually less than $n-1$.
However, in step $i$ of the tree surgery method, if vertex $n-i+1$
is not on the path from 1 to 0, then performing the tree surgery of
pointing its edge at 1 does not create a cycle.  Thus, 
$w_{n-i+1}=b_{\verb+succ+(n-i+1)}$ will be the $(n-i)^{\rm th}$ entry
in the code.  This matches the effect of the fast algorithm, which 
essentially starts with the na\"\i ve code and then changes the entries
only of vertices on the path from 1 to 0.  

For vertices that lie between 1 and 0 and do not have any inversions, 
the tree surgery algorithm notes
the cycle that has appeared and does the equivalent of reverting to the
original tree from the start of step $i$ and switching the successors
of $n-i+1$ and 1.  Then, to get rid of the smaller cycle that this graph
has, it removes the new edge $(n-i+1)\to \verb+succ+(1)$ and adds an edge
$(n-i+1)\to 1$ with weight \verb+succ+(1).  All of this amounts to exactly
what the fast algorithm does.  Because the tree surgery algorithm changes
the edges of the vertices from $n$ down to 2, the largest vertex on the
current path from 1 to 0 is the one whose edge will point initially back
at 1 and create a cycle, though in fact what will happen is that whatever
was the current \verb+succ+(1) is what gets put as the weight of the edge.
This corresponds to a cycle in the sense of permutations, and since this
cycle is removed from the path, the rest of the vertices in that cycle
now keep their original successors for the final code.  
\end{proof}

\chapter{Relationship between Codes}\label{yellow}
There is actually a close relationship between the Happy Code and the 
Dandelion Code. 

\begin{ex}
If we start with the tree $1\to 6\to 4\to 9\to 8\to 3\to 2\to 5\to 7\to 0$, 
we found in \S\ref{joy} that its Dandelion Code was (5,2,9,7,4,3,8,6).
If we reverse the order of the vertices between 1 and 0, the new
tree is
\[
1\to 7\to 5\to 2\to 3\to 8\to 9\to 4\to 6\to 0,
\]
Note that the Dandelion Code of the tree with the reversed path from 1 to 0
is {\em not} the reverse of the Dandelion Code of the original tree:
\[
(7,5,2,3,8,9,4,6)\longrightarrow(752389)(46)\longrightarrow
\binom{23456789}{38624597}\longrightarrow\binom{23456789}{38624597}
\]
However, the Happy Code for this new tree yields the Dandelion Code for the
original tree. 
The na\"\i ve code is $B_7 B_3 B_8 B_6 B_2 B_0 B_5 B_9 B_4$.  The
fast Happy Code algorithm goes as follows:

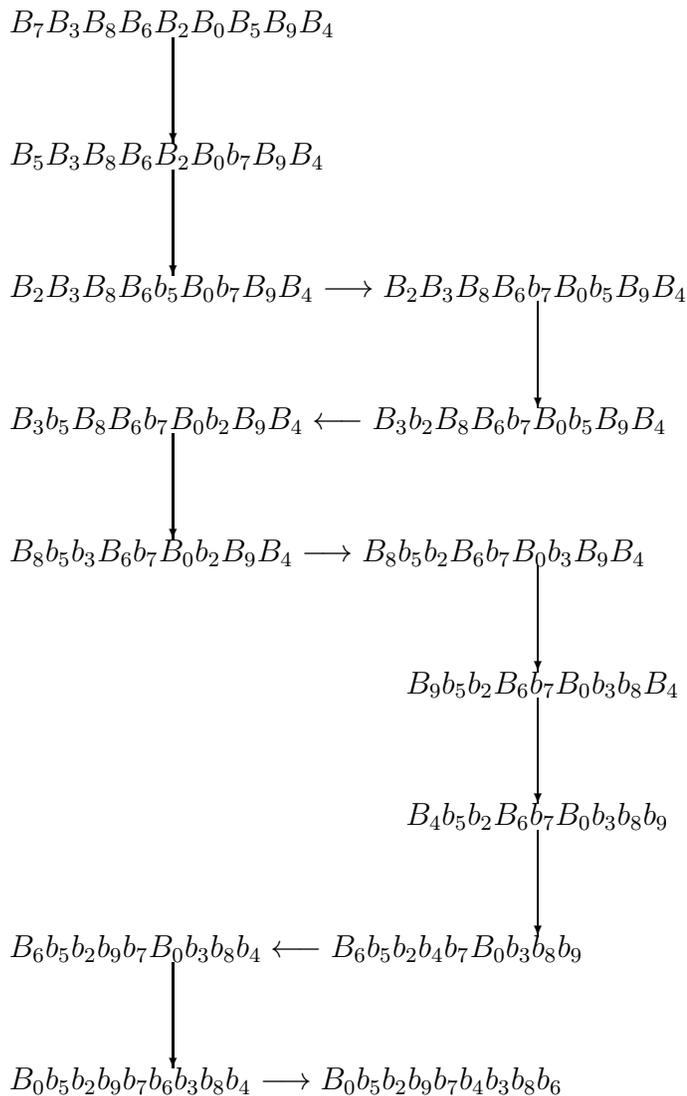
\vspace{15pt}
\begin{picture}(400,500)
\put(0,500){$B_7 B_3 B_8 B_6 B_2 B_0 B_5 B_9 B_4$}
\put(62,498){\vector(0,-1){40}}
\put(0,450){$B_5 B_3 B_8 B_6 B_2 B_0 b_7 B_9 B_4$}
\put(62,448){\vector(0,-1){40}}
\put(0,400){$B_2 B_3 B_8 B_6 b_5 B_0 b_7 B_9 B_4\longrightarrow B_2 B_3 B_8
B_6 b_7 B_0 b_5 B_9 B_4$}
\put(200,398){\vector(0,-1){40}}
\put(0,350){$B_3 b_5 B_8 B_6 b_7 B_0 b_2 B_9 B_4\longleftarrow B_3 b_2 B_8 
B_6 b_7 B_0 b_5 B_9 B_4$}
\put(62,348){\vector(0,-1){40}}
\put(0,300){$B_8 b_5 b_3 B_6 b_7 B_0 b_2 B_9 B_4\longrightarrow B_8 b_5 b_2
B_6 b_7 B_0 b_3 B_9 B_4$}
\put(200,298){\vector(0,-1){40}}
\put(150,250){$B_9 b_5 b_2 B_6 b_7 B_0 b_3 b_8 B_4$}
\put(200,248){\vector(0,-1){40}}
\put(150,200){$B_4 b_5 b_2 B_6 b_7 B_0 b_3 b_8 b_9$}
\put(200,198){\vector(0,-1){40}}
\put(0,150){$B_6 b_5 b_2 b_9 b_7 B_0 b_3 b_8 b_4\longleftarrow B_6 b_5 b_2 
b_4 b_7 B_0 b_3 b_8 b_9$}
\put(62,148){\vector(0,-1){40}}
\put(0,100){$B_0 b_5 b_2 b_9 b_7 b_6 b_3 b_8 b_4\longrightarrow B_0 b_5 b_2 
b_9 b_7 b_4 b_3 b_8 b_6$}
\end{picture}

\vspace{-90pt}
\noindent
The subscripts give the Happy Code: (5,2,9,7,4,3,8,6), which is the same
as the Dandelion Code of the tree where the path from 1 to 0 was in the
other order.
\end{ex}

\begin{theorem} If the order of the vertices on the
path from 1 to 0 is reversed, the Happy Code of the new
tree will be the same as the Dandelion Code of the original tree (and vice
versa).
\end{theorem}
\begin{proof}
To understand how the Happy Code and the Dandelion Code are so
closely related, we note that both depend on the path from 1 to 0.  Vertices
occurring elsewhere in the tree have the same effect on the code using either
algorithm; if $i$ is such a vertex then $\verb+succ+(i)$ will appear in the
$(i-1)^{\rm th}$ position of both the Dandelion Code and the Happy Code for
both the tree and its path-reversed modification.

Recalling the tree surgery method for the Happy Code, we construct a similar
method to the faster algorithm of the Dandelion Code.  First, we write out 
only the path from 1 to 0.  We know that as we move from left to right (from
1 to 0) along it, each vertex gets placed in a cycle, immediately following
the largest vertex already in a cycle.  Thus we are comparing each vertex 
with the vertices to its left in the path.  We end a cycle just {\em before}
a new largest vertex (among the labels to its left).  
This is the reverse of the fast Dandelion Code, which
ends a cycle just {\em after} the largest vertex among the labels to its right.

Meanwhile, within the cycles, each new label is inserted after the largest 
label in the cycle for the Happy Code.  
But the first element in the cycle is the largest (by 
virtue of how we have split path into cycles), and the vertices are added to
it one at a time from left to right--always inserted after this largest label.
The effect is that of reversing the path order of the remaining vertices in the
cycle.  The resulting cycle is exactly the cycle that arises from the Dandelion
Code of the path-reversed modification of the original tree.
\end{proof}

The Happy Code is another implementation of Joyal's almost-bijection.  This
time the choice of bijection between linear orderings and sets of cycles is
not as natural because it changes more of the edges of the tree.

\chapter{Conclusion}
\section{The codes are distinct}
An example suffices to prove that these codes are different from one
another and from the Pr\"ufer Code. 

\begin{ex}
Consider the tree

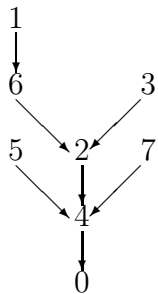
\vspace{15pt}
\begin{picture}(100,100)
\put(50,0){0}
\put(50,25){4}
\put(53,23){\vector(0,-1){15}}%  4->0
\put(25,50){5}
\put(28,48){\vector(1,-1){20}}%  5->4
\put(50,50){2}
\put(53,48){\vector(0,-1){15}}%  2->4
\put(75,50){7}
\put(75,48){\vector(-1,-1){19}}% 7->4
\put(25,75){6}
\put(75,75){3}
\put(28,73){\vector(1,-1){20}}%  6->2
\put(75,73){\vector(-1,-1){19}}% 3->2
\put(25,100){1}
\put(28,98){\vector(0,-1){15}}%  1->6
\end{picture}

\noindent
whose Pr\"ufer Code we calculated in \S\ref{proofer} to be (6,2,4,2,4,4).

The Dandelion Code for the tree is found as follows:
\[
(6,2,4)\longrightarrow(6)(24)\longrightarrow\binom{246}{426}\longrightarrow
\binom{234567}{422464}
\]
So the Dandelion Code for the tree is (4,2,2,4,6,4).

The Happy Code can be found by reversing the order of the path from 1 to 0 
and finding the Dandelion Code of the altered tree:

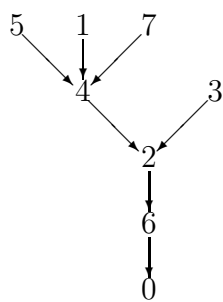
\vspace{15pt}
\begin{picture}(100,100)
\put(50,0){0}
\put(50,25){6}
\put(53,23){\vector(0,-1){15}}%  6->0
\put(50,50){2}
\put(53,48){\vector(0,-1){15}}%  2->6
\put(25,75){4}
\put(30,75){\vector(1,-1){19}}%  4->2
\put(75,75){3}
\put(75,75){\vector(-1,-1){19}}% 3->2
\put(0,100){5}
\put(5,100){\vector(1,-1){19}}%  5->4
\put(25,100){1}
\put(28,98){\vector(0,-1){15}}%  1->4
\put(50,100){7}
\put(50,100){\vector(-1,-1){19}}%7->4
\end{picture}

\noindent 
This tree's Dandelion Code is (6,2,2,4,4,4):
\[
(4,2,6)\longrightarrow(426)\longrightarrow\binom{246}{624}\longrightarrow
\binom{234567}{622444}.
\]
Thus the Happy Code of our main example tree is (6,2,2,4,4,4).

The Blob Code takes a little more work:

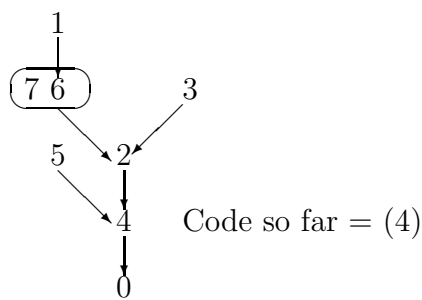
\vspace{15pt}
\begin{picture}(100,100)
\put(50,0){0}
\put(50,25){4}
\put(53,23){\vector(0,-1){15}}%  4->0
\put(25,50){5}
\put(28,48){\vector(1,-1){20}}%  5->4
\put(50,50){2}
\put(53,48){\vector(0,-1){15}}%  2->4
\put(25,75){6}
\put(15,75){7}
\put(25,79){\oval(30,15)}
\put(75,75){3}
\put(28,71.5){\vector(1,-1){20}}%  blob->2
\put(75,73){\vector(-1,-1){19}}% 3->2
\put(25,100){1}
\put(28,98){\vector(0,-1){15}}%  1->6
\put(75,25){Code so far = (4)}
\end{picture}

It'll take a few more steps.

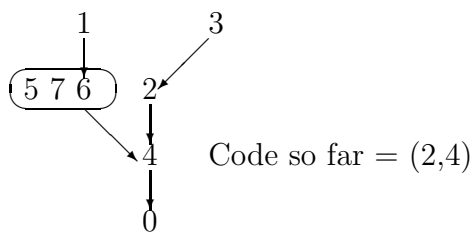
\begin{picture}(100,100)
\put(50,0){0}
\put(50,25){4}
\put(53,23){\vector(0,-1){15}}%  4->0
\put(50,50){2}
\put(53,48){\vector(0,-1){15}}%  2->4
\put(25,50){6}
\put(15,50){7}
\put(5,50){5}
\put(20,54){\oval(40,15)}
\put(75,75){3}
\put(28,46.5){\vector(1,-1){20}}%  blob->4
\put(75,73){\vector(-1,-1){19}}% 3->2
\put(25,75){1}
\put(28,73){\vector(0,-1){15}}%  1->6
\put(75,25){Code so far = (2,4)}
\end{picture}

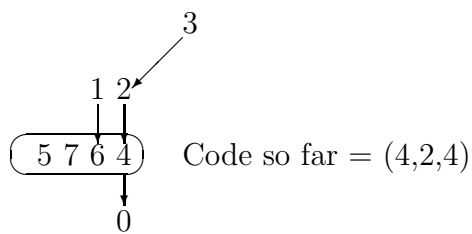
\vspace{15pt}
\begin{picture}(100,75)
\put(50,0){0}
\put(50,25){4}
\put(53,21.5){\vector(0,-1){12}}%  blob->0
\put(50,50){2}
\put(53,48){\vector(0,-1){15}}%  2->4
\put(40,25){6}
\put(30,25){7}
\put(20,25){5}
\put(35,29){\oval(50,15)}
\put(75,75){3}
\put(75,73){\vector(-1,-1){19}}% 3->2
\put(40,50){1}
\put(43,48){\vector(0,-1){15}}%  1->6
\put(75,25){Code so far = (4,2,4)}
\end{picture}

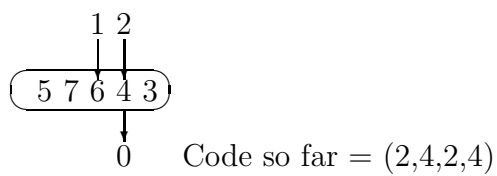
\vspace{15pt}
\begin{picture}(100,60)
\put(50,0){0}
\put(50,25){4}
\put(53,21.5){\vector(0,-1){12}}%  blob->0
\put(50,50){2}
\put(53,48){\vector(0,-1){15}}%  2->4
\put(40,25){6}
\put(30,25){7}
\put(20,25){5}
\put(40,29){\oval(60,15)}
\put(60,25){3}
\put(40,50){1}
\put(43,48){\vector(0,-1){15}}%  1->6
\put(75,0){Code so far = (2,4,2,4)}
\end{picture}

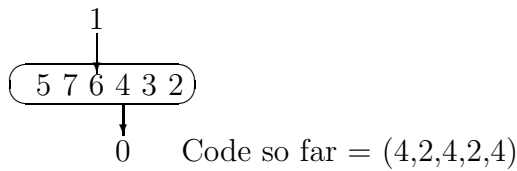
\vspace{15pt}
\begin{picture}(100,50)
\put(50,0){0}
\put(50,25){4}
\put(53,21.5){\vector(0,-1){12}}%  blob->0
\put(70,25){2}
\put(40,25){6}
\put(30,25){7}
\put(20,25){5}
\put(45,29){\oval(70,15)}
\put(60,25){3}
\put(40,50){1}
\put(43,48){\vector(0,-1){15}}%  1->6
\put(75,0){Code so far = (4,2,4,2,4)}
\end{picture}

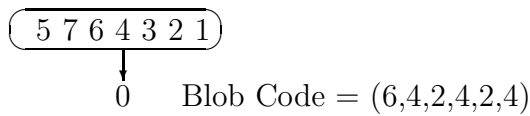
\vspace{15pt}
\begin{picture}(100,50)
\put(50,0){0}
\put(50,25){4}
\put(53,21.5){\vector(0,-1){12}}%  blob->0
\put(70,25){2}
\put(40,25){6}
\put(30,25){7}
\put(20,25){5}
\put(50,29){\oval(80,15)}
\put(60,25){3}
\put(80,25){1}
\put(75,0){Blob Code = (6,4,2,4,2,4)}
\end{picture}

\vspace{15pt}
\noindent  This is different from the other codes.

So, to review, the tree we started with has the following codes:

\begin{center}
\begin{tabular}{c|c}
Method & Code\\
\hline
Blob & (6,4,2,4,2,4)\\
Happy & (6,2,2,4,4,4)\\
Dandelion & (4,2,2,4,6,4)\\
Pr\"ufer & (6,2,4,2,4,4)\\
\end{tabular}
\end{center}
\end{ex}

\vspace{8pt}\noindent
Thus, we conclude that the various codes are all distinct.  

\section{Clever weighting of edges}
In \cite{E-R}, E{\u{g}}ecio{\u{g}}lu and Remmel 
use a six-variable weighted version of 
Cayley's formula instead of the $(n+1)$-variable version we have been
using.  They were able to produce a bijection that counts descents and
ascents.  
Specifically, where we have given the edge $i\to j$ the weight $b_j$, 
they have given it the weight $x q^i t^j$ if the edge is a descent
and $y p^i s^j$ if it is an ascent or loop.  

It is possible to extend both their results and ours 
by clever weighting of edges.
We examine the result of weighting edges as follows:
\[
W(i\to j)=\left\{\begin{array}{ll} b_j & 
       \text{if } i\to j \text{ is not an ascent}\\
a_{ij} & \text{if } i\to j \text{ is an ascent} 
\end{array}\right.
\]
Here, loops are considered not to be ascents.

Using these weights, we show an example for $n=4$:
\[
U_0'=\left[\begin{smallmatrix} 
b_0+b_1+a_{12}+a_{13}+a_{14}-b_1 & -a_{12} & -a_{13} & -a_{14}\\
-b_1 & b_0+b_1+b_2+a_{23}+a_{24}-b_2 & -a_{23} & -a_{24}\\
-b_1 & -b_2 & b_0+b_1+b_2+b_3+a_{34}-b_3 & -a_{34}\\
-b_1 & -b_2 & -b_3 & b_0+b_1+b_2+b_3+b_4-b_4
\end{smallmatrix}\right]
\]  
First we subtract row 3 from row 4, obtaining
\begin{comment}
\[
U_1=\left[\begin{smallmatrix}
b_0+b_1+a_{12}+a_{13}+a_{14}-b_1 & -a_{12} & -a_{13} & -a_{14}\\
-b_1 & b_0+b_1+b_2+a_{23}+a_{24}-b_2 & -a_{23} & -a_{24}\\
-b_1 & -b_2 & b_0+b_1+b_2+b_3+a_{34}-b_3 & -a_{34}\\
-b_1+b_1 & -b_2+b_2 & -b_3-b_0-b_1-b_2-b_3-a_{34}+b_3 & b_0+b_1+b_2+b_3+b_4-b_4
+a_{34}
\end{smallmatrix}\right].
\]
Cancelling in row 4 yields\end{comment}
\[
U_1=\left[\begin{smallmatrix}
b_0+b_1+a_{12}+a_{13}+a_{14}-b_1 & -a_{12} & -a_{13} & -a_{14}\\
-b_1 & b_0+b_1+b_2+a_{23}+a_{24}-b_2 & -a_{23} & -a_{24}\\
-b_1 & -b_2 & b_0+b_1+b_2+b_3+a_{34}-b_3 & -a_{34}\\
0 & 0 & -b_0-b_1-b_2-b_3-a_{34} & b_0+b_1+b_2+b_3+a_{34}
\end{smallmatrix}\right].
\]
Next we add column 4 to column 3.
\begin{comment}
\[
U_1''=\left[\begin{smallmatrix}
b_0+b_1+a_{12}+a_{13}+a_{14}-b_1 & -a_{12} & -a_{13}-a_{14} & -a_{14}\\
-b_1 & b_0+b_1+b_2+a_{23}+a_{24}-b_2 & -a_{23}-a_{24} & -a_{24}\\
-b_1 & -b_2 & b_0+b_1+b_2+b_3+a_{34}-b_3-a_{34} & -a_{34}\\
0 & 0 & -b_0-b_1-b_2-b_3-a_{34}+b_0+b_1+b_2+b_3+a_{34} & b_0+b_1+b_2+b_3+a_{34}
\end{smallmatrix}\right].
\]
Cancelling in row 4 yields
\end{comment}
\[
U_1'=\left[\begin{smallmatrix}
b_0+b_1+a_{12}+a_{13}+a_{14}-b_1 & -a_{12} & -a_{13}-a_{14} & -a_{14}\\
-b_1 & b_0+b_1+b_2+a_{23}+a_{24}-b_2 & -a_{23}-a_{24} & -a_{24}\\
-b_1 & -b_2 & b_0+b_1+b_2+b_3+a_{34}-b_3-a_{34} & -a_{34}\\
0 & 0 & 0 & b_0+b_1+b_2+b_3+a_{34}
\end{smallmatrix}\right].
\]

The method is parallel to that of the Blob Code, only our weights are
slightly different.  We continue until we reach the final matrix, 
an upper-triangular matrix whose
$i^{\rm th}$ diagonal entry is $\sum_{k=0}^{i-1} b_k + \sum_{k=i}^n a_{i-1,k}$,
except in row 1 where the diagonal entry is $b_0$.  This yields both algebraic
and bijective proofs of a 
generalized
version of Cayley's formula, which we refer to as the UCSD formula (since
the inspiration for it came from methods of E{\u{g}}ecio{\u{g}}lu and Remmel).

The UCSD formula for the sum of the weights of all possible trees is
\[
\displaystyle{\sum_\tau W(\tau)}=
\det(U_{n-1}')=b_0\displaystyle{\prod_{i=2}^{n}\left[\sum_{k=0}^{i-1}
b_k + \sum_{j=i}^{n} a_{i-1,j}\right]}.
\]
A code is a term from this product, kept in the order of the columns from
which it came.  Specifically, the set of codes is
$\{(x_1,x_2\dots,x_n)|x_1=b_0,\text{ and for } 2\leq i\leq n$ 
$x_i=a_{i-1,j}\text{ for some } j>i-1
\text{ or }x_i=b_j\text{ for some }j\leq i-1 \}$.

The advantage of this new weighting system is that the code reveals all 
ascents and descents to and from each vertex.  The ascending edges can
be read directly from the subscripts of the $a$ weights, while the descending
indegree of any vertex $j$ is given by the number of occurrences of $b_j$.
Let $I_A(j)$ denote the ascending indegree of $j$, 
$I_D(j)$ the descending indegree,
$O_A(j)$ the ascending outdegree and $O_D(j)$ the descending outdegree.
The total indegree of $j$ is the number of occurrences of $b_j$ plus the number
of times that $j$ occurs as the second subscript of an $a$.  The descending
outdegree of a vertex $j\neq 0$ is simply $O_D(j)=1-O_A(j)$.

\begin{ex} If a tree turns out to have code $(b_0,a_{13},b_2,b_0,a_{45})$,
then we know the following:

\begin{center}
\begin{tabular}{c|c|c|c|c}
Vertex & $I_A$ & $I_D$ & $O_A$ & $O_D$\\
\hline
0 & 0 & 2 & 0 & 0\\
1 & 0 & 0 & 1 & 0\\
2 & 0 & 1 & 0 & 1\\
3 & 1 & 0 & 0 & 1\\
4 & 0 & 0 & 1 & 0\\
5 & 1 & 0 & 0 & 1
\end{tabular}
\end{center}

\noindent
From the code we can thus also conclude not only that the edges $1\to 3$ and
$4\to 5$ are in the tree, but that so is 
the edge $2\to 0$ (because 2 has descending
outdegree of 1, and 1 has indegree of zero; 2 must point at something less
than 2 but it can't be 1).  All that remains is to figure
out the edges from 3 and 5.  One must point at 2 and the other at 0 to use
up all of our indegrees.  By defining involutions as we did for the original
Blob Code, we could find it using the matrix method with the above matrices.
We can also use the inverse tree surgery algorithm from \S\ref{backwards}.  

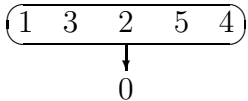
\begin{picture}(100,50)
\put(50,0){0}
\put(12,25){1}
\put(29,25){3}
\put(53,20.5){\vector(0,-1){10}}
\put(53,28){\oval(90,15)}
\put(50,25){2}
\put(71,25){5}
\put(88,25){4}
\end{picture}

\noindent
The initial $b_0$ in the code tells us that the \verb+blob+ points at 0.
The next element in the code, $a_{13}$, indicates that when we remove 1 
from the \verb+blob+, its edge points at 3 and the \verb+blob+ stays where
it is.  

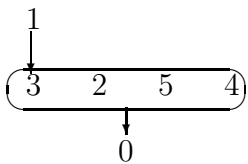
\begin{picture}(100,75)
\put(50,0){0}
\put(15,50){1}
\put(15,25){3}
\put(17,49){\vector(0,-1){16}}
\put(53,20){\vector(0,-1){10}}
\put(53,27){\oval(90,15)}
\put(40,25){2}
\put(65,25){5}
\put(90,25){4}
\end{picture}

\noindent
The next part of the code is $b_2$.  If 2 is removed from the \verb+blob+, 
then the (nonexistent) path from 2 to 0 does not 
pass through the \verb+blob+, so we remove the edge $\verb+blob+\to 0$
and add edges $\verb+blob+\to 2$ and $2\to 0$.

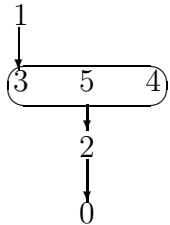
\begin{picture}(100,100)
\put(50,0){0}
\put(25,75){1}
\put(25,50){3}
\put(27,74){\vector(0,-1){16}}
\put(53,45){\vector(0,-1){10}}
\put(53,52){\oval(60,15)}
\put(50,25){2}
\put(53,24){\vector(0,-1){16}}
\put(50,50){5}
\put(75,50){4}
\end{picture} 

\noindent
The next weight in the code is $b_0$.  
We remove 3 from the \verb+blob+.  The path from 0 to 0 does not go
through the \verb+blob+, so we remove the edge $\verb+blob+\to 2$ and
add edges $3\to 2$ and $\verb+blob+\to 0$.

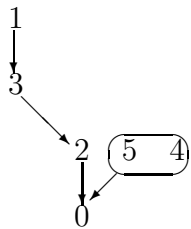
\begin{picture}(100,100)
\put(50,0){0}
\put(25,75){1}
\put(25,50){3}
\put(27,74){\vector(0,-1){16}}
\put(30,49){\vector(1,-1){18}}
\put(66,20){\vector(-1,-1){10}}
\put(78,27){\oval(30,15)}
\put(50,25){2}
\put(53,24){\vector(0,-1){16}}
\put(68,25){5}
\put(86,25){4}
\end{picture} 

\noindent
The final piece of information from the code is $a_{45}$.  This automatically
tells us what the final edge is.

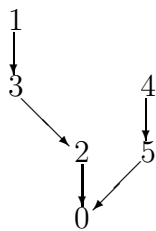
\begin{picture}(100,100)
\put(50,0){0}
\put(25,75){1}
\put(25,50){3}
\put(27,74){\vector(0,-1){16}}
\put(30,49){\vector(1,-1){18}}
\put(50,25){2}
\put(53,24){\vector(0,-1){16}}
\put(75,25){5}
\put(75,25){\vector(-1,-1){18}}
\put(75,50){4}
\put(77,49){\vector(0,-1){16}}
\end{picture} 
\end{ex}

This weighted version of the Blob Code is just as easily calculated as 
the original Blob Code, but displays more information.  

\begin{comment}
The Dandelion Code can also be generalized by weighting as follows:
\[
W(i\to j)=\left\{\begin{array}{ll} d_i b_j & 
       \text{if } i\to j \text{ is not an ascent}\\
a_i b_j & \text{if } i\to j \text{ is an ascent} 
\end{array}\right.
\]
This is a bit more complicated because the row operations of the Dandelion
Code matrix method must be altered slightly to allow for the extra $a_i$
and $d_i$ variables.  For example, if we start with the same matrix $U_0'$ 
as before, we will need to add $\frac{(-d_4)}{a_1}$ times Row 1 to Row 4,
then $\frac{(-d_3)}{a_1}$ times Row 1 to Row 3, etc.  Even then we will end
up with a few strange terms in column 1.
\end{comment}

We can also use the Dandelion Code to verify directly that the UCSD formula
holds.  The right side of the equation, a product of sums of monomials, 
expands out to a sum of terms of degree $n$.  Each term represents a 
happy functional digraph consisting of the edges $i\to j$ whenever
the $i^{\rm th}$ indeterminate in the sequence is $b_j$ or $a_{kj}$ for
some $k$.  The Dandelion Code gives a bijection between this set of happy
functional digraphs and the set of trees, which preserves weights.  Thus,
the right hand side of the equation must equal the left hand side.
This is essentially the same proof that E{\u{g}}ecio{\u{g}}lu and Remmel use
for their six-variable version of the Cayley formula.

Using the Dandelion Code matrix method for an algebraic proof of this formula
is less straightforward than using the Blob Code matrix method.

\section{Applications and more questions}
All of our simple (non-weighted) codes have interesting features.
The Happy Code is less natural than the other two and probably can not
be generalized to display more information than the Pr\"ufer Code.
It is the hardest of the three codes to get a mental handle on, 
because in the matrix method, we only apply the bijective proof of the
Matrix Tree Theorem 
to the original matrix.  This drastically complicates the proof that
the matrix method and tree surgery method for this code are equivalent.  
However, Lemma \ref{Escher},
required in that proof, is easy to state, beautiful, and surprising.

The Dandelion Code is very efficiently calculated, and allows us an easier
way to find the Happy Code.  It implements Joyal's almost-bijective proof 
of Cayley's formula in a beautiful and natural way.  If thought of
in the way suggested by the fast algorithm for it 
(as E{\u{g}}ecio{\u{g}}lu and Remmel did), it preserves most of the edges
of the original tree.  Furthermore, this bijection provides a direct proof
of the UCSD formula.

The simple Blob Code
is interesting in that it elaborates on some of Orlin's ideas and provides a
bijection behind his manipulatoric proof of the formula for the number of
trees.  
Furthermore, it doesn't single out vertex 1 as being more special 
than the others, whereas the other two codes require one (rather arbitrarily)
to examine the path from 1 to 0.  Best of all,
the matrix and tree surgery methods both generalize easily to a weighted
code that keeps track of all ascents and descents in the tree.  
\begin{comment} It is also generalizable to spanning trees
of graphs other than the complete digraph, more easily than the other two
codes.  This is simply because edges are removed from the tree but never
added; the \verb+blob+ might move and subsume a vertex but an impossible
edge can never appear.  Both of the other codes require replacing edges in the
tree by other edges, and in an arbitrary graph that may not be possible.
\end{comment}

All of these codes share the property that they are consistent with the Matrix
Tree Theorem.  They are natural in that we can undo them one step at a
time, in reverse order from the way they were found, simply by following the
involutions through in the other order.  They also can be found by simpler,
tree-surgical bijections similar to that of Pr\"ufer, yet the inverses of 
these methods are the simple inverse operations of the formations of the
codes. Meanwhile, there does not seem to be any way to ``matrixify'' the 
Pr\"ufer Code, and its inverse is decidedly unnatural.  In addition, our three
codes lose none of the information encapsulated in the Pr\"ufer Code (the
indegrees of each vertex; which vertices are leaves).

Furthermore, the Dandelion Code generalizes to forests (collections of
rooted trees) very nicely \cite{R},\cite{G}.  
Since the Matrix Tree Theorem also generalizes
to forests of $k$ rooted trees (where $k<n+1$)
using minors obtained by crossing out $k$ rows and columns,
it is possible that the bijective proof by Chaiken \cite{C} can lead
to extensions of the Blob and Happy Codes to forests as well.  However, 
this is not necessary.  We can easily extend any of the three
codes to forests of $k$ trees with roots $-1,-2,\dots,-k$ and non-root
vertices $1,2,\dots,n-k+1$
by replacing $b_0$ with $b_{-1}+b_{-2}+\dots+b_{-k}$ (and
$B_0$ with $B_{-1}+\dots+B_{-k}$) whenever they appear in the matrices. 
The result is immediate using the exact same methods, and the tree surgery
methods are not affected substantially by the change.  

The codes themselves may not be useful for much yet.  Although each of them
has some relationship with the idea of inversions, none actually count
inversions.  (An inversion occurs whenever a vertex
$j>i$ appears on the path from $i$ to 0.)
A future direction for research might be to attempt to find a
code that is both consistent with the Matrix Tree Theorem {\em and} able to
enumerate the inversions of the tree, because
the total number of inversions in a tree, $inv(\tau)$, is
of interest to algebraic combinatorists. The Hilbert series of the
space of diagonal harmonics (when restricted to t=1), 
$Hilb_n(t,q)|_{t=1}$, is conjectured to be 
$\sum q^{inv(\tau)}$ where $\tau$ ranges over all trees with vertices 
$0,\dots,n$.  Thus a statistic on one of the codes that
has the same distribution as $inv(\tau)$ might assist algebraic 
combinatorists in finding a basis for the space of diagonal harmonics.  
Unfortunately, such a code is elusive.

Another possible direction for future research is to examine the method
of the Happy Code when applied to the Blob Code's row and column operations.
Namely, if we use a placeholder $\lambda$ in the (0,0) position and only 
apply the Matrix Tree Theorem to the original matrix, setting $B_j=b_j$ 
at the end, do we get a different code?  If so, does it have any advantages
over the codes we have already found?

As noted in \S\ref{setup}, 
there are many sequences of row and column operations
that can lead to an easily calculated determinant.  
Since the matrix involution
method is quite general, any of these should give a coding algorithm for trees.
We know that not all coding algorithms correspond to matrix methods.
Naturally we are led to wonder whether
there are always simple tree surgical methods that correspond to the codes
we find through matrices. 
The true beauty of these results is that each code was defined through
row operations on the matrix before the 
corresponding tree surgical methods were discovered. 
Thus, linear algebra gave birth to bijections who grew up and became 
independent proofs in their own right.

\bibliographystyle{abbrv}
\bibliography{refs1}

\begin{thebibliography}{10}

\bibitem{B}
C.~W. Borchardt.
\newblock Uber eine der {I}nterpolation entsprechende {D}arstellung der
  {E}liminations-{R}esultante.
\newblock {\em Journal f{\"u}r die Reine und Angewandte Mathematik},
  57:111--121, 1860.

\bibitem{Ca}
A.~Cayley.
\newblock A theorem on trees.
\newblock {\em Quarterly Journal of Pure and Applied Mathematics}, 23:376--378,
  1889.

\bibitem{C}
S.~Chaiken.
\newblock A combinatorial proof of the all minors matrix tree theorem.
\newblock {\em SIAM Journal of Algebraic Discrete Methods}, 3(3):319--329,
  1982.

\bibitem{E-R}
{\"{O}}.~E{\u{g}}ecio{\u{g}}lu and J.~Remmel.
\newblock Bijections for {C}ayley trees, spanning trees, and their
  $q$-analogues.
\newblock {\em Journal of Combinatorial Theory}, 42(1):15--30, 1986.

\bibitem{G-M}
A.~Garsia and S.~Milne.
\newblock Method for constructing bijections for classical partition
  identities.
\newblock {\em Proceedings of the National Academy of Sciences, USA},
  78:2026--2028, 1981.

\bibitem{J}
A.~Joyal.
\newblock Une th{\'e}orie combinatoire des s{\'e}ries formelles.
\newblock {\em Advances in Mathematics}, 42:1--82, 1981.

\bibitem{K}
D.~E. Knuth.
\newblock Oriented subtrees of an arc digraph.
\newblock {\em Journal of Combinatorial Theory}, 3:309--314, 1967.

\bibitem{O}
J.~B. Orlin.
\newblock Line-digraphs, arborescences, and theorems of {T}utte and {K}nuth.
\newblock {\em Journal of Combinatorial Theory}, 25:187--198, 1978.

\bibitem{P}
H.~Pr{\"u}fer.
\newblock Neuer {B}eweis eines {S}atzes {\"u}ber {P}ermutationen.
\newblock {\em Arch. Math. Phys.}, 27:142--144, 1918.

\bibitem{R}
J.~Remmel.
\newblock personal communication.

\bibitem{S-W}
D.~Stanton and D.~White.
\newblock {\em Constructive Combinatorics}.
\newblock Springer-Verlag, 1986.

\bibitem{G}
G.~Tesler.
\newblock personal communication.

\bibitem{T}
W.~Tutte.
\newblock The dissection of equilateral triangles into equilateral triangles.
\newblock {\em Proceedings of the Cambridge Philosophical Society},
  44:463--482, 1948.

\bibitem{Z}
D.~Zeilberger.
\newblock A combinatorial approach to matrix algebra.
\newblock {\em Discrete Mathematics}, 56:61--72, 1985.

\end{thebibliography}

\end{document}